\documentclass[11pt]{article}
\usepackage[utf8]{inputenc}
\usepackage{lmodern}
\usepackage{subfiles}
\usepackage{enumitem}
	\setenumerate{label={\normalfont(\roman*)}, itemsep=0em} 
        
\usepackage{amsfonts}
\usepackage{amsthm}
\usepackage{amsmath}
\usepackage{amssymb}
\usepackage{amscd}
\usepackage{mathrsfs}
\usepackage{mathtools}
\usepackage{bm}
\usepackage{esint}

\usepackage[margin=3cm]{geometry}
\usepackage{indentfirst}
\usepackage{graphicx}
\usepackage{graphics}
\usepackage{lscape}
\usepackage{tikz-cd}
\usepackage{color}
\usepackage{pict2e}
\usepackage{epic}
\usepackage{epstopdf}
\usepackage{titlesec, titlefoot}
	\titleformat{\section}[block]{\Large\bfseries\filcenter}{\thesection}{1em}{}
\usepackage{float}
\usepackage{caption}
\usepackage{etoolbox}
\usepackage[affil-it]{authblk}
\usepackage{combelow}
\usepackage[hidelinks]{hyperref}
\hypersetup{bookmarksopen=true} 
\usepackage{hypcap}

\graphicspath{{./Pictures/}}
\allowdisplaybreaks

\usepackage{listings}
\usepackage{fancybox}

\usepackage[capitalize,noabbrev]{cleveref}

\theoremstyle{plain}

\renewcommand*\thesection{\arabic{section}}
\numberwithin{equation}{section}

\theoremstyle{plain}
\newtheorem{thm}{Theorem}[section]
\newtheorem{theorem}[thm]{Theorem}

\newtheorem{lemma}[thm]{Lemma}
\newtheorem{proposition}[thm]{Proposition}
\newtheorem{prop}[thm]{Proposition}
\newtheorem{corollary}[thm]{Corollary}

\theoremstyle{definition}

\newtheorem{remark}[thm]{Remark}

\newtheorem{assumption}[thm]{Assumption}

\expandafter\let\expandafter\oldproof\csname\string\proof\endcsname
\let\oldendproof\endproof
\renewenvironment{proof}[1][\proofname]{%
  \oldproof[\upshape \bfseries #1]%
}{\oldendproof}

\makeatletter
\def\@makechapterhead#1{%
  \vspace*{50\p@}%
  {\parindent \z@ \raggedright \normalfont
    \interlinepenalty\@M
    \Huge\bfseries  \thechapter.\quad #1\par\nobreak
    \vskip 40\p@
  }}
\makeatother

\newcommand\restr[2]{{
  \left.\kern-\nulldelimiterspace 
  #1 
  \vphantom{|} 
  \right|_{#2} 
  }}

\newcommand{\reqnomode}{\tagsleft@false}
\def\dx{\,{\rm d}x}

\def\dy{\,{\rm d}y}
\def \d{\,{\rm d}}

\def\supp{\,{\rm supp }}
\def\diam{\,{\rm diam}}

\allowdisplaybreaks
\makeatletter
\DeclareRobustCommand*{\bfseries}{%
  \not@math@alphabet\bfseries\mathbf
  \fontseries\bfdefault\selectfont
  \boldmath
}

\NewDocumentCommand{\esssup}{o}{%
  \operatorname*{ess\,sup}_{\IfValueT{#1}{#1}}}%

\makeatother

\newlength{\defbaselineskip}
\newcommand\eps\varepsilon
\def\mean#1{\mathchoice%
          {\mathop{\kern 0.2em\vrule width 0.6em height 0.69678ex depth -0.58065ex
                  \kern -0.8em \intop}\nolimits_{\kern -0.4em#1}}%
          {\mathop{\kern 0.1em\vrule width 0.5em height 0.69678ex depth -0.60387ex
                  \kern -0.6em \intop}\nolimits_{#1}}%
          {\mathop{\kern 0.1em\vrule width 0.5em height 0.69678ex
              depth -0.60387ex
                  \kern -0.6em \intop}\nolimits_{#1}}%
          {\mathop{\kern 0.1em\vrule width 0.5em height 0.69678ex depth -0.60387ex
                  \kern -0.6em \intop}\nolimits_{#1}}}

\numberwithin{equation}{section}
\allowdisplaybreaks[4]

\newcommand{\step}[1]{\medskip\noindent\textbf{Step #1. }}
\newcommand{\substep}[1]{\medskip\noindent\textit{Substep #1. }}

\newcommand\R{\mathbb{R}}

\newcommand{\ip}[2]{\langle #1,\,#2\rangle}

\def \p{\partial}
\def \e{\varepsilon}

\def \dd{\mathrm d}

\title{Global regularity and sparsity for regularised optimal transport}
\author[1]{Xiaopeng Cheng}
\author[2]{Lukas Koch}
\author[2]{Haotian Xiao}

\affil[1]{\small Max Planck Institute for Mathematics in the Sciences, 04103 Leipzig,Germany
\protect \\
  {\tt{xiaopeng.cheng@mis.mpg.de}}
  \vspace{1em} \ }

\affil[2]{\small University of Sussex, UK
\protect \\
  {\tt{lukas.koch@sussex.ac.uk, }}
  {\tt{hx75@sussex.ac.uk}}
  \vspace{1em} \ }

\usepackage{etoolbox}
\makeatletter
\patchcmd{\@adjustvertspacing}
  {\jot\baselineskip \divide\jot 4}
  {\jot=.5\baselineskip}
  {}{}
\makeatother

\makeatother

\begin{document}
\maketitle

\begin{abstract}
We study regularised optimal transport with subquadratic or entropic regularisation. Under $C^{1,\alpha}$-assumptions on the marginals and the domains of their support, we prove global gradient-Lipschitz bounds on the potentials, uniform in the regularisation parameter. Moreover, we show that, up to the boundary, the support of the conditional supports $\supp\,\pi_\e(\cdot| x)$ behaves like balls of radius $\e^\frac 2 {d(p-1)+2}$. 

Our work was carried out concurrently and independently of \cite{GSN}. We obtain the results of \cite[Theorem 1.2]{GSN} under weaker assumptions and with a different proof.
\end{abstract}

\section{Introduction}
We study properties of minimisers of the regularised optimal transport problem\
\begin{align}\tag{ROT}\label{eq:regularisedProblem}
OT_{\e,h}(\lambda,\mu):= \min_{\pi\in \Pi(\lambda,\mu)} \int_{\R^d\times \R^d}|x-y|^2 \dd \pi+\e^2 \int_{\R^d\times \R^d}h\left(\frac{\dd \pi}{\dd P}\right) \dd P,
\end{align}
where $\lambda,\mu\in \mathscr M(\R^d)$ have the same mass and $\Pi(\lambda,\mu)$ is the set of couplings between $\lambda$ and $\mu$ and $P\in \mathscr M(\R^d\times \R^d)$ a reference measure. If $\e$ is fixed, we denote $OT_h=OT_{\e,h}$. Throughout this paper, we will assume that $\lambda,\mu$ are absolutely continuous with respect to Lebesgue measure and the existence of $c>0$ such that the densities of $\lambda,\mu$, still denoted by $\lambda,\mu$, satisfy
\begin{align}\label{eq:densityBound}
    c^{-1}\leq \lambda,\mu\leq c
\end{align}
on $\supp\,\lambda$, $\supp\,\mu$, respectively. Further, we assume that $\supp\,\lambda$ and $\supp\,\mu$ are connected. We always consider the  reference measure $P=\lambda\otimes \mu$ and mainly focus on the following choice of regularisation $h$,
\begin{align}\label{eq:hp}
h_p(z)=\begin{cases}
\frac{|z|^p-1}{p-1} \quad& \text{ for } p\in (1,2]\\
z\log(z)\quad& \text{ for } p=1.
    \end{cases}
\end{align}
We write $OT_p(\lambda,\mu):=OT_{h_p}(\lambda,\mu)$ and refer the reader to Section \ref{sec:notation} for an overview of our notation.

Optimal transport is a convex minimisation problem, but not strictly convex. In contrast, if $h$ is strictly convex, the regularised problem \eqref{eq:regularisedProblem} is a strictly convex minimisation problem. This has many beneficial effects, including aiding numerical computations \cite{Cuturi2013}, improving sample complexity \cite{Genevay2019,Mena2019} and improving the stability of couplings \cite{Ghosal2022,Bayraktar2025}.

Under
\begin{assumption}\label{ass:h}
$h\colon [0,\infty)\to \R$ is continuous, twice differentiable on $(0,\infty)$, strictly convex, bounded below and superlinear at $+\infty$,
\end{assumption}
the regularised transport problem \eqref{eq:regularisedProblem} admits a dual formulation of the form
\begin{align}\label{eq:dualProblem}\tag{dual-ROT}
    \sup_{f\in L^1(\lambda),\,g\in L^1(\mu)} \int_{\R^d} f \dd \lambda + \int_{\R^d} g \dd \mu-\e^2\int_{\R^d\times \R^d} h^\ast\left(\frac{f(x)+g(y)-|x-y|^2}{\e^2}\right)\d P,
\end{align}
where $h^\ast$ is the convex conjugate of $h$ with the convention that we extend $h=+\infty$ on $(-\infty,0)$. It turns out that the supremum is assumed and that the optimal choice of potentials $(f_\e,g_\e)$ is unique up to constants. Indeed, the potentials $(f_\e,g_\e)$ will be our main object of interest, though we will phrase most of our results in terms of $\phi_\e(x) = |x|^2-f_\e(x)$ and $\psi_\e(y)=|y|^2-g_\e(y)$. The potentials are related to the optimal transport plan $\pi_\e$ solving $OT_h$ via the relation
\begin{align*}
    \frac{\dd \pi_\e}{\dd P} = (h')^{-1}\left(\frac{f_\e(x)+g_\e(y)-|x-y|^2}{\e^2}\right).
\end{align*}

If $T$ is the optimal transport map between $\lambda$ and $\mu$, then one expects the conditional supports of $\pi_\e$ shrink to a single point. Thus we are interested in the size of the conditional supports
\begin{align*}
\mathcal S_x&:=\supp\, \pi_\e(\cdot\mid x)=\mbox{closure of }\{y\mid \frac{\d\pi_\e}{\d P}(x,y)> 0\},\\
\mathcal T_y&:=\supp \pi_\e (\cdot\mid y)=\mbox{closure of }\{x\mid \frac{\d\pi_\e}{\d P}(x,y)> 0\}.
\end{align*}
The main goal of this paper is to prove the following result, containing both a global regularity result for potentials of entropic optimal transport and geometric information on their support. Let $B(z,R)\subset\R^d$  be the ball of radius $R$ centered at $z\in\R^d$.
\begin{theorem}\label{thm:main}
Let $p\in [1,2]$ and $\alpha\in (0,1]$. Let $\lambda,\mu$ be probability measures on $\R^d$ admitting $C^{0,\alpha}$-densities satisfying $0<c^{-1}\leq \lambda,\mu\leq c$ on their support for some $c>0$. Assume $\supp\,\lambda$ and $\supp\,\mu$ are $C^{1,\alpha}$-domains. Suppose the optimal transport map  $T\colon \supp\,\lambda \to \supp\,\mu$ between $\lambda$ and $\mu$ is a bi-$C^{1,\alpha}$ map.  Consider $\pi_\e\in \Pi(\lambda,\mu)$ an optimal transport plan for regularised transport with regularisation parameter $\e^2>0$ and regularisation function $h_p$ with potentials $\phi_\e$ and $\psi_\e$. There is $\e_1>0$ such that if $\e\leq \e_1$, then $\phi_\e\in W^{2,\infty}(\supp\,\lambda)$, $\psi_\e\in W^{2,\infty}(\supp\,\mu)$ and there is $C>0$, independent of $\e$, such that
\begin{align*}
C^{-1}\textup{id} \leq \nabla^2 \phi_\e \leq C \textup{id}\quad \text{ a.e. on } \supp\,\lambda,\\
C^{-1}\textup{id} \leq \nabla^2 \psi_\e \leq C \textup{id}\quad \text{ a.e. on } \supp\,\mu.
\end{align*}
If $p\in [1,2)$, then $\phi_\e\in C^2(\supp\,\lambda)$, $\psi_\e\in C^2(\supp\,\mu)$. 

Furthermore if $p\in(1,2]$,  with $R_c:= \e^\frac 2 {d(p-1)+2}$, for every\footnote{$A\simeq B$ means $A\lesssim B$ and $B\lesssim A$, see Section \ref{sec:notation}.} $x\in \supp\,\lambda$, $y\in \supp\,\mu$,
\begin{gather*}
B\left(\frac 1 2\nabla \phi_\e(x),C^{-1}R_c\right)\cap \supp\,\mu\subset \mathcal S_x\subset B\left(\frac 1 2\nabla \phi_\e(x),C R_c\right)\\
B(\frac{1}{2}\nabla \psi_\e(y),C^{-1}R_c)\cap \supp\,\lambda\subset \mathcal T_y\subset B(\frac{1}{2}\nabla \psi_\e(y),C R_c)\\
\left|B\left(\frac 1 2\nabla \phi_\e(x),C^{-1}R_c\right)\cap \supp\,\mu\right|\simeq |B(C^{-1}R_c)| \simeq \left|B\left(\frac 1 2\nabla \psi_\e(y), C^{-1}R_c\right)\cap \supp\,\lambda\right|.
\end{gather*}
\end{theorem}

\begin{remark}
If $\supp\,\lambda$ and $\supp\,\mu$ are convex, the $C^{1,\alpha}$ regularity assumption on the optimal transport map $T$ is automatically satisfied due to \cite{Caffarelli1992}  .
\end{remark}

Theorem \ref{thm:main} is a global version of the local results \cite[Theorem 3]{GvalaniKoch1} and \cite[Theorem 3.1]{GonzalezGvalaniKoch}. We refer the reader to these papers for an overview of the related literature and focus here only on very recent literature and results that are relevant to boundary regularity. In the case of quadratic regularisation $p=2$, \cite{Nguyen2026} proved lower bounds of order $\e^\frac 2 {d+2}$ on the Hausdorff distance from the regularised support to the Monge-graph. In the special case of an affine Brenier-map, they provide matching upper bounds. In the case $p=2$, the results in Theorem \ref{thm:main} was obtained under the stronger assumptions of $C^{1,\alpha}$ regularity of the marginals, and $C^{3,\alpha}$ regularity and uniform convexity of their supports, in \cite[Theorem 1.2.]{GSN}. \cite{GSN} furthermore provides sharp estimates on the error $\|\nabla \phi_\e-\nabla \phi_0\|_{L^\infty(\supp\,\lambda)}\simeq \e^\frac 2 {d+2}$, where $\phi_0$ is the Kantorovich potential associated to $W_2^2(\lambda,\mu)$. We note that we consider the cost function $|x-y|^2$ for $W_2$, see \eqref{e:W2-def}, and thus $\phi_0$ is twice the standard Brenier potential.

Both \cite{GSN} and this work build on the philosophy from \cite{GvalaniKoch,GvalaniKoch1}--- at large scales $R\gg R_c$ closeness to the optimal transport problem is exploited, while at the critical scale $R\sim R_c$ the structure of \eqref{eq:dualProblem} is exploited. In contrast to \cite{GSN}, which uses the boundary $\e$-regularity of \cite{Chen2015} to build affine comparison maps from a linearisation of the Monge-Amp\'{e}re equation near the boundary, we extend the interior argument from \cite{GvalaniKoch,GvalaniKoch1} to the boundary. In particular, we rely on the variational approach to optimal transport first introduced in \cite{GoldmanOtto}. We refer the reader to \cite{KochOtto} for an introduction to this approach. \cite{MiuraOtto} implements the variational approach to optimal transport in the boundary setting, leading to a sharp $\e$-regularity result in the style of \cite{Chen2015}. A key insight in \cite{MiuraOtto} is the necessity of a topological condition that allows to restrict to a purely local setting.

\begin{assumption}\label{ass:topological}
We say a plan $\pi\in \Pi(\lambda,\mu)$ satisfies the topological condition\footnote{$Q_R(p)$ is a cube of sidelength $2R$ centered at $p$, see Section \ref{sec:notation}.} in $Q_R(p)$, if 
\begin{align*}
(x,y)\in \supp\,\pi \cap ((Q_{R/2}(p)\cap \supp\,\lambda)\times \R^d) \quad\Rightarrow \quad y\in Q_R(p),\\
(x,y)\in \supp\,\pi \cap (\R^d\times (Q_{R/2}(p)\cap \supp\,\mu))\quad\Rightarrow \quad x\in Q_R(p).
\end{align*}
\end{assumption}

\cite{MiuraOtto} only considers minimisers of quadratic optimal transport with constant marginals and relies heavily on the Benamou-Brenier formulation of optimal transport. We update their ideas using the streamlined interior machinery of \cite{KochOtto} to obtain a large-scale boundary $\e$-regularity result for quasi-minimisers of optimal transport for non-constant densities. In the case of minimisers of quadratic optimal transport, this provides the following boundary $\e$-regularity result, which extends \cite{MiuraOtto} to non-constant densities.
\begin{corollary}\label{cor:OT}
 Suppose $\lambda,\mu$ admit $C^{0,\alpha}$ densities and $\Omega_0=\supp\,\lambda$, $\Omega_1=\supp\,\mu$ are $C^{1,\alpha}$-domains and let $T$ be the optimal transport map between $\lambda$ and $\mu$. Assume that
 $p\in \partial \Omega_0\cap \partial \Omega_1$ with $\nu_0(p)=\nu_1(p)$ and $\partial \Omega_i\cap Q_R(p)$ is connected for $i=0,1$, and Assumption \eqref{ass:topological}  holds in $Q_R(p)$. 
 There exist constants $\e,C>0$ depending only on $d,\alpha$ with the following property:  Set $\pi=(\text{id},T)_{\#}\lambda$, if 
    \begin{align*}
    &\e':=\frac{1}{R^{d+2}}\int_{(Q_R(p)\times \R^d)\cup (\R^d\times Q_R(p))}|y-x|^2\dd \pi\\
    &\quad+R^{2\alpha}([\lambda]_{\alpha,Q_R(p)}^2+[\mu]_{\alpha,Q_R(p)}^2+[\nu_0]_{\alpha,Q_R(p)}^2+[\nu_1]^2_{\alpha,Q_R(p)})\leq \e,
    \end{align*}
    then $T$ is of class $C^{1,\alpha}$ in $\overline{Q_{R/16}(p)\cap \supp\,\lambda}$ and
    \begin{align*}
        R^{2\alpha}[\nabla T]_{\alpha,Q_{R/16}(p)}^2 \leq C \e'.
    \end{align*}
\end{corollary}

At small scales, we exploit the structure of \eqref{eq:dualProblem}, building on the ideas of \cite{GvalaniKoch1,GonzalezGvalaniKoch}. While in the case of convex supports, the proofs of \cite{GvalaniKoch1,GonzalezGvalaniKoch} can be adapted without significant problems, in our case, particular care is required to deal with the potential non-convexity of the domains.

The structure of our paper is as follows. In Section \ref{sec:preliminaries}, we collect our notation and some elliptic regularity estimates. Section \ref{sec:regularisedOT} is devoted to collecting a number of known results on regularised optimal transport and establishing some easy consequences near the boundary. In Section \ref{sec:harmonixApproximation} we develop the harmonic approximation result for quasi-minimisers that lies at the heart of the variational approach. We apply this result to obtain a large-scale regularity theory in Section \ref{sec:large-scaleRegularity}. In combination with small-scale results obtained in Section \ref{sec:smallScaleResults}, this will give the proof of our main result in \ref{sec:mainResultProof}.

\section{Preliminaries}\label{sec:preliminaries}
\subsection{General notation}\label{sec:notation}
We write $A\lesssim B$ if there is $C>0$ such that $A\leq CB$. We write $A\ll B$ to mean the assumption that there exists $c>0$ such that if $A\leq c B$, then the given claim holds. We generally allow $C,c$ to depend on the dimension $d$, the parameter $p\in [1,2]$ determining the regularisation, the constant in \eqref{eq:densityBound} and $\alpha$ the H\"older continuity parameter of the marginals.

We denote the standard Euclidean inner product by $\langle\cdot,\cdot\rangle$. 

Given a set $\Omega \subset \R^d$, $\bar \Omega$ denotes its closure, while $\textup{conv}(\Omega)$ denotes its convex hull. Given two sets $A,B\subset \R^d$, $A+B=\{x\in \R^d\colon x=a+b \text{ for some } a\in A, b\in B\}$ denotes their Minkowski sum

Given a Lipschitz set $\Omega$, we denote the tangential gradient  on the boundary of $\Omega$, $\partial \Omega$ by $\nabla'$. The normal is denoted by $\nu$. We recall that Lipschitz domains satisfy a uniform inner cone condition: If $\Omega$ is a Lipschitz domain, then there are $\alpha,\rho>0$ such that for any $x\in \Omega$, there is a cone $C(\alpha,\rho)$ with vertex $x$, aperture $\alpha$ and height $\rho$ such that $C(\alpha,\rho)\subset\Omega$.

For $R>0$, we denote
\begin{align*}
	Q_R&:=(-R,R)^d\quad\text{ and } \quad\#_R:=(Q_R\times\R^d)\cup (\R^d\times Q_R).
    \end{align*}
Given $x\in \R^d$, $Q_R(x):=x+Q_R$ is the cube centered at $x$. We also set $Q_R':=(-R,R)^{d-1}$. 

For a $C^{k,\alpha}$-continuous function $f$ on $Q_R$, we denote its H\"older-seminorm on $Q_R$ by $[f]_{k,\alpha,R}$. If $k=0$, we denote $[f]_{\alpha,R}=[f]_{0,\alpha,R}$. The full H\"older norm is denoted by $\|f\|_{C^{k,\alpha}(Q_R)}$.

Given two Borel, nonnegative measures with finite second moments and equal mass $\lambda,\mu$, we let 
\begin{align*}
\Omega_0:=\supp\,\lambda,\quad \Omega_1:=\supp\,\mu
\end{align*}
denote their support respectively. 
If $\lambda,\mu$ admit Lebesgue densities, then we identify the measures and their density so that $\d\lambda(x)=\lambda(x)\dx$ and $\d\mu(y)=\mu(y)\dy$. The set of couplings between $\lambda$ and $\mu$ is denoted by $\Pi(\lambda,\mu)$, that is, $\pi\in \Pi(\lambda,\mu)$ if for all test functions $\zeta$,
\begin{align*}
 \int \zeta(x)\d\pi(x,y)=\int \zeta\d\lambda,
 \quad
 \int \zeta(y)\d\pi(x,y)=\int \zeta\d\mu.
\end{align*}
The quadratic optimal transport problem is given by
\begin{align}\label{e:W2-def}
 W_2^2(\lambda,\mu)
 :=\min_{\pi\in\Pi(\lambda,\mu)}
       \int_{\R^d\times\R^d}|x-y|^2\d\pi(x,y).
\end{align}
We refer to \cite{Villani2003} for an account of this problem. Particularly, under the assumption of Theorem \ref{thm:main}, there exists a $\lambda$-a.~e.~unique optimal transport map $T:\supp\,\lambda\to \supp\,\mu$ such that $\pi=(\textup{id,T})_\#\lambda$ is optimal for \eqref{e:W2-def}.

For any set $U\subset \R^d$, we define
\begin{align*}
	\p(U):=\{(x,y)\in\supp\,\pi\cap ((U\times \R^d)\cup (\R^d\times U))\mid \exists t\in[0,1]\:\:X(t)\in\p U \}
\end{align*}
where for $x,y\in \R^d$, we write $X(t):=(1-t)x+ty$ for $t\in[0,1]$.  We set  $\p(\#_R):=\p(Q_R)$.

\subsection{Notation for key quantities}
We collect here the definitions of a number of key quantities we use throughout the paper.

Given a coupling $\pi\in \Pi(\lambda,\mu)$ and $R>0$, we denote the local transport cost by
\begin{align*}
E(\pi,R) := \frac 1 {R^{d+2}} \int_{\#_R} |x-y|^2 \d \pi.
\end{align*}
When considering regularised optimal transport, it is useful to consider the extended energy
\begin{align}\label{eq:extendedEnergy}
E_p(\pi,R):=\begin{cases}
E(\pi,R) \quad& \text{ if } p=1\\
E(\pi,R)+R^{-(d+2)}\e^2 \int_{\#_R} h_p\left(\frac{\dd\pi}{\dd P}\right) \dd P\quad& \text{ if } p>1.
\end{cases}
\end{align}

We need to track two sources of error arising from the marginals. First, we need to quantify how close the marginals locally are to the Lebesgue measure. For measures $\lambda,\mu$, we denote
\begin{align*}
\kappa_{\lambda,R} := \frac{\lambda(Q_R)}{|Q_R\cap \supp\,\lambda|},\quad \kappa_{\mu,R} := \frac{\mu(Q_R)}{|Q_R\cap \supp\,\mu|}
    \end{align*}
and we define 
\begin{align*}
	D(R,\lambda,\mu)&:=\frac{1}{R^{d+2}}W_2^2(\lambda\llcorner Q_R,\kappa_{\lambda,R}\d x\llcorner\supp\,\lambda\cap Q_R)+(\kappa_{\lambda,R}-1)^2\\
	&\quad+\frac{1}{R^{d+2}}W_2^2(\mu\llcorner Q_R,\kappa_{\mu,R}\d x\llcorner\supp\,\mu\cap Q_R)+(\kappa_{\mu,R}-1)^2
\end{align*}
where we think of $D(R,\lambda,\mu)$ as controlling the local ``flatness'' of $\lambda,\mu$ in $Q_R$. If the choice of measure $\lambda,\mu$ is clear from the context, we write $D(R)$. 

Second, we need to measure how flat the boundary is. We will usually reduce to the following normalised situation. Let $B(p)$ be either a cube or a ball with center $p$.
\begin{assumption}\label{ass:tangency}
Given $C^{1,\alpha}$-domains $\Omega_0,\Omega_1$, we say the tangency condition is satisfied in $B(p)\subset \R^d$ if
\begin{align*}
    p\in \partial \Omega_0\cap \partial\Omega_1,\quad \nu_0(p)=\nu_1(p),\quad \partial \Omega_i\cap B(p) \text{ is connected for } i=0,1,
\end{align*}
where $\nu_i(p)$ is the normal to $\Omega_i$ at $p$.
\end{assumption}

We frequently use the following equivalent (up to a translation and rotation of coordinates) characterization of Assumption \ref{ass:tangency} in local coordinates: 
\begin{assumption}\label{ass:tangencyLocal}
There are $C^{1,\alpha}$ functions $g_0,g_1:B'\to\R$ such that for $i=0,1$,
\begin{align*}
    \Omega_i \cap B(p) = \{(x_1,x')\in B(p) \mid x_1>g_i(x'), x'\in B'\},\quad g_i(0)=0,\, \nabla' g_i(0)=0,
\end{align*}
where $B' = B\cap \{x_1=0\}$.
\end{assumption}

We write 
\begin{align*}
D_{bdy}(R):= \sup_{Q_R'} |g_0|^2+\sup_{Q_R'} |g_1|^2.
\end{align*}
Note that if Assumption \ref{ass:tangencyLocal} holds, then 
\begin{align*}
    D_{bdy}(R)\lesssim R^2\left( Lip(g_0,Q_R)^2+Lip(g_1,Q_R)^2 \right).
\end{align*}
Finally, we combine information on the boundary regularity and the regularity of marginals into a single data term
\begin{align*}
\mathscr D(R,\lambda,\mu):= R^{2\alpha}([\lambda]_{\alpha,R}^2+[\mu]_{\alpha,R}^2)+R^{2\alpha}([\nabla g_0]_{\alpha,Q_R'}^2+[\nabla g_1]_{\alpha,Q_R'}^2).
\end{align*}

\subsection{Regularity estimates for the Poisson equation}
In this section we provide a number of regularity estimates for solutions of the following problem
\begin{align*}
	-\Delta\phi=c\text{ on }Q,\quad \nabla\phi\cdot\nu=f\text{ on }\partial Q
\end{align*}
where $Q:=(0,a)\times (0,1)^{d-1}$ for some $a>0$, and $f$ is Lipschitz on 
each open face of $\partial Q$ (i.e.\ wherever the normal $\nu$ to $\partial Q$ is continuous). The constant $c$ is such that the problem is solvable. We consider the unique solution satisfying 
\begin{align}\label{c13}
	\int_{Q}\phi=0.                       
\end{align} 
\begin{lemma}\label{lem:regularity} For any $0<\alpha<1$, the following estimates hold:
	\begin{align}
		\|\phi\|_{C^{1,\alpha}(\bar Q)}&\lesssim \|f\|_{C^\alpha(\partial Q)},\label{c52}\\
		\|\phi\|_{C^{3}(\overline{\Omega})}&\lesssim \|f|_{\{0\}\times (0,1)^{d-1}}\|_{C^{2}}+\|f\|_{L^2(\p Q) }\label{c53}
	\end{align}
    where $\Omega=\left(0,\frac{a}{2}\right)\times \left(\frac{1}{4},\frac{3}{4}\right)^{d-1}$.
\end{lemma}

\begin{proof}
We first show \eqref{c52}. Label the $2d$ open faces of $Q$ as $F_1^\pm,\dots,F_d^\pm$, where $F_i^- = \{x_i=0\}\cap\partial Q$ and $F_i^+ = \{x_i=1\}\cap\partial Q$ (with the exception for $F_1^+ = \{x_1=a\}\cap \p Q$). Writing $f$ as a sum of its restrictions to each open face and by linearity, it suffices to assume that $f$ is supported on a single face, say on	$F_1^- = \{0\}\times (0,1)^{d-1}$. Hence, we may identify $f$ with a Lipschitz function $f:(0,1)^{d-1}\to\R$. It then suffices to show that
	\begin{align*}
		\|\phi\|_{C^{1,\alpha}([0,a/2]\times [0,1]^{d-1}) }\lesssim \|f\|_{C^{\alpha}}
	\end{align*}
which combined with	interior Schauder estimates as well as the standard energy estimate ${\|\phi\|_{L^2(Q)}\lesssim \|f\|_{L^2}}$ completes the proof, see e.g. \cite{GT}.

\medskip
Let $\tilde Q:=(0,a)\times (-1,2)^{d-1}$. Extend $\phi$ and $f$ by even reflection. More precisely,	define $\sigma: (-1,2)\to (0,1)$ by  $\sigma(t)=|t|$ for $t\in(-1,1] $ and $ \sigma(t)=2-t$ for $t\in (1,2) $. Define
\begin{align*}
\tilde\phi(x_1,\dots,x_d):= \phi\big(x_1,\,\sigma(x_2),\dots,\sigma(x_d)\big)
\end{align*}
and similarly for $\tilde f$. Note that $\|\tilde f\|_{W^{1,\infty}((-1,2)^{d-1})}\leq \|f\|_{W^{1,\infty}((0,1)^{d-1})}$ and
\begin{align}
	-\Delta\tilde\phi=c\text{ on }\tilde Q,\quad
	\nabla\tilde\phi\cdot\nu=\tilde f\text{ on }\partial\tilde Q.
\end{align}
Let $\Omega$ be a smooth domain compactly supported in 
$\overline{\tilde Q} $ such that $ [0,a/2]\times [0,1]^{d-1} $ is compactly supported in $\bar\Omega$. As $\Omega$ is a smooth domain, by boundary Schauder theory and energy estimate, see \cite{GT},
\begin{align}
	\|\phi\|_{C^{1,\alpha}([0,a/2]\times [0,1]^{d-1}) }\leq \|\tilde\phi\|_{C^{1,\alpha}(\bar\Omega)}\lesssim \|\nabla\phi\|_{L^2(\tilde Q)}+ \|\bar f\|_{C^{\alpha}(\p \tilde Q)}\lesssim \|f\|_{C^\alpha}
\end{align}
which concludes the proof. 

To show \eqref{c53}, let $\tilde{\Omega}$ be a smooth domain compactly supported in $Q$ such that $\Omega\subset \tilde{\Omega}$. The conclusion then follows from standard boundary Schauder theory, cf. \cite[Theorem 6.30]{GT}.

\end{proof}

\section{Regularized optimal transport}\label{sec:regularisedOT}
Assume $h:[0,\infty)\to\R$ satisfies Assumption \ref{ass:h} and consider $OT_h$.  Under assumptions in this paper, the above problem admits a minimizer $\pi_\e$. 

Note that adding an affine function to $h$ changes the functional in
\eqref{eq:regularisedProblem} by a constant; while multiplying $h$ by a factor $c>0$ results in $OT_{\eps,ch}=OT_{c^{1/2}\eps,h}$. Hence, various normalization convention for entropic and polynomially regularized optimal transport in the literature are equivalent.

\subsection{Duality and basic properties of the potentials}
As noted in the introduction $OT_{\e,h}(\lambda,\mu)$ being a convex minimization problem, admits a dual problem,
\begin{align}\label{e:ROT-dual}\tag{dual-OT}
\sup_{f\in L^1(\lambda),g\in L^1(\mu)}\int_{\R^d}f\d\lambda+\int_{\R^d}g\d \mu-\e^2\int_{\R^d\times \R^d}h^\ast \left( \frac{f(x)+g(y)-|x-y|^2}{\e^2} \right)\d P
\end{align}
where $h^\ast$ is the convex conjugate of $h$. We use the shorthand notation
\begin{align}\label{e:xi-def}
\xi(x,y):=f_\e(x)+g_\e(y)-|x-y|^2
\end{align}
Recall that we write
\begin{align*}
\phi_\e(x)=|x|^2-f_\e(x),\quad \psi_\e(y)=|y|^2-g_{\e}(y).
\end{align*}
We denote
\begin{align*}
\mathcal S_x = \{y\in \supp\,\mu\colon \xi(x,y)\geq \e^2 h'(0)\},\quad \mathcal T_y = \{x\in \supp\,\lambda\colon \xi(x,y)\geq \e^2 h'(0)\}.
\end{align*}

We collect some basic properties of the dual problem and the potentials
\begin{proposition}[Basic properties of the potentials]
\label{sec:basicPropertiesPotentials}	
Let $\lambda,\mu$ be compactly supported, absolutely continuous
probability measures on $\mathbb R^d$. Assume $h$ satisfies Assumption \ref{ass:h} and consider the problems \eqref{eq:regularisedProblem} and \eqref{eq:dualProblem}. Then the following statements hold.

\begin{enumerate}
\item The problem \eqref{eq:regularisedProblem} admits a unique minimiser
$\pi_\e$, and its dual \eqref{e:ROT-dual} admit a maximising pair
$(f_\e,g_\e)$. It holds
\begin{align*}
\frac{d\pi_\e}{dP}(x,y) &= (h')^{-1}\left(\frac{\xi(x,y)}{\e^2}\right)\\
\supp\,\pi_\e &= \{(x,y)\in \supp\,\lambda\times \supp\,\mu\colon \xi(x,y)\geq \e^2 h'(0)\}.
\end{align*}
In particular, there is no duality gap.
\item The  potentials $(\phi_\e,\psi_\e)$ admit convex Lipschitz representatives on their respective
supports, satisfying
\[
\mbox{Lip}(\phi_\e;\mbox{supp}\,\lambda)
+
\mbox{Lip}(\psi_\e;\mbox{supp}\,\mu)
\leq C,
\]
where $C$ depends only on the marginals and is independent of
$\e$. Moreover, $(\phi_\e,\psi_\e)$ may be extended to convex Lipschitz functions on $\R^d$ satisfying for every $x,y\in \R^d$,
\[
\int \frac{d\pi_\e}{dP}(x,y)\,d\mu(y)=1,
\qquad
\int \frac{d\pi_\e}{dP}(x,y)\,d\lambda(x)=1
\]

\item In the case $h=h_p$ with $p\in [1,2]$, it holds
\begin{equation}\label{eq:gradPhiPower}
\nabla\phi_\e(x)
= \begin{cases}
2\frac{\int_{\mathcal S_x} y\,\xi(x,y)^{\frac{2-p}{p-1}}\,d\mu(y)}
{\int_{\mathcal S_x}
\xi(x,y)^{\frac{2-p}{p-1}}\,d\mu(y)}\quad& \text{ if } p>1\\\
2\int y\,
\exp\!\left(\frac{\xi(x,y)}{\e^2}-1\right)\,d\mu(y). \quad&\text{ if } p=1.
\end{cases}
\end{equation}
Analogous formulas hold for $\nabla\psi_\e$.
\end{enumerate}
\end{proposition}

\begin{proof}
The statements on $\supp\,\lambda\times \supp\,\mu$ follow from \cite[Proposition 18 and Lemmas 19--21]{GvalaniKoch1};
continuity extends the marginal identities to every point of the supports.

We now detail how to obtain the extensions to $\R^d$. Extend $\phi_\e$ by imposing its second marginal identity at every $x\in\mathbb R^d$. Keeping the original values of $\psi_\e$ on
$\operatorname{supp}\mu$, for each fixed $x$ we seek $t$ satisfying
$F(x,t)=1$, where
\[
F(x,t):=\int (h')^{-1}\left(\frac{2\langle x,y\rangle - t-\psi_\e(y)}{\e^2}\right)\,d\mu(y).
\]
The second marginal identity gives $F(x,\phi_\e(x))=1$ for $x\in \supp\,\lambda$. 
For each fixed $x$, boundedness of $y$ and $\psi_\e(y)$ on
$\operatorname{supp}\mu$ gives
\[
\lim_{t\to-\infty}F(x,t)=+\infty,
\qquad \lim_{t\to+\infty}F(x,t)=0.
\]
The function $F(x,\cdot)$ is continuous and strictly decreasing while positive.
Define $\phi_\e(x)$ as its unique solution of $F(x,t)=1$. Uniqueness preserves the original values on $\operatorname{supp}\lambda$. Define $\psi_\e$ symmetrically from the original values of $\phi_\e$ on $\operatorname{supp}\lambda$.
With $\xi(x,y)=\langle x, y\rangle-\phi_\e(x)-\psi_\e(y)$, we now have
\begin{equation}\label{eq:globalRow}
\int (h')^{-1}\left(\frac{\xi_\e(x,y)}{\e^2}\right)\,d\mu(y)=1
\end{equation}
and the analogous result for every $y\in\mathbb R^d$. The convexity argument
in the proof of \cite[Lemma 10]{GvalaniKoch1} relies only on \eqref{eq:globalRow} and its analogous result for $y\in \R^d$. Consequently, the extensions are convex and thus automatically Lipschitz on $\R^d$. The gradient formula now follows as in \cite[Proposition 4.2]{GonzalezGvalaniKoch}.
\end{proof}

\subsection{Convergence to optimal transport}
The issue of convergence of ROT to OT has been established in various senses. We recall here the convergence of the energy. 
\begin{prop}[Convergence of the energy]\label{prop:EnergyConvergence}
Let $\lambda,\mu$ be measures with compact support and bounded densities. Then for $\e\ll 1$,
\begin{align*}
OT_{\e,h_p}(\lambda,\mu)-W_2^2(\lambda,\mu)
 \lesssim
 \begin{cases}
\e^2\log(\e^{-2})+\e^2,&p=1,\\
\e^{\frac{4}{d(p-1)+2}},&1<p\leq2,
 \end{cases}
\end{align*}
where the implicit constant depends on $d$, $\lambda,\mu$, and $p$. If $p=1$ and $\pi_\e$ solves $OT_{\e,h_1}$, it holds in addition that
\begin{align*}
\int |x-y|^2 \dd \pi_\e-W_2^2(\lambda,\mu)\lesssim \e^2.
\end{align*}
\end{prop}
This result under various regularity of the assumptions has been established by many authors. In the stated form it is a direct consequence of \cite{EcksteinNutz}. We refer to \cite{CCN} for an account of the literature and for sharp assumptions as well as asymptotic results.

Next suppose that $\lambda,\mu$ also admit Lebesgue densities, then Brenier's theorem supplies a optimal transport map $T=\nabla\phi$ for some convex potential $\phi$ so that $T_\#\lambda\mu$. The convergence  rate of ROT above also gives an estimate of the $L^2$ distance of the support of $\pi_\e$ to the graph of $T$, thanks to the following  Lemma, cf. \cite[Lemma 3.8]{MalamutSylvestre}.

\begin{lemma}\label{lem:EnergyConvergence}
Suppose $T$ is globally Lipschitz.  Then every $\pi\in\Pi(\lambda,\mu)$ satisfies
\begin{align*}
 \int|y-T(x)|^2\d\pi
   \leq \|\nabla T\|_{L^\infty}\left(\int|x-y|^2\d\pi-W_2^2(\lambda,\mu)\right).
\end{align*}
\end{lemma}

\subsection{Affine changes of coordinates}\label{sec:changeOfCoordinates}
We recall that the $\textup{ROT}_{\e,h}$ problem is affine invariant in the following sense. Consider a tuple
\begin{align*}
s:=(A,b,\gamma,\kappa)
\end{align*}
where $A\in\R^{d\times d}$ is positive definite and symmetric, $b\in\R^d$ and $\gamma,\kappa>0$. Given a coupling $\pi\in\Pi(\lambda,\mu)$ and $s$ given above, there is a coupling $\pi_s\in\Pi(\lambda_s,\mu_s)$ defined by 
\begin{align*}
\lambda_s:=\kappa (Q_1)_\#\lambda,\quad \mu_s:=\kappa (Q_2)_\#\mu,\quad \pi_s:=\kappa Q_\#\pi
\end{align*}

where $Q:\R^d\times\R^d\to \R^d\times\R^d$ is an affine map given by
\begin{align*}
Q(x,y)=\left(Q_1(x),Q_2(y)\right):=(A^{-1}x,\gamma A(y-b)).
\end{align*}
We remark that composing the change of coordinates from two tuples $s_1$ and $s_2$ gives a third change of coordinates $s'$.

The problem $OT_{\e,h_p}$ is affine invariant in the sense that if $\pi_\e$ is optimal for $OT_{\e,h_p}(\lambda,\mu)$, then $\pi_{\e,s}$ is optimal for  $OT_{\e\kappa^\frac{p-1} 2\gamma^{\frac{1}{2}},h_p}(\lambda_s,\mu_s)$.

\subsection{Boundary normalization}\label{SS:bdy-normalization}
In this section we work under the assumptions of Theorem \ref{thm:main}. 

 Let $\pi_\e$ be a solution to \eqref{eq:regularisedProblem}. We show how to reduce to Assumption \ref{ass:tangency}. Indeed, fix any $x_0\in \partial \supp\,\lambda$. By translation, we may assume $x_0=0$. Consider the affine change of coordinates associated with
\begin{align*}
s=(A,b,\gamma,\kappa):=\left( \nabla T(0)^{\frac 1 2}, T(0),\left(\frac{\mu(T(0))}{\lambda(0)\det(\nabla T(0))}\right)^{\frac 1 d}, \frac{1}{\lambda(0)\det(\nabla T(0)^\frac 1 2)} \right).
\end{align*}

Correspondingly, set $ T_s(\cdot):= Q_2(T(\cdot))$, $ \pi_{\e,s} =\kappa Q_{\#} \pi_\e$, $\lambda_s = \kappa(Q_1)_\# \lambda$, $ \mu_s = \kappa(Q_2)_\# \mu$. Then $T_s$ is optimal for the optimal transport problem between $ \lambda_s$ and $ \mu_s$ and $\pi_{\e,s}$ is optimal for  for  $\textup{ROT}_{\e\kappa^\frac{p-1} 2\gamma^{\frac{1}{2}},h}(\lambda_s,\mu_s)$.  
Moreover, 
\begin{align*}
T_s(0)=0,\:\nabla T_s(0)=\textup{Id}, \quad \lambda_s(0)=\mu_s(0)=1.
\end{align*}
Consequently, with $\Omega_0 = \supp\,\lambda_s$ and $\Omega_1=\supp\,\mu_s$, Assumption \ref{ass:tangency} is satisfies. Indeed, let $\gamma(t)$ be a curve on $\partial \Omega_0=\partial Q_1(\supp\,\lambda)$ with $\gamma(0)=0$. Then $T_s(\gamma(t))$ is a curve on $\partial \Omega_1=\partial Q_2(\supp\,\mu)$ with $T_s(\gamma(0))=0$. Moreover,
\begin{align*}
\restr{\frac{\dd}{\dd t} T_s(\gamma(t))}{t=0} = \nabla T_s(\gamma(0))\gamma'(0)=\gamma'(0).
\end{align*}
Thus, $T_s$ preserves tangent vectors at $0$ and the normals of $\Omega_0$ and $\Omega_1$ at $0$ agree.

\subsection{Boundary \texorpdfstring{$L^\infty$}{} bound for ROT: \texorpdfstring{$1<p\leq 2$}{}}\label{sec:Linfty}
We  establish a boundary $L^\infty$ bound for \eqref{eq:regularisedProblem} when $1<p\leq 2$. We denote $\Omega_0 = \supp\,\lambda$, $\Omega_1=\supp\,\mu$ and assume Assumption \ref{ass:tangencyLocal} holds. We note that in view of Section \ref{SS:bdy-normalization}, this assumption is not restrictive.
As it is important to keep track of the precise constant at times in the proof, we introduce the following notations:
\begin{align*}
	w_d:=|B_1| \text{ with $B_1$ the unit ball in }\R^d,\\
	E=E(\pi,2R,b)=\frac{1}{(2R)^{d+2}}\int_{B_{2R}\times\R^d}|y-x-b(x)|^2d\pi,\\
	L_1:=\mbox{Lip}(g_0,B_{2R})+\mbox{Lip}(b,B_{2R}),\\
	D:=\text{uniform density lower bound of }\lambda\mbox{ in }B_{2R}
\end{align*}
where $b$ is a given bi-Lipschitz vector field.
As $g$ is Lipschitz, by geometric considerations, there exists $\theta=\theta(L_1), C_0=C_0(L_1)$ such that for $e\in\mathbb{S}^{d-1}$  and $x\in\Omega_0\cap Q_R$, it holds that 
\begin{align}\label{c40}
	\langle e,e_1\rangle\geq \theta,\, r\leq\frac{1}{1+C_0}\quad \Rightarrow \quad Q_{C_0rR}(x+rRe)\subset \Omega_0\cap Q_{2R}.
\end{align}

\begin{proposition}\label{prop:LinftyPartial}
Assume $h\colon [0,\infty)\to \R$ satisfies Assumption \ref{ass:h} and $h'(0)=0$. Let $p\in (1,2]$. Suppose $\pi$ is a solution of $OT_{\e,h}(\lambda,\mu)$. 

There exists a critical length scale $R_c=R_c(\epsilon,d,L_1)$ and a modulus of continuity $\tau:(0,\infty)\to(0,\infty)$, such that for any $R>R_c$ for which Assumption \ref{ass:tangencyLocal} holds in $Q_{2R}$, the following hold.
	\begin{enumerate}
	    \item Partial control up to the boundary: For any $e\in\mathbb{S}^{d-1}$ with $\langle e,e_1\rangle\geq \theta$ and any bi-Lipschitz vector field $b$,
	\begin{align*}
		\sup_{(x,y)\in (\Omega_0\cap Q_R)\times\R^d\cap \supp \pi}\frac{(y-x-b(x))\cdot e}{R}\lesssim \max\left(E^{\frac{1}{2}},E^{\frac{1}{d+2}},\tau(\epsilon)/R \right)
	\end{align*}
	where the implicit constant depends on $d,D,L_1$.
	\item Interior estimate: If $x\in \Omega_0\cap Q_R$ satisfies
	\begin{align}\label{c48}
		x\cdot e_1\geq RL_1+R\max\left( E^{\frac 1 2},E^{\frac{1}{d+2}}, \tau(\epsilon)/R \right),
	\end{align}
	then for any $y$ such that $(x,y)\in (\Omega_0\cap B_R)\times\R^d\cap \supp\pi$,
	\begin{align*}
		\frac{|y-x-b(x)|}{R}\lesssim \max\left( E^{\frac 1 2},E^{\frac{1}{d+2}},\tau(\epsilon)/R \right)
	\end{align*}
	where the implicit constant depends on $d,D,Lip(b,Q_{2R})$.
    \end{enumerate}
    If $h=h_p$ for $p>1$, then $R_c\simeq \e^\frac 2 {d(p-1)+2}$.
\end{proposition}

\begin{proof}The proof proceeds in a number of steps.

\step 1 Two facts about optimal plans for ROT.
First, recall monotonicity of $\supp\,\pi$: for $(x,y),(x',y')$ write
\begin{align*}
	\Delta=\Delta(x,y,x',y')
	:=|x-y'|^2+|x'-y|^2-|x-y|^2-|x'-y'|^2,
\end{align*}
Then  for  any $(x,y),\,(x',y')\in\supp\pi$ with $\Delta\leq0$, (see \cite[Equation (7)]{GvalaniKoch1},
\begin{align}\label{c44}
	(h')^{-1}\left(\frac{-\Delta}{2\epsilon^2}\right)\leq \max\left(\frac{d\pi}{dP}(x',y),\frac{d\pi}{dP}(x,y')\right).
\end{align}
The second fact claims existence of points with bounded density:  Let $(x,y)\in \supp\pi$ and $B\subset\R^d\times\R^d$ be such that $\pi(B)>0$, then 
\begin{align}\label{c43}
	\pi\left(\{(x',y')\in B\mid \max\big(\frac{d\pi}{dP}(x,y'),\frac{d\pi}{dP}(x',y) \big)\}\leq \frac{2}{\pi(B)} \right)>0.
\end{align}
Indeed,  it suffices to show that $\pi(B_1)<\frac{\pi(B)}{2}$ and $\pi(B_2)<\frac{\pi(B)}{2}$ where
\begin{align*}
	B_1:=\{(x',y')\in B\mid \frac{d\pi}{dP}(x',y)>\frac{2}{\pi(B)} \},\quad B_2:=\{(x',y')\in B\mid \frac{d\pi}{dP}(x,y')>\frac{2}{\pi(B)}\}.
\end{align*}
We next show that $\pi(B_1)<\frac{\pi(B)}{2}$ while the argument for $B_2$ is the same. Let $A_1\subset\R^d$ be the projection to the first marginal of $B_1$ so that $\pi(B_1)\leq \pi(A_1\times\R^d)\leq\lambda(A_1)$. Now the conclusion follows from the optimality condition
\begin{align*}
	1=\int \frac{d\pi}{dP}(x',y)\,d\lambda(x')\geq\int_{A_1}\frac{d\pi}{dP}(x',y)\,d\lambda(x')>\frac{2}{\pi(B)}\lambda(A_1)\geq\frac{2}{\pi(B)}\pi(B_1).
\end{align*}

\step 2 Proof of (i).
Fix $(x,y)\in\supp\pi$ with $x\in\Omega_0\cap Q_R$ and $e\in\mathbb{S}^{d-1}$ with $\langle e, e_1\rangle\geq \theta$. Let $0<r<\frac{1}{1+C_0}$ and $\bar r>0$ to be fixed below. We also use the shorthand notation
\begin{align*}
	v:=y-x-b(x),\qquad v':=y'-x'-b(x').
\end{align*}
We define 
\begin{align*}
	B_{C_0rR}&:=B_{C_0rR}(x+rRe)\overset{\eqref{c40}}{\subset} \Omega_0\cap Q_{2R},\\
	B(x,y)&:=\Bigl\{(x',y')\in B_{C_0rR}\times\R^d\mid \
	|y'-x'-b(x')|\le \bar r R,\quad
	\langle x'-(x+rRe),v\rangle\geq 0\Bigr\}.
\end{align*}
\substep 1 Estimate of $\Delta$. We show that for any $(x',y')\in B(x,y)$,
\begin{align}\label{c41}
	-\Delta\geq rR \langle y-x-b(x),e\rangle\quad \mbox{provided}\quad \langle y-x-b(x),e\rangle\geq 2(L_1+1)(C_0+1)^2(r+\bar r)R.
\end{align}

For $(x',y')\in B(x,y)$, we have
\begin{align*}
	\Delta=2\ip{x-x'}{y-y'}
	=2\ip{x-x'}{v-v'}+2\ip{x-x'}{(x-x')+(b(x)-b(x'))}.
\end{align*}
Noting $x-x'=(x-x'+rRe)-rRe$ in the first term and estimating the last term using that $b$ is Lipschitz, we obtain that
\begin{align*}
	\Delta\leq 2\ip{x-x'+rRe}{v-v'}+2\ip{-rRe}{v-v'}+2(L_1+1)(C_0+1)^2
	(rR^2).
\end{align*}
Next, by definition of $B(x,y)$,
\begin{align*}
	\ip{x-x'+rRe}{v}\le0,\qquad -\ip{x-x'+rRe}{v'}\le C_0 r\bar r R^2,\\
	2\ip{-rRe}{v-v'}=-2rR\langle v,e\rangle+2rR\langle e,v'\rangle\le-2rR\langle v,e\rangle+2r\bar rR^2.
\end{align*}
Collecting estimates gives
\begin{align*}
	\Delta\leq\ -2rR\langle v,e\rangle+2(C_0+1)r\bar rR^2+2(L+1)(1+C_0)^2(rR)^2,
\end{align*}
from which \eqref{c41} follows.

\substep 2 Lower bound on $\pi(B(x,y))$. We show that with an appropriate choice of $r,\bar r$, 
\begin{align*}
	\pi(B(x,y))\sim (rR)^d.
\end{align*}
Let $H:=\{z\mid\langle z-(x+rRe),v\rangle\ge0\}$, so
$|B\cap H|=\omega_d(C_0rR)^d/2$. Since $B\cap H\subseteq\Omega_{0}\cap B_{2R}$, we have
\begin{align*}
	\pi\big((B\cap H)\times\R^d\big)=\lambda(B\cap H)\geq D\omega_d (C_0rR)^d/2.
\end{align*}
Thus by Markov's inequality,
\begin{align*}
	\pi(B_{C_0rR}\times\R^d\backslash B(x,y))\leq\frac{1}{(\bar rR)^2}\int_{B_{C_0rR}\times\R^d}|v'|^2\,d\pi
	\ \le\ \frac{(2R)^{d+2}E}{(\bar rR)^2}
	\ =\ \frac{2^{d+2}R^d E}{\bar r^2}.
\end{align*}
Choosing $\bar r$ to be
\begin{equation}\label{c42}
	\bar r=\Bigl(\frac{2^{\,d+4}}{D\,C_0^{\,d}\,\omega_d}\Bigr)^{1/2} r^{-d/2}E^{1/2},
\end{equation}
we obtain that $2^{d+2}R^d E/\bar r^2 \leq D\omega_d(C_0rR)^d/4$. Hence, we find that
\begin{align*}
	\pi\bigl(B(x,y)\bigr)\geq
	\frac{1}{4} D C_0^d\omega_d (rR)^d \gtrsim (rR)^d .
\end{align*}
The opposite inequality is obvious.

\substep 3 Conclusion. By Substep 1 and 2, we know that whenever $(x,y)$ satisfies \eqref{c41} with the choice of $r,\bar r$ also satisfying \eqref{c42}, by \eqref{c44} and \eqref{c43} we can find $(x',y')\in B(x,y)$ such that for some $C<\infty$,
\begin{align}\label{c45}
	(rR)^d(h')^{-1}\left(\frac{rR(y-x-b(x))\cdot e}{2\epsilon^2} \right)\leq C.
\end{align}

We begin by choosing
\begin{align*}
	\bar r=E^{\frac 1 2},\quad\mbox{hence by \eqref{c42}}\quad r=K_1(d,D):=\left(\frac{2^{d+4}}{DC_0^dw_d} \right)^{\frac{1}{d}}.
\end{align*}
Therefore, if we assume
\begin{align}\label{c46}
	\langle y-x-b(x),e\rangle\geq 2(L_1+1)(C_0+1)^2(K_1+E^{\frac 1 2})R\geq K_2rR,
\end{align}
with $K_2=K_2(d,L_1):=2(L_1+1)(C_0+1)^2$, then \eqref{c45} holds, and consequently by monotonicity of $(h')^{-1}$ and the fact that $R\geq R_c$, we see that 
\begin{align*}
	(rR_c)^d(h')^{-1}\left(\frac{K_2(rR_c)^2}{2\epsilon^2} \right)\leq C.
\end{align*}
Now we choose $R_c$ such that it satisfies
\begin{align*}
	(rR_c)^d(h')^{-1}\left(\frac{K_2(rR_c)^2}{2\epsilon^2} \right)=2C.
\end{align*}
This provides a contradiction and we conclude that the opposite of \eqref{c46} must be true:
\begin{align}\label{c47}
	\langle y-x-b(x)), e\rangle\leq K_2 (K_1+E^{\frac 1 2})R\quad\text{for}\quad R\geq R_c.
\end{align}

If $E\geq 1$, \eqref{c47} completes the proof. Hence, we may assume that $E\leq 1$ and that \eqref{c47} holds. We choose now
\begin{align*}
	r&=\frac{(y-x-b(x))\cdot e}{R}K_3\mbox{ with } K_3=\min\{\frac{1}{K_2(1+K_1)(1+C_0)},\frac{1}{4K_1(L_1+1)(C_0+1)^2} \},\\
	\bar r&=K_1^{\frac d 2}\left(\frac{(y-x-b(x))\cdot e}{R} \right)^{-\frac d 2}E^{\frac{1}{2}}.
\end{align*}
Now we assume that \eqref{c41} holds for $(x,y)$, then the choice of $K_3$ ensures that $r\leq \frac{1}{1+C_0}$ and that if \eqref{c41} holds, then 
\begin{align*}
	\frac{\langle y-x-b(x),e\rangle}{R}\geq K_4(d,D,L_1) E^{\frac{1}{d+2}}\text{ with } K_3:=\left(\frac{2(L_1+1)(C_0+1)^2K_1^{\frac d 2}}{1-2(L_1+1)(C_0+1)^2K_1K_3} \right)^{\frac{2}{d+2}}.
\end{align*}
Therefore, by \eqref{c45} we have that
\begin{align*}
	K_3^{d}((y-x-b(x))\cdot e)^{d}(h')^{-1}\left(\frac{K_3\big((y-x-b(x))\cdot e\big)^2}{2\epsilon^2} \right)\leq C.
\end{align*}
Consequently, if we define $\tau(\epsilon)$ to be the solution of
\begin{align*}
	K_3^d (\tau(\epsilon))^{d}(h')^{-1}\left(\frac{K_3(\tau(\epsilon))^2}{2\epsilon^2} \right)=C,
\end{align*}
we obtain by monotonicity that
\begin{align*}
	\langle y-x-b(x),e\rangle\leq \tau(\epsilon),
\end{align*}
which finishes the proof.

\step 3 Proof of (ii). Let $(x,y)\in\supp\pi\cap (Q_R\cap \Omega_0)\times\R^d$ be given. We first observe that we may assume  $E\leq1$. Indeed, if $E\ge1$ the proof of \cite[Theorem 6]{GvalaniKoch1} applies without change: for $r<\frac 1 3$ and $x\in \Omega_0\cap Q_R$ satisfying \eqref{c48},
\begin{align*}
B_{rR}\left(x+2rR\frac{y-x-b(x)}{|y-x-b(x)|}\right)\subset\Omega_0\cap B_{2R},
\end{align*}
since  by \eqref{c48}, $x$ is of distance at least $RE^{1/2}\ge R$ away from $\p\Omega_0$. By exactly the same proof as \cite[large transport regime in Theorem 6]{GvalaniKoch1}, it follows that
\begin{align*}
\frac{|y-x-b(x)|}{R}\lesssim E^{\frac{1}{2}}\quad\text{whenever}\quad E\geq 1.
\end{align*}

We may therefore assume $E\leq 1$ in the following. We also set
\begin{align*}
\delta:=RL_1+R\,\max\left(E^{\frac1{d+2}},\frac{\tau(\epsilon)}{R}\right)\mbox{ so that assumption \eqref{c48} is }\langle x,e_1\rangle\ge\delta.
\end{align*}
 For $r,\bar r$ to be chosen, define 
\begin{align*}
B_{rR}&:=B_{rR}\left(x+2rR\frac{y-x-b(x)}{|y-x-b(x)|}\right),\\
B(x,y)&:=\{(x',y')\in B\times \R^d\mid |y'-x'-b(x')|\leq\bar rR \}.
\end{align*}
We shall always chose $r$ such that $r<\min\{\frac{1}{3},\frac{1}{3}\max(E^{\frac{1}{d+2}},\frac{\tau}{R})\}$ so that $B_{rR}\subset Q_{2R}\cap \Omega_0$.

The computations of \cite[Step 2, Theorem 6]{GvalaniKoch1} give
\begin{align}
	&\Delta\le -rR|y-x-b(x)|\quad \forall\,(x',y')\in B(x,y)\label{c49}
    \end{align}
    as long as $|y-x-b(x)|\ge 18(L_1+1)rR+6\bar rR$. Further,
    \begin{align}
    \bar r=\Bigl(\frac{2\cdot 3^{\,d+2}}{D\,\omega_d}\Bigr)^{1/2}r^{-d/2}E^{1/2} 	\quad\Rightarrow\quad \pi\left(B(x,y)\right)\gtrsim(rR)^d. \label{c50}
\end{align}
As before, combining \eqref{c44} and \eqref{c43}, for any $r,\bar r$ and $(x,y)$ satisfying the above two constraints, we have
\begin{equation}\label{c51}
	(rR)^d(h')^{-1}\left(\frac{rR|y-x-b(x)|}{2\epsilon^2}\right) \le C 
\end{equation}
for some $C<\infty$.
Also let $P:=\dfrac{1}{54(L_1+1)}$.

\substep 1 The regime $E^\frac 1 {d+2}\geq \frac{\tau(\epsilon)}{R}$. In this regime, we show that
\begin{align*}
|y-x-b(x)|\leq (1+K_5)RE^{\frac{1}{d+2}}\quad\text{where}\quad K_5:=6\Bigl(\dfrac{2\cdot3^{d+2}}{D\omega_d}\Bigr)^{1/2}\bigl(54(L_1+1)\bigr)^{d/2}.
\end{align*}

Assume for contradiction that the above is not true, then we choose $r,\bar r$ in \eqref{c50} to be
\begin{align*}
r=\frac{E^{\frac1{d+2}}}{54(L_1+1)},\quad r=\frac{K_5}{6}E^{\frac1{d+2}}.
\end{align*}
This choice as well the contradicting assumption implies that
\begin{align*}
|y-x-b(x)|\geq (1+K_5)RE^{\frac{1}{d+2}}\geq 18(L_1+1)rR+6\bar rR,
\end{align*}
which means that \eqref{c51} holds. Noting that $rR=PE^{\frac 1{d+2}}R$ combined with the assumption that $|y-x-b(x)|\geq (1+K_5)E^{\frac{1}{d+2}}R$ allows us to conclude
\begin{align*}
(PRE^{\frac1{d+2}})^d
(h')^{-1}\left(\frac{P(1+K)(RE^{\frac1{d+2}})^2}{2\epsilon^2}\right)\leq C.
\end{align*}
Since we are in the regime $RE^{\frac1{d+2}}\ge\tau(\epsilon)$ and
$z\mapsto (Pz)^d(h')^{-1}\!\bigl(\tfrac{P(1+K)z^2}{2\eps^2}\bigr)$ is increasing, we find that
\begin{align*}
(P\tau(\epsilon))^d\,
(h')^{-1}\left(\frac{P(1+K)\tau(\epsilon)^2}{2\epsilon^2}\right) \le C .
\end{align*}
Now we define  $\tau(\epsilon)$ to be the unique solution of
\begin{align*}
(P\tau(\epsilon))^d\,
(h')^{-1}\left(\frac{P(1+K)\tau(\epsilon)^2}{2\epsilon^2}\right)=2C.
\end{align*}
This provides the desired contradiction.

\substep 2. The regime  $\frac{\tau(\epsilon)}{R}\geq E^{\frac1{d+2}}$. In this regime, we claim that
\begin{align*}
|y-x-b(x)|\leq (1+K_6)\tau(\epsilon)\quad\text{for some}\quad K_6=K_6(d,D,L_1)
\end{align*}
Indeed, the proof argues via contradiction and can be obtained by verbatim replacing $E^{\frac{1}{d+2}}$ by $\frac{\tau(\epsilon)}{R}$ in the proof of Step 1. We omit the details.

Finally, we remark that in the case $h=h_p$, explicit calculations easily show that one may take $R_c\simeq R^\frac 2 {d(p-1)+2}$.
\end{proof}

In order to obtain a full $L^\infty$-bound from Proposition \ref{prop:LinftyPartial}, we need to impose a variant of the topological condition Assumption \ref{ass:topological} in $Q_{2R}$. We assume that for some $C>0$,
		\begin{align}\label{eq:topological2}
			(x,y)\in (\Omega_0\cap Q_R)\times\R^d\cap \supp\pi\quad\Rightarrow\quad |y-x-b(x)|\leq CR \quad\text{ and }\quad \langle b(x),e_1\rangle \leq CL_1.
	\end{align}
\begin{corollary}\label{cor:Linfty}
Suppose the assumptions of Proposition \ref{prop:LinftyPartial} hold. Assume in addition \eqref{eq:topological2} applies. If $R_c<R$ and Assumption \ref{ass:tangencyLocal} applies in $Q_{2R}$, then
\begin{align*}
    \sup_{(x,y)\in ((\Omega_0\cap Q_R)\times \R^d)\cap\supp\,\pi_\e}\frac{|y-x-b(x)|}{R}\lesssim \max\left(E^\frac 1 2, E^\frac 1 {d+2}, \tau(\e)/R\right)+L_1,
\end{align*}
where the implicit constant depends on $d, D$, $C,L_1$  and $Lip(b,B_{2R})$.
\end{corollary}
\begin{proof}
Fix $(x,y)\in ((\Omega_0\cap Q_R)\times \R^d)\cap \supp\,\pi_\e$. If \eqref{c48} holds, we may use (ii) in Proposition \ref{prop:LinftyPartial} to conclude. Hence, we may assume
\begin{align*}
\langle x,e_1\rangle \leq R L_1 + R \max\left(E^\frac 1 2, E^\frac 1 {d+2}, \tau(\e)/R\right).
\end{align*}
Note that \eqref{eq:topological2} implies that $\langle y,e_1\rangle \geq -L_1 R$. Indeed, by \eqref{eq:topological2}, $y\in Q_{2R}$ and hence by Assumption \ref{ass:tangencyLocal} $\langle y,e_1\rangle \geq - L_1 R$. Consequently,
\begin{align}\label{conclusion1}
\langle y-x-b(x),-e_1\rangle\leq (2+C) R L_1 + R\max\left(E^\frac 1 2, E^\frac 1 {d+2}, \tau(\e)/R\right).
\end{align}
On the other hand, by (i) in Proposition \ref{prop:LinftyPartial} for any $e\in \partial B_1$ with $\langle e,e_1\rangle \geq \theta$,
\begin{align}\label{conclusion2}
    \langle y-x-b(x),e\rangle \lesssim \max\left(E^\frac 1 2, E^\frac 1 {d+2},\tau(\e)/R\right).
\end{align}
Combining \eqref{conclusion1} and \eqref{conclusion2} completes the proof.
\end{proof}

\subsection{Regularity estimates}
We note that with the $L^\infty$-bounds of Section \ref{sec:Linfty} at hand, the proof of interior regularity \cite[Proposition 16]{GvalaniKoch1} transfers almost without change to the boundary setting.
\begin{proposition}\label{prop:smallScale}
Assume $\lambda,\mu$ are probability measures with Lipschitz supports $\supp\,\lambda$, $\supp\,\mu$, respectively. Let $p\in [1,2]$. Suppose $\pi$ is a minimiser of $OT_{\e,h_p}$. There is $\e_1>0$ such that if $E(\pi,C_1 R_c)+\e+C_1^{-1}\leq \e_1$ for some $C_1>0$, then there is $C_2$, independent of $\e$ such that
\begin{align*}
 \|\nabla^2 f_\e\|_{L^\infty(Q_{C_1 R_c/3})}+\|\nabla^2 g_\e\|_{L^\infty(Q_{C_1 R_c/3})}\leq C_2.
\end{align*}
If $p<2$, then $f_\e,g_\e\in C^2(Q_{C_1 R_c/3})$.
\end{proposition}
\begin{proof}
By symmetry, it suffices to prove the claim for $f_\e$. 
Using the $C^1$-regularity and convexity of potentials from Section \eqref{sec:basicPropertiesPotentials} and the Lipschitz regularity of $\supp\,\lambda$ and $\supp\,\mu$, we may apply Reynold's transport theorem, c.f. \cite[Theorem A.1]{GonzalezGvalaniKoch}, to find for $\lambda$-a.e. $x$,
\begin{align*}
\nabla^2 f_\e(x) = 2I-\begin{cases}(q-2) \frac{\int_{\mathcal S_x}  \nabla_x\xi(x,y)\otimes \nabla_x \xi(x,y) \xi(x,y)^{q-3}\dd \mu(y)}{\int_{\mathcal S_x} \xi(x,y)^{q-2}\dd \mu(y)}\quad& \text{ for } 1<p <2 \\
\frac{\int_{\partial \mathcal S_x\setminus \partial \Omega_1} \frac{\nabla_x\xi(x,y)\otimes \nabla_x \xi(x,y)}{|\nabla g_\e(y)+2x-2y|}\mu \dd \mathscr H^{d-1}(y)}{\mu(\mathcal S_x)} \quad& \text{ if } p=2
\end{cases}
\end{align*}
with $\frac 1 p+\frac 1 q=1$,.

Moreover, from Section \ref{sec:Linfty}, it follows that
\begin{align}
  \sup_{x\in Q_{C_1 R_c/3},y\in \mathcal S_x, \hat x\in \mathcal T_y}  |y-\hat x|\leq \tau(C_1)R_c:=L_\infty,
\end{align}
where $\tau(C_1)\to 0$ as $C_1\to\infty$.

Fix $x\in Q_{C_1 R_c/3}$ and denote by $\xi_m$ the maximum of $\xi$ in $\mathcal S_x$, attained at $y_m$. By the interior cone condition, there is a cone $C(\alpha,\rho)$ with vertex $y_m$, height $\rho$ and aperture $\alpha$ such that $C(\alpha,\rho)\subset \supp\,\mu$. 
Then, 
\begin{align*}
    \int_{\mathcal S_x} (\xi_m-C L_\infty |y-y_m|)_+^{q-2}d y\geq &\int_{C(\alpha,\rho} (\xi_m-C L_\infty |y-y_m|)_+^{q-2} d y\\
    \gtrsim&_{\alpha,\rho} \int (\xi_m - C L_\infty t)_+^{q-2} t^{d-1} d t\gtrsim \frac{\xi_m^{q-2+d}}{L_\infty^d}.
\end{align*}

With these facts at hand, the proof of \cite[Proposition 16]{GvalaniKoch1} can be carried out ad verbatim.
\end{proof}

\subsection{Boundary control of long trajectories for ROT: \texorpdfstring{$p=1$}{}}
In this section, we consider entropic transport, that is $OT_{\e,h_1}$. We begin with a general control on the mass of very long trajectories for quasi-minimisers of quadratic optimal transport. 
\begin{lemma}\label{lem:mediumTrajectories}
Assume $\lambda,\mu\in \mathscr M(\R^d)$ admit $C^{\alpha}$-densities and $\supp\,\lambda$ and $\supp\,\mu$ are Lipschitz. Suppose $\pi\in \Pi(\lambda,\mu)$ satisfies the following: there is $\Lambda>0$ such that for any $\tilde \pi =\hat \pi + \restr{\pi}{\#_R^c}\in \Pi(\lambda,\mu)$,
\begin{align*}
\int_{\#_R}|x-y|^2\dd\pi-\int |x-y|^2 \dd \hat \pi\leq \Lambda R^{d+2}.
\end{align*}
Then there are $R_0, C,\e_1>0$ such that if $R\leq R_0$, $R^{d+2}E(\pi,R)\ll \rho^{d+2}\ll R^{d+2}$ and $D(\rho,\lambda,\mu)\leq \e_1$, then
\begin{align*}
    \pi(\{(x,y)\in \#_{R/2}\cap \supp\,\pi\colon |x-y|\geq \rho\})\lesssim \frac{\Lambda R^{d+2}}{\rho^{d+2}} R^d.
\end{align*}
\end{lemma}
\begin{proof}
The proof essentially follows \cite[Lemma 4]{GvalaniKoch} with some additional care needed at the initial covering step, which we now detail. 

As the statement is symmetric under exchanging the roles of $x,y$, it suffices to show
\begin{align*}
\pi(\{(x,y)\in (Q_{R/2}\cap \supp\,\lambda)\times \R^d\colon |x-y|\geq \rho\})\lesssim \frac{\Lambda R^{d+2}}{\rho^{d+2}}R^d.
\end{align*}
Choosing $R_0$ sufficiently small, we may cover $Q_{R/2}\cap \supp\,\lambda$ by $O(R^d r^{-d})$ balls of radius $r$, where $r$ is a small fraction of $\rho$. Additionally, we may ensure that for each such ball $B$, $|\supp\,\lambda \cap B|\geq \frac 1 4 |B|$. Thus, it suffices to show that for a ball $B$ of radius $r$ with $|B\cap \Omega|\geq \frac 1 4 |B|$,
\begin{align*}
    r^2 \pi(B_r\times B_\rho^c)\lesssim \Lambda R^{d+2}.
\end{align*}
Covering the unit sphere by geodesic balls, it suffices to show that for any direction $e$,
\begin{align*}
r^2 \pi(B_R\times C_{\alpha,e})\lesssim \Lambda R^{d+2},
\end{align*}
where $C_{\rho,e}$ is the intersection of $B_\rho^c$ with the cone of aperture $\alpha$ in direction $e$. If $m=\pi(B_r\times C_{\alpha,e})=0$, the statement is obvious. Otherwise, the proof can now be carried out verbatim as in \cite[Lemma 4]{GvalaniKoch}.
\end{proof}

We further require the following analogue of \cite[Proposition 14]{GvalaniKoch1}.

\begin{lemma}\label{lem:longTrajectories}
Let $\pi$ be a minimiser of $OT_{\e,h_1}$. Assume $\Omega_0=\supp\,\lambda$ and $\Omega_1=\supp\,\mu$ are Lipschitz domains. Let $T>0$ and assume $L_1:=\mbox{Lip}(g,Q_{(T+1)R})+\mbox{Lip}(b,B_{(T+1)R})<\frac 1 4$. There are $C,C_0>0$ such that if $R\geq C \e$ and  and Assumption \ref{ass:tangencyLocal} holds in $Q_{(T+1)R}$., then for almost every $x\in Q_R$,
\begin{align*}
    \frac 1 {R^2}\int_{\{T R\geq |x-y|\geq C_0 R\}} |x-y|^2 \frac{\dd\pi}{\dd P}(x,y)\dd\mu(y)\lesssim e^{-\frac{R^2}{\e^2}}(1+E(\pi_\e,2R))+ L_1 R.
\end{align*}
\end{lemma}
\begin{proof}
Recall that $\Delta = |x-y'|^2+|x'-y|^2-|x-y|^2-|x'-y'|^2$.
Set for $C_1>0$ to be determined,
\begin{align*}
    A(x,y):=&\{(x',y')\in \#_R\colon -\Delta \geq R^2,\\
&\quad|x-x'|\leq 2R,\,|x'-y'|^2 \leq R^2 C_1\max( E(\pi,(T+1)R),1).
\end{align*}
Arguing as in \cite[Proposition 14]{GvalaniKoch1}, we find for almost every $x\in \supp\,\lambda$,
\begin{align*}
\int \int_{A(x,y)} 1_{\{T R\geq |x-y|\geq C_0 R\}} |x-y|^2 \frac{\dd \pi}{\dd P}(x,y)\dd \pi(x',y')\dd\mu(y)\lesssim R^{d+2} e^{-\frac{R^2}{\e^2}}(1+E_{\pi,2R}).
\end{align*}

If $\langle x,e_1\rangle \geq R L_1$, then $B_{RL_1/2}(x+R L_1/2 e)\subset \supp\,\lambda$ and by the interior argument from \cite[Proposition 7]{GvalaniKoch}, $\pi(A(x,y))\gtrsim R^d$. This gives the desired estimate. 

Hence, we assume $\langle x,e_1\rangle \leq R L_1$. As $|y-x|\leq T R$, $\langle y,e_1\rangle \geq - L_1 R$ and we find
\begin{align}\label{eq:geometryInsight}
    \langle y-x,e_1\rangle\geq - 2L_1 R.
\end{align}
As in \cite[Proposition 7]{GvalaniKoch}, for any sufficiently large constant $C_1>0$,
\begin{align*}
A(x,y)\supset& \{(x',y')\colon x'\in C_\alpha(x,y)\cap (B_{R}(x)\setminus B_{R/2}(x))\\
&\quad\text{ and } |x'-y'|^2\leq C_1 R^2\max(E(\pi,2R),1)\}=:B(x,y),
\end{align*}
where $C_\alpha(x,y)$ is the cone in direction $y-x$ with vertex $x$ and aperture $\alpha$.
In light of \eqref{eq:geometryInsight}, for $x'\in C_{\alpha}(x,y)\cap (B_{R}(x)\setminus B_{R/2}(x))$ with $\langle x'-|x'|\frac{y-x}{|y-x|},e_1\rangle \geq |x'|/2$,
\begin{align*}
\langle x',e_1\rangle = \langle x'-|x'|\frac{y-x}{|y-x|},e_1\rangle + |x'|\langle \frac{y-x}{|y-x|},e_1\rangle \geq -c \alpha |x'|+|x'|/2.
\end{align*}
Choosing $\alpha$ sufficiently small, so that $-c\alpha|x'|+|x'|/2\geq L_1  |x'|$, it follows that 
\begin{align*}
\pi((C_\alpha(x,y)\cap (B_{R}(x)\setminus B_{R/2}(x)))\cap \R^d)= \lambda(C_\alpha(x,y)\cap (B_{R}(x)\setminus B_{R/2}(x)))\gtrsim R^d.
\end{align*}
Moreover,
\begin{align*}
&\pi\left(((C_\alpha(x,y)\cap (B_{R}(x)\setminus B_{R/2}(x)))\times \R^d)\cap \left\{|x'-y'|^2\geq \frac{1}{C_1 R^2\max(E(\pi,2R),1)}\right\}\right)\\
\leq& \frac 1 {C_1 R^2 \max(E(\pi,2R),1)} \int_{B(x,y)} |x'-y'|^2 \dd \pi\leq \frac{R^d}{C_1}.
\end{align*}
For a sufficiently large choice of $C_1$, we conclude that for some $c>0$ and a sufficiently large choice of $C_1$,
\begin{align*}
\pi(A(x,y))\geq \pi(B(x,y))\geq c R^d-\frac{R^d}{C_1}\geq\frac{c} 2 R^d.
\end{align*}
This implies the desired estimate.
\end{proof}

We finally combine Lemma \ref{lem:mediumTrajectories} and Lemma \ref{lem:longTrajectories} into the following form.
\begin{corollary}\label{eq:crossingEntropic}
Let $\pi$ be a minimiser of $OT_{\e,h_1}. $Assume $\lambda,\mu$ admit $C^{\alpha}$-densities and further $\supp\,\lambda$ and $\supp\,\mu$ are $C^{1,\alpha}$-domains. Let $T>1$. There is $C>0$ such that if $R\geq C\e$ and Assumption \ref{ass:tangencyLocal} holds in $Q_{(T+1)R}$ and $(TR)^{-(d+2)}\int_{\#_R\cap \{|x-y|\geq TR\}} |x-y|^2 \dd \pi\leq \frac{\e^2}{T^2 R^2}+D((T+1)R,\lambda,\mu)+D_{bdy}((T+1)R))$, Then for any $\tau>0$, there are $\e_1,C_\tau>0$ such that if $E(\pi,2R)\leq \e_1$ and $D_{bdy}((T+1)R)\leq \frac 1 4$, then \[\frac 4 {R^{d+3}}\int_{R/2}^{3R/4}\int_{\partial(\#_r)}\lvert x-y\rvert^2\;d\pi dr \leq \tau E(\pi,R)+C_\tau \left(D_{bdy}((T+1)R)+\frac{\e^2}{R^2}\right)\] and \[\frac 4 R\int_{\frac R 2}^{\frac{3R} 4}\pi(\#_r) dr\leq \tau.\]
\end{corollary}
\begin{proof}
For $X(t)=x(1-t)+ty$, note 
	\begin{equation}\label{crossing trajectories}
		\begin{aligned}
	       \int_{R/2}^{3R/4}\int_{\partial \#_r}\lvert x-y\rvert^2\;d\pi d r &= \int_{\#_{R}}\lvert x-y\rvert^2\int_{R/2}^{3R/4}\mathbf{1}\left\{\exists t: X(t)\in\partial Q_R\right\} dr\;d\pi \\
           &= \int_{\#_{R}}\lvert x-y\rvert^2 L(x,y)d\pi
          \end{aligned}
    \end{equation} where 
    \begin{align*}
        L(x,y) \leq \min\left\{R, \max_t\lvert X(t)\rvert-\min_t\lvert X(t)\rvert\right\} \leq \min\left\{R, \lvert x-y\rvert\right\}.
        \end{align*}

	Using Lemma \ref{lem:mediumTrajectories} and Lemma \ref{lem:longTrajectories}, \eqref{crossing trajectories} is therefore bounded by 
	\begin{align*}
    &R\int_{(x,y)\in \#_R\colon |x-y|\geq TR} |x-y|^2d\pi+
    	R\int_{(x,y)\in\#_{R}: T R\geq \lvert x-y\rvert\geq C_0 R}\lvert x-y \rvert^2 d\pi \\
        &+R\int_{(x,y)\in \#_{R}: \rho\leq \lvert x-y\rvert\leq C_0 R}\lvert x-y\rvert^2 d\pi+\rho \int_{\#_R} |x-y|^2\\
    	\leq& \e^2(T R)^{d}+(TR)^{d+2}D_{bdy}((T+1)R))+R^{d+3}e^{\frac{R^2}{\varepsilon^2}}+ C C_0^2 R^{d+3}\frac{\varepsilon^2 R^{d}}{\rho^{d+2}}+\rho R^{d+2} E(\pi,R).
	\end{align*}
    Choosing $\rho=\tau R$ for $\tau$ sufficiently small, the first estimate holds. A similar argument using Markov's inequality gives the second estimate.
\end{proof}

\subsection{Local quasiminimality of regularised optimal transport}
In this section, we recall the local quasiminimality properties of regularised optimal transport. 
\begin{lemma}\label{lem:localQuasiminimality}
Let $p\in (1,2]$. Assume $\lambda,\mu$ are measures with Lipschitz supports $\Omega_0,\Omega_1$. Suppose $\pi$ is a minimiser of $OT_{\e,h_p}$ satisfying Assumption 1.4. Then for any $\tau>0$,
\begin{align}\label{eq:localQuasiminimality}
E_p(R,\pi)\leq \frac 1 {R^{d+2}} W_2^2(\bar \lambda,\bar \mu)+\tau E_p(2R,\lambda,\mu)+C_\tau \frac{R_c^2}{R^2}.
\end{align}
If $p=1$ and $\pi$ is a minimiser of $OT_{\e,h_1}$ such that for some $T<1$, Assumption \ref{ass:tangencyLocal} holds in $Q_{2(T+1)R}$ and $(2TR)^{-(d+2)} \int_{\#_{2R}\cap \{|x-y|\geq TR\}}|x-y|^2 \dd \pi \leq \frac{\e^2}{T^2 R^2}+D((T+1)R,\lambda,\mu)+D_{bdy}((T+1)R)$, then \eqref{eq:localQuasiminimality} holds with an additional term $D_{bdy}((T+1)R)$ on the right-hand side.
\end{lemma}
\begin{proof}
The proof of the interior versions \cite[Lemma 3.8]{GvalaniKoch1} in the case $p=1$ and \cite[Section 2.1]{GvalaniKoch} can be transferred verbatim to the boundary case up to replacing $L^\infty$-estimates and control of long trajectories by their respective boundary versions from Section \ref{sec:Linfty}.
\end{proof}

\section{Boundary harmonic approximation}\label{sec:harmonixApproximation}
The aim of this section is to prove a boundary harmonic approximation result for almost-minimisers of optimal transport. Our strategy adapts the interior proof in \cite{KochOtto} with extra care regarding a number of subtleties arising from the presence of the boundary. 

Throughout this section we use the following set-up: Let $\lambda,\mu\in \mathcal M(\R^d)$ and set $\Omega_0:=\supp \,\lambda$, $ \Omega_1:=\supp\, \mu $. Fix $T\geq 1$.
For $R$ fixed, we assume the following:
\begin{itemize}
	\item[(A)] Assumption \ref{ass:tangencyLocal} holds in $Q_{TR}$.
\item[(B)] We assume either (B1) or (B2):

(B1) Average control of crossing trajectories and regular marginal: For any $\tau>0$, there are $\e_1,C_\tau>0$ such that if $E(\pi,R)\leq \e_1$, then
\begin{align}
\frac{4}{R}\int_{\frac{R}{2}}^{\frac{3R}{4}}\int_{\p(\#_r)}|x-y|^2d\pi dr=\tau E(\pi,R)+C_\tau\left(D(TR)+D_{bdy}(TR)\right),\label{c59}\\
\frac{4}{R}\int_{\frac{R}{2}}^{\frac{3R}{4}}\pi(\p(\#_r))dr\leq \tau\label{c60}.
\end{align} 

 (B2) Boundary $L^\infty$ bound: For any $\tau>0$, there is $\e_1>0$ such that if $E(\pi,R)+D(R,\lambda,\mu)\leq \e_1$, 
\begin{align}\label{c63}
	\sup_{(x,y)\in \#_{R/2}\cap\supp \pi}\frac{|y-x|}{R}\leq \tau.
\end{align} 

\item[(C)]Local almost minimality: There exists $\Lambda>0$ such that 
	\begin{align*}
		\int_{\#_{R}}|y-x|^2d\pi\leq W_2^2(\bar\lambda,\bar\mu)+\Lambda R^{d+2},
	\end{align*}
	where $\bar\lambda,\bar\mu$ are the two marginals of $\pi\llcorner\#_R$.
\end{itemize}

 \begin{remark}\label{rk:1}
The boundary $L^\infty$ bound in (B2) implies \eqref{c59} and \eqref{c60} in (B1) with $T=1$.
\end{remark}
\medskip

We further define for $r\leq 4R$,
\begin{align}\label{c2}
	D_{bdy}(r):=\sup_{Q_r'}|g_0|^2+\sup_{Q_r'}|g_1|^2,\quad\tilde Q_r:=Q_r\cap \{x_1\geq -D_{bdy}(r)^{\frac 1 2} \}.
\end{align}
Note that under Assumption \ref{ass:tangencyLocal}, we have $ D_{bdy}(r)^\frac 1 2\lesssim r\left(Lip(g_0,Q_r')+Lip(g_1,Q_r')\right):=D_{bdy}'(r)$.

Finally, we define a mollification on the boundary $\partial Q$ of a rectangle $Q$ as follows. Let $f\in L^1(\p  Q)$. Fix a face of $F$ of $\p Q$, which we identify as a rectangle in $\R^{d-1}$. Extend $\restr{f}{F}$ periodically by even reflection to $\R^{d-1}$ and set
\begin{align}\label{eq:mollification}
f^r:= \restr{f*\phi_r}{F}
\end{align}
where $\phi_r$ is an even mollifier in $\R^{d-1}$.

Up to the change of variable
\begin{align*}
	\frac{R}{4}\hat x= x,\quad \frac{R}{4}\hat y=y
\end{align*}
we shall set $R=4$. 
\begin{prop}\label{P:1}
	Given any $ \tau>0 $, under assumption (A)-(C), there exists $ \epsilon(d,\tau)>0 $ and $ C(d,\tau)<\infty $ such that the following holds. If 
	\begin{align*}
		E(\pi,4)+D(4)+D_{bdy}'(4)\leq \epsilon,
	\end{align*}
	then there exists $ R\in(2,3) $ and a  harmonic gradient $ \nabla\phi $ defined on $ \tilde Q_R $ such that
	\begin{align*}
		\int_{\#_1}|y-x-\nabla\phi(x)|^2d\pi&\leq\tau E(\pi,4)+C(D(4)+D_{bdy}'(4T)+\Lambda),\\
		\int_{\tilde Q_R}|\nabla\phi|^2&\leq C(E(\pi,4)+D(4)+D_{bdy}'(4T)).
	\end{align*}
    If (B2) holds, it is possible to set $T=1$.
Furthermore,
\begin{align}\label{e:phi-C2alpha}
\|\phi\|^2_{C^{3}(\tilde Q_1)}\leq C \left( E(\pi,4)+D(4)+[\nabla g_0]^2_{\alpha,Q_4'}+ [\nabla g_1]^2_{\alpha,Q_4'}\right).
\end{align}
More precisely, $ \phi $ solves a Poisson equation of the following form
	\begin{align*}
		-\Delta\phi=c\text{ in }\tilde Q_R,\quad \nabla\phi\cdot\nu=h\text{ on }\p\tilde Q_R,
	\end{align*}
	where $ c $ is a constant and $ h $ is a piecewise smooth function satisfying the following: 
	\begin{align*}
		&h|_{\p \tilde Q_R\cap \{x_1>-D_{bdy}^{\frac 1 2}\}}\mbox{ is smooth wherever $\nu$ is smooth;}\\
		&\nabla' h|_{\p\tilde Q_R\cap \{x_1=-D_{bdy}^{\frac 1 2}\}} \mbox{ is smooth and } h'(0)=0.
	\end{align*}
\end{prop}

Before we begin the proof we introduce some more notation. We denote
\begin{align*}
	\hat Q_R&:=Q_R\cap \{x_1\geq D_{bdy}^{\frac 1 2} \}.
\end{align*}
Further, let
\begin{align}
\hat\kappa_{\lambda,R}:=\frac{\kappa_{\lambda,R}|\hat Q_R|}{|\tilde Q_R|},\quad \hat\kappa_{\mu,R}:=\frac{\kappa_{\mu,R}|\hat Q_R|}{|\tilde Q_R|}.\label{c6}
\end{align}
Finally, we will use the short-hand notation
\begin{align*}
\Omega_{0,R}:=Q_R\cap \Omega_0,\quad \Omega_{1,R}:=Q_R\cap \Omega_1.
\end{align*}

\subsection{Choosing Data for the Poisson Equation}
In this subsection, we construct boundary data, which we will use to define $\phi$ in Proposition \ref{P:1}. The construction is analogous to that in \cite[Lemma 4]{KochOtto}.

For $(x,y)\in \supp\,\pi$, define the entering and exiting time of $X(t)$ as
\begin{align*}
    \sigma =& \sigma(x,y,R):= \min\{t\in [0,1]\colon X(t)\in \bar Q_R\},\\
    \tau =& \tau(x,y,R):= \max\{t\in [0,1]\colon X(t)\in \bar Q_R\}.
\end{align*}
We use $\sigma,\tau$ to define measures $f_R,g_R$ via
\begin{align}
\int \zeta d f_R := \int_{\#_R\cap \{X(\sigma)\in \partial Q_R\}} \zeta(X(\sigma))d\pi,\quad \int \zeta d g_R:= \int_{\#_R\cap \{X(\tau)\in \partial Q_R\}} \zeta(X(\tau))d\pi.
\end{align}
\begin{lemma}\label{L:2} Assume $D(4)\lesssim 1$. Under assumptions (A)-(C), for each $R\in(2,3)$, there exists a non-negative function $\bar g_R$ defined in $\p \tilde Q_R$ such that
\begin{align}
W_2^2(f_R,\bar f_R)+W_2^2(g_R,\bar g_R)&\lesssim \int_{\p(\#_R)}|y-x|^2d\pi+D(4),\label{c61}\\ \int_{2}^3\int_{\p\tilde Q_R}\bar f_R^2+\bar{g}_R^2dR&\lesssim E(\pi,4)+D(4).\label{c62}
\end{align}
\end{lemma}
\begin{proof}
By symmetry it suffices to prove the statements involving $g_R$ and to construct $\bar g_R$. Let $\bar\pi$  be such that
\begin{align*}
\bar\pi\mbox{ is optimal for }W_2(\mu\llcorner Q_4,\kappa_{\mu, 4}dx\llcorner\Omega_{1,4}).
\end{align*}
and then extend $\bar\pi$ to $\R^d\times\R^d$ by identity. By disintegration 
\begin{align*}
\bar \pi(dydz)=\bar\pi(dz|y)\mu(dy)\quad\mbox{where}\quad \{\bar\pi(dz|y)\}_{y\in Q_4} \mbox{are probability measures},
\end{align*}
so that we can define a measure on $\R^d\times\R^d\times\R^d$ by
\begin{align*}
\tilde\pi:=\bar\pi(dz|y)\pi(dxdy)
\end{align*}
which has the property that 
\begin{align*}
\mbox{marginal of $\tilde\pi$ w.~r.~t.~$(x,y)=\pi$},\\
\mbox{marginal of $\tilde\pi$ w.~r.~t.~$(y,z)=\bar\pi$}.
\end{align*}
From $\tilde\pi$, we define $g'_R$ which tracks the distribution of $z$ for which $(x,y,z)\in\supp\pi$ and $X(\tau)\in\p Q_R$, more precisely,
\begin{align*}
\int\zeta dg_R'=\int_{\#_R\cap \{X(\tau)\in Q_R\}\times \R^d}\zeta(z)\tilde\pi(dxdydz).
\end{align*}
Now let $e_1=(1,0,\ldots,0)\in Q_R$ and let $P_R:\R^d\mapsto\p Q_R$ be the projection of $\R^d$ onto $\p Q_R$ with respect to $e_1$. We define $\bar g_R$ via
\begin{align*}
\int\zeta d\bar g_R:=\int\zeta(P_R z)g'_R(dz).
\end{align*}

For the proof \eqref{c61}, one observes that the following is an admissible coupling of $W_2(g_R,\bar g_R)$, 
\begin{align*}
\int_{\#_R\cap\{X(\tau)\in\p Q_R\}}\zeta(X(\tau),P_R z)d\tilde\pi,
\end{align*}
so that
\begin{align*}
W_2^2(g_R,\bar g_R)\lesssim \int_{\#_R\cap\{X(\tau)\in\p Q_R\}}|x-y|^2d\pi+D\lesssim\int_{\p(\#_R)}|x-y|^2d\pi+D(4).
\end{align*}

\medskip

For the proof of \eqref{c62}, we claim the following estimates:
\begin{align*}
\int_{2}^{3}\int ||z|-R|dg_R'dR\lesssim E(\pi,4)+D(4),\\
\int_{\p Q_R}\bar g_R^2\lesssim \esssup g'_R\int||z|-R|g'_Rdz.\\
\esssup g'_R\lesssim 1
\end{align*}
Combining these estimates implies \eqref{c62}. The first estimate is \cite[(60)]{KochOtto}. If $g_R'$ has compact support, the second estimate is \cite[(62)]{KochOtto}.  Depending on whether (B1) or (B2) holds, we have 
\begin{align*}
	\text{either }\supp g_R'\subset \supp\, \mu \quad\text{or}\quad \supp\, g_R'\subset Q_4.
\end{align*}
Regarding the third estimate, if assumption (B1) is satisfied, then  for $\zeta \geq 0$,
\begin{align*}
\int\zeta dg_R'\leq \int\zeta(z)d\mu(z)\quad\text{ so that }g_R'\leq \|\mu\|_{L^\infty}.
\end{align*}
If (B2) is satisfied, then we have that $y\in B_4$ for any $y$ such that $X(\tau)\in \p Q_R$ for some $R\in(2,3) $. Since $\bar\pi(Q_4\times Q_4^c)=0$, we see that for $\zeta \geq 0$,
\begin{align*}
\int\zeta d g_R'\leq\int_{\{z\in Q_4\}}\zeta(z)\bar\pi(dydz)=\kappa_{\mu, 4}\int_{Q_4}\zeta(z)dz\text{ for }\zeta\geq 0,\quad \text{ so that } g_R'\lesssim \kappa_{\mu,4}.
\end{align*}
\end{proof}

\subsection{Choice of good radius}\label{sec:goodRadius}
In this section, we make a choice of $R\in (2,3)$ such that a number of crucial estimates hold.

Our goal is to find $R\in (2,3)$ with the property that for any $\tau>0$ there exists $\e_1>0$ such that if $E(\pi,4)+D(4)+D_{bdy}(4) \leq \e_1$, then the following estimates hold:

First of all, we use (B1) to ensure that
\begin{align}\label{eq:goodCrossing}
\int_{\p(\#_R)}|x-y|^2d\pi dr\leq& \tau(E(\pi,4)+C_\tau(D(4T)+D_{bdy}(4T))\\
\label{eq:goodCrossingMeasure}
\pi(\p(\#_R))\leq& \tau.
\end{align}

Integrating \eqref{c61} over $R\in (2,3)$, we find using (B1),
\begin{align*}
\int_2^3 W_2^2(f_R,\bar f_R)+W_2^2(g_R,\bar g_R)\lesssim& \int_2^3 \int_{\partial(\#_R)} |y-x|^2 d\pi + D(4)\\
\leq& \tau E(\pi,4)+D((4)+C_\tau(D(4)+D_{bdy}(4T)).
\end{align*}
Hence, in light of Lemma \ref{L:2}, we may ensure
\begin{align}
W_2^2(f_R,\bar f_R)+W_2^2(g_R,\bar g_R)\leq \tau E(\pi,4)+C_\tau(D(4)+D_{bdy}(4T))\label{eq:goodWasserstein}\\
    \int_{\partial \tilde Q_R} \bar f_R^2+\bar g_R^2 \lesssim E(\pi,4)+D(4)\label{eq:goodBoundaryData}.
\end{align}

We claim that we may furthermore ensure that 
\begin{align}
&W_2^2(\lambda \llcorner Q_R,\hat\kappa_{\lambda,R}\llcorner \tilde Q_R+\hat g_{1,R})+(\hat \kappa_{\lambda,R}-1)^2+ W_2^2(\mu \llcorner Q_R,\hat\kappa_{\mu,R}\llcorner \tilde Q_R+\hat g_{1,R})+(\hat \kappa_{\lambda,R}-1)^2 \nonumber\\
\lesssim& D(4)+D_{bdy}(4),\label{eq:goodData}\\
&W_2^2(\lambda\llcorner Q_R,\kappa_{\lambda,R}dx\llcorner\Omega_{0,R})+(\kappa_{\lambda,R}-1)^2+W_2^2(\mu\llcorner Q_R,\kappa_{\mu,R}dx\llcorner\Omega_{1,R})+(\kappa_{\mu,R}-1)^2\nonumber\\
\lesssim& D(4)+D_{bdy}(4).
\label{eq:goodData2}
\end{align}
where on $\p\tilde Q_R\cap\{x_1=-D_{bdy}^{\frac 1 2}\}$,
\begin{align}
	\hat g_{0,R}:=\kappa_{\lambda,R}(D_{bdy}^{\frac 1 2}-g_0(x')) dx'\llcorner Q_R\cap \{x_1=-D_{bdy}^{\frac 1 2} \}\label{hatg-def}\\
	\hat g_{1,R}:=\kappa_{\mu,R}(D_{bdy}^{\frac 1 2}-g_1(x')) dx'\llcorner Q_R\cap \{x_1=-D_{bdy}^{\frac 1 2} \}.\nonumber
\end{align}
\eqref{eq:goodData} is a consequence of the following estimate, which is the analog of \cite[Lemma 6]{KochOtto}.
\begin{lemma}\label{L:1}
Assume that $D(4)+D_{bdy}(4)\ll1$, then
\begin{align*}
&\int_{2}^{3}W_2^2(\lambda\llcorner Q_R,\hat\kappa_{\lambda,R}dx\llcorner\tilde Q_R+\hat g_{0,R})+(\hat\kappa_{\lambda,R}-1)^2dR\\
&\quad+\int_{2}^{3}W_2^2(\mu\llcorner Q_R,\hat\kappa_{\mu,R}dx\llcorner\tilde Q_R+\hat g_{1,R})+(\hat\kappa_{\mu,R}-1)^2dR\lesssim D(4)+D_{bdy}(4).
\end{align*}
\end{lemma}
\begin{proof}
By symmetry it suffices to prove the estimate for the terms involving $\lambda$. As $\hat\kappa_{\lambda, R}dx\llcorner\tilde Q_R$ is uniformly positive on $\tilde Q_R$ by definition, following line-by-line the argument in [KO, Corollary 2 and Lemma 6], we find
\begin{align*}
&\int_{2}^{3}W_2^2(\lambda\llcorner Q_R,\hat\kappa_{\lambda,R}dx\llcorner\tilde Q_R+\hat g_{0,R})+(\hat\kappa_{\lambda,R}-1)^2dR\\
\lesssim& W_2^2(\lambda \llcorner \Omega_{0,4}\hat \kappa_{\lambda,4}dx \llcorner \tilde Q_4+\hat g_{0,4})+(\hat \kappa_{\lambda,4}-1)^2.
\end{align*}

We now first note that 
\begin{align*}
	|\hat\kappa_{\lambda,4}-\kappa_{\lambda,4}|^2=\kappa_{\lambda,4}^2\frac{||\hat Q_4|-|\tilde Q_4||^2}{|\tilde Q_4|^2}\lesssim D_{bdy}(4).
\end{align*}
Next, we estimate by triangle inequality,
\begin{align*}
&W_2(\lambda\llcorner\Omega_{0,4},\hat\kappa_{\lambda,4}dx\llcorner\tilde Q_4+\hat g_{0,4})\\
&\quad\leq W_2(\lambda\llcorner\Omega_{0,4},\kappa_{\lambda,4}\llcorner\Omega_{0,4})+W_2(\kappa_{\lambda,4}dx\llcorner\Omega_{0,4},\hat\kappa_{\lambda,4}dx\llcorner\tilde Q_R+\hat g_{0,4})\\
&\quad\leq D^{\frac{1}{2}}+W_2(\kappa_{\lambda,4}dx\llcorner\hat Q_4,\hat\kappa_{\lambda,4}dx\llcorner \tilde Q_4)+W_2(\kappa_{\lambda,4}dx\llcorner \Omega_{0,4}\backslash\hat Q_4,\hat g_{0,4}).
\end{align*}
Note that the map
\begin{align*}
	T(x_1,x')=\left(4+\frac{4+D_{bdy}^\frac 12}{4-D_{bdy}^\frac 1 2}(x_1-4),x' \right),
\end{align*}
satisfies $T_{\# }\kappa_{\lambda,4} \dx \llcorner \hat Q_4 = \hat \kappa_{\lambda,4}\dx \llcorner \tilde Q_4$, so that
\begin{align*}
	&W_2^2(\kappa_{\lambda,4}dx\llcorner\hat Q_4,\hat\kappa_{\lambda,4}dx\llcorner\tilde Q_4)&\leq \int_{\hat Q_4}|x-T(x)|^2\kappa_{\lambda,4}dx\\
    \lesssim& (1-\frac{4+D_{bdy}^\frac 1 2}{4-D_{bdy}^\frac 1 2})^2\overset{D_{bdy}(4)\ll 1}{\lesssim} D_{bdy}(4).
\end{align*}
Further, $\tilde T(x_1,x')=(-D_{bdy}^\frac 1 2,x')$ satisfies $\tilde T_{\# \kappa_{\lambda,4} dx \llcorner \Omega_{0,4}\setminus\hat Q_4} = \hat g_{0,4}$, so that
\begin{align*}
W_2^2(\kappa_{\lambda,4}dx\llcorner\Omega_{0,4}\setminus\hat Q_4,\hat g_{0,4})\leq& \int_{\Omega_{0,4}\setminus\hat Q_4}|x-\tilde T(x)|^2\kappa_{\lambda,4}dx\\
\lesssim& D_{bdy}(4) |\Omega_{0,4}\setminus \hat Q_4|\lesssim D_{bdy}(4)^\frac 3 2.
\end{align*}
Collecting estimates finishes the proof.
\end{proof}
We  remark that the proof of the second inequality in Lemma \ref{L:1} shows in the data term $D(R)$, we may measure the flatness of $\lambda\llcorner Q_R$ using either $\kappa_{\lambda, R}dx\llcorner\Omega_{0,R}$ or $\hat\kappa_{\lambda,R}dx\llcorner\tilde Q_R+\hat g_{0,R}$ up to a $O(D_{bdy}(4))$ error. More precisely, the above lemma also yields that provided $D+D_{bdy}(4)\ll 1$,
\begin{align}\label{c19}
\int_{2}^{3}W_2^2(\lambda\llcorner Q_R,\kappa_{\lambda,R}dx\llcorner\Omega_{0,R})+(\kappa_{\lambda,R}-1)^2dR\lesssim D(4)+D_{bdy}(4),
\end{align}
justifying \eqref{eq:goodData2}.

For later use, we explicitly record the following fact contained in the above proof.
\begin{corollary}\label{cor:changingCubes}
Assume Assumption \ref{ass:tangencyLocal} holds in $Q_R$. If $D_{bdy}(R)\ll 1$, then there is $C>0$ such that
\begin{align*}
D(R)\leq W_2^2(\lambda \llcorner \tilde Q_R+\hat f_R,\hat \kappa_{\lambda,R} dx \llcorner \tilde Q_R+\hat g_{0,R})+W_2^2(\mu \llcorner Q_R+\hat g_R,\hat \kappa_{\mu,R} dx \llcorner \tilde Q_R+\hat g_{1,R}),
\end{align*}
where for $(D_{bdy}^\frac 1 2,x)\in Q_R$,
\begin{align}
\hat f_R(x) = \int_{\supp\,\lambda\cap \{x'<D_{bdy}^\frac 1 2\}} f(x',x) dx',\quad \hat g_R(x) = \int_{\supp\,\mu\cap \{x'<D_{bdy}^\frac 1 2\}} g(x',x) dx'.
\end{align}
\end{corollary}
\medskip

\subsection{Proof of Proposition \ref{P:1}}
We are now ready to present the proof of Proposition \ref{P:1}. 
\begin{proof}[Proof of Proposition \ref{P:1}]
The proof is divided into a number of steps.

\step 1 The main estimate. Fix $R\in (2,3)$ so that \eqref{eq:goodWasserstein}, \eqref{eq:goodBoundaryData}, \eqref{eq:goodData} and \eqref{eq:goodData2} from Section \ref{sec:goodRadius} hold. Introduce $\phi$ with $\int_{\tilde Q_R}\phi = 0$ satisfying
\begin{align*}
	-\Delta\phi =\hat\kappa_{\mu,R}-\hat\kappa_{\lambda,R}\mbox{ in }\tilde Q_R,\quad \nabla\phi\cdot\nu = \bar g_R+\hat g_{1,R}-\bar f_R-\hat g_{0,R}\mbox{ on }\p\tilde Q_R.
\end{align*}
as well as $\phi^r$ with $\int_{\tilde Q_R} \phi^r=0$ solving
\begin{align}\label{c12}
-\Delta\phi^r = \hat\kappa_{\mu,R}-\hat\kappa_{\lambda,R}\mbox{ in }\tilde Q_R,\quad \nabla\phi^r\cdot\nu = \bar g_R^r+\hat g_{1,R}-\bar f^r_R-\hat g_{0,R}\mbox{ on }\p\tilde Q_R.
\end{align}
where $\bar f^r, \bar g^r $ are mollification of $f,g$ at scale $r$ such that $\bar f^r,\bar g^r$ are smooth on each of the faces of $\p \tilde Q_R$, c.f. \eqref{eq:mollification} . Note in particular, that for $r\leq 1$, $\int_F f^r=\int_F f$, so that \eqref{c12} is well-posed. We record that combining the standard energy estimate, Lemma \ref{lem:regularity} and \eqref{eq:goodBoundaryData},
\begin{align}
\int_{\tilde Q_R} |\nabla \phi^r|^2\lesssim E(\pi,4)+D(4)+D_{bdy}(4T),\label{energy}\\
	\sup_{\tilde Q_R}|\nabla\phi^r|^2\lesssim\frac{1}{r^{d-1}}(E(\pi,4)+D(4)+D_{bdy}(4T)),\label{c15}\\
[\nabla\phi^r]_{\alpha,\tilde Q_R}^2\lesssim \frac{1}{r^{d-1+2\alpha}}(E(\pi,4)+D(4)+D_{bdy}(4T))\label{c16}.
\end{align}

Let $\tau>0$ be given and consider $\epsilon_1>0$ to be fixed later. 
Noting that $f_R,g_R$ are supported on $\p\tilde Q_R$ by \eqref{c2}, following the proof of \cite[Corollary 1]{KochOtto} line by line, we find that
\begin{align}\label{c3}
	\int_{\#_R}\int_{\sigma}^{\tau}|\dot X(t)-\nabla\phi^r(X(t))|^2dtd\pi&\leq \int_{\#_R}|y-x|^2d\pi-\int_{\tilde Q_R}|\nabla\phi^r|^2\\\label{c4}
	&+\int_{\#_R}\int_{\sigma}^{\tau}|\nabla\phi^r(X(t))|^2dtd\pi-\int_{\tilde Q_R}|\nabla\phi^r|^2\\\label{c5}
	+2\int_{\tilde Q_R}\phi^r(-\Delta\phi^r-d(\mu-\lambda))&+2\int_{\p \tilde Q_R}\phi(\nabla\phi^r\cdot\nu-d(g_R-f_R)).
\end{align}
We estimate each term on the right-hand side in turn.

\step 2 The first term. 
We first
estimate $W_2^2(\bar\lambda,\bar\mu)$ where $\bar \lambda,\bar\mu$ are marginals of $\pi\llcorner\#_R$. We observe that we can write
\begin{align*}
\bar\lambda=\lambda\llcorner Q_R+f'\quad\text{ where }\quad\int\zeta df':=\int_{Q_R^c\times Q_R}\zeta(x)d\pi,\\
\bar\mu =\mu\llcorner Q_R+g'\quad\text{ where }\quad\int\zeta dg':=\int_{Q_R\times Q_R^c}\zeta(y)d\pi.
\end{align*}
Thus, using 
\begin{align*}
\int_{Q_R^c\times Q_R}\zeta(X(\sigma),x)d\pi
\end{align*}
as a competitor, we obtain that after decreasing $\e_1$ if necessary,
\begin{align*}
W_2^2(f_R,f')\leq&\int_{Q_R^c\times Q_R}|X(\sigma)-x|^2d\pi\leq\int_{\p(\#_R)}|y-x|d\pi\\
\overset{\eqref{eq:goodCrossing}}{\leq}&\tau E(\pi,4)+C_\tau(D(4)+D_{bdy}(4T)),
\end{align*}
and similarly for $W_2^2(g,g')$. Consequently, by assumption (C), we obtain that
\begin{align*}
&\int_{\#_R}|x-y|^2d\pi\leq W_2^2(\lambda\llcorner Q_R+f',\mu\llcorner Q_R+g')+\Lambda R^{d+2}\\
&\leq W_2^2(\lambda\llcorner Q_R+f_R,\mu\llcorner Q_R+g_R)+\Lambda R^{d+2}+\tau E(\pi,4)+C_\tau(D(4)+D_{bdy}(4T))
\end{align*}
Hence we focus on the first RHS term. By triangle inequality, \eqref{eq:goodWasserstein} and \eqref{eq:goodData},
\begin{align*}
W_2(\lambda\llcorner Q_R+f_R&,\nu\llcorner Q_R+g_R)\leq W_2(\hat\kappa_{\lambda,R}dx\llcorner \tilde Q_R+\hat g_{0,R}+\bar f_R,+\hat\kappa_{\mu,R}dx\llcorner \tilde Q_R+\hat g_{1,R}+\bar g_R)\\
&+W_2(f_R,\bar f_R)+W_2(g_R,\bar g_R)+ W_2(\lambda\llcorner Q_R,\hat\kappa_{\lambda,R}dx\llcorner\tilde Q_R+\hat g_{0,R})\\
&+W_2(\mu\llcorner Q_R,\hat\kappa_{\mu,R}dx\llcorner\tilde Q_R+\hat g_{1,R})+\\
&\leq W_2(\hat\kappa_{\lambda,R}dx\llcorner \tilde Q_R+\hat g_{0,R}+\bar f_R,\hat\kappa_{\mu,R}dx\llcorner \tilde Q_R+\hat g_{1,R}+\bar g_R)\\
&+\tau E(\pi,4)^{\frac{1}{2}} +C(D(4)^{\frac{1}{2}}+D_{bdy}(4T)^\frac 1 2)
\end{align*}
Now the benefit of $\hat\kappa_{\lambda,R}dx\llcorner \hat Q_R$ and $\hat\kappa_{\mu,R}dx\llcorner\tilde Q_R$ in the marginals of the first right hand side term is that they are uniformly positive on their support. This allows us to use Benamou-Brenier formula to construct a competitor for  the first RHS term above, cf.~\cite[Lemma 3]{KochOtto}, in order to obtain 
\begin{align*}
	W_2^2(\hat\kappa_{\lambda,R}\llcorner\tilde Q_R+\bar f_R+\hat g_0,\hat \kappa_{\mu,R}dx\llcorner\tilde Q_R+\bar g_R+\hat g_1)\leq \frac{1}{\min\{\hat\kappa_{\lambda,R},\hat\kappa_{\mu,R}\}}\int_{\tilde Q_R}|\nabla\phi|^2
\end{align*}

\medskip

Collecting the estimates, and using that $a^2-b^2 \leq 2a(a-b)$, we deduce
\begin{align}
&\int_{\#_R}|y-x|^2d\pi-\int_{\tilde Q_R}|\nabla\phi^r|^2\\
\lesssim& \left(\tau^\frac 1 2 E(\pi,4)^\frac 1 2+D(4)^\frac 1 2+D_{bdy}(4T)^\frac 1 2+\Lambda^\frac 1 2\left(\frac 1 {\min(\hat \kappa_{\lambda,R},\hat \kappa_{\mu,R})}\int_{\tilde Q_R} |\nabla \phi|^2\right)^\frac  1 2\right)\\
&\times\left(\tau^\frac 1 2 E(\pi,4)^\frac 1 2+D(4)^\frac 1 2+D_{bdy}(4T)^\frac 1 2+\Lambda^\frac 1 2+\left(\frac 1 {\min(\hat \kappa_{\lambda,R},\hat \kappa_{\mu,R})}\int_{\tilde Q_R} |\nabla \phi|^2\right)^\frac  1 2-\left(\int_{\tilde Q_R}|\nabla \phi^r|^2\right)^\frac 1 2\right)
\end{align}
Recall the energy estimate \eqref{energy}. Standard mollification estimates give
\begin{align}
\int_{\tilde Q_R}|\nabla\phi^r-\nabla\phi|^2\lesssim r(E(\pi,4)+D(4T)),\label{c14}
\end{align}
In combination with elementary estimates, this allows us to deduce
\begin{align*}
	&\int_{\#_R}|y-x|^2d\pi-\int_{\tilde Q_R}|\nabla\phi^r|^2\\
    \lesssim& \tau E(\pi,4)+D(4)+D_{bdy}(4)+\Lambda\\
    &+C(\tau) \left(1-\frac 1 {\min(\hat \kappa_{\lambda,R},\hat \kappa_{\mu,R})}\right)\int_{\tilde Q_R} |\nabla \phi|^2+C(\tau)\int_{\tilde Q_R} |\nabla \phi-\nabla \phi^r|^2\\
    \stackrel{\eqref{eq:goodData}}\lesssim& (\tau+r)E(\pi,4)+C_\tau\left(D(4)+D_{bdy}(4T)+\Lambda\right).
\end{align*}

\step 3 The second term.
We write \eqref{c4} as follows
\begin{align*}
\int_{\#_R}&\int_{\sigma}^{\tau}|\nabla\phi^r(X(t))|^2dtd\pi-\int_{\tilde Q_R}|\nabla\phi^r|^2\\
&=\int_{\#_R}\int_{0}^{1}I(X(t)\in\tilde Q_R)|\nabla\phi^r(X(t))|^2dtd\pi-\int_{\tilde Q_R}|\nabla\phi^r|^2dx=I
\end{align*}
Then observe that
\begin{align*}
	I(X(t)\in\tilde Q_R)\leq I\left(X(t)\in\tilde Q_R,\:\exists s\in[0,1]\: X(s)\notin\p(\tilde Q_R)\right)+I\left(X(t)\in\tilde Q_R,\forall t\in[0,1]\right)
\end{align*}
so that we may estimate
\begin{align*}
I\lesssim& \sup_{\tilde Q_R}|\nabla\phi^r|^2\pi(\p(\#_R))+[|\nabla\phi^r|^2]_{\alpha,\tilde Q_R}\int_{\#_R}|y-x|^\alpha d\pi\\
\stackrel{\eqref{c15},\eqref{c16}}\lesssim& \left(\frac 1 {r^{d-1}}(\pi(\p(\#_R))+\frac 1 {r^{d-1+\alpha}} \left(\int_{\#_R} |y-x|^2 d\pi\right)^\frac \alpha 2\right) (E(\pi,4)+D(4)+D_{bdy}(4)) \\
\stackrel{\eqref{eq:goodCrossingMeasure}}\lesssim& \left(\frac{\tau}{r^{d-1}}+\frac{E(\pi,4)^\alpha}{r^{d-1+\alpha}}\right) E(\pi,4)+D(4)+D_{bdy}(4).
\end{align*}

\step 4 The third term.
By \eqref{c12} and \eqref{c13}, we need to estimate
\begin{align*}
&\int_{\tilde Q_R}\phi^r(-d(\mu-\lambda))+\int_{\p\tilde Q_R}\phi^r(\bar g^r_R-\bar f^r_R+\hat g_{1,R}-\hat g_{0,r}-d(g_R-f_R))\\
&=\left[\int_{\tilde Q_R}\phi^r(\kappa_{\lambda,R}\mathbf{1}_{\Omega_0}-\kappa_{\mu,R}\mathbf{1}_{\Omega_1})+\int_{\p\tilde Q_R}\phi^r(\hat g_{1,R}-\hat g_{0,R})\right]\\
&+\int_{\tilde Q_R}\phi^r(-\kappa_{\lambda,R}\mathbf{1}_{\Omega_0}+\kappa_{\mu,R}\mathbf{1}_{\Omega_1}-d(\mu-\lambda))+\int_{\p\tilde Q_R}\phi^r(\bar g^r_R-\bar f^r_R-d(g_R-f_R))\\
&=:I+II+III.
\end{align*}
Let us first estimate $I$. Using the normalizing condition $\int_{\tilde Q_R}\phi^r=0$ and that $\tilde Q_R=\hat Q_R\cup (Q_R\cap\{-D_{bdy}(4)^{\frac 1 2}< x_1\leq D_{bdy}(4)^{\frac 1 2}\})$, we can rewrite the first term in $I$ as follows,
\begin{align*}
&\int_{\tilde Q_R}\phi^r(\kappa_{\lambda,R}\mathbf{1}_{\Omega_0}-\kappa_{\mu,R}\mathbf{1}_{\Omega_1})=\int_{\hat Q_R}\phi^r(\kappa_{\lambda,R}-\kappa_{\mu,R})+\int_{-D_{bdy}(4)^{\frac 1 2}}^{D_{bdy}(4)^{\frac 1 2}}\int_{Q_R'}\phi^r(\kappa_{\lambda,R}\mathbf{1}_{\Omega_0}-\kappa_{\mu,R}\mathbf{1}_{\Omega_1})\\
=&\int_{-D_{bdy}(4)^{\frac 1 2}}^{D_{bdy}(4)^{\frac 1 2}}\int_{ Q_R'}\phi^r(\kappa_{\lambda,R}-\kappa_{\mu, R})+\int_{-D_{bdy}(4)^{\frac 1 2}}^{D_{bdy}(4)^{\frac 1 2}}\int_{Q_R'}\phi^r(\kappa_{\lambda,R}\mathbf{1}_{\Omega_0}-\kappa_{\mu,R}\mathbf{1}_{\Omega_1})
\end{align*}
where $Q_R':=(-R,R)^{d-1}$ is the $(d-1)$-dimensional cube. Thus we have that $I=I_1+I_2$ where 
\begin{align*}
I_1&:=\int_{-D_{bdy}(4)^{\frac 1 2}}^{D_{bdy}(4)^\frac 1 2}\int_{ Q_R'}\phi^r(\kappa_{\lambda,R}-\kappa_{\mu, R}),\\
I_2&:=\int_{-D_{bdy}(4)^{\frac 1 2}}^{D_{bdy}(4)^{\frac 1 2}}\int_{Q_R'}\phi^r(\kappa_{\lambda,R}\mathbf{1}_{\Omega_0}-\kappa_{\mu,R}\mathbf{1}_{\Omega_1})+\int_{\p\tilde Q_R}\phi^r(\hat g_{1,R}-\hat g_{0,R}).
\end{align*}
We now estimate $I_1$. Note that $\sup_{\tilde Q_R}|\phi^r|\lesssim \sup_{\tilde Q_R}|\nabla\phi^r|$, therefore,
\begin{align*}
|I_1|&\lesssim  D_{bdy}(4)^\frac 1 2|\kappa_{\lambda, R}-\kappa_{\mu, R}|\sup_{\tilde Q_R}|\nabla\phi^r|\\
&\overset{\eqref{eq:goodData2},\eqref{c15}}{\lesssim} D_{bdy}(4) (D(4))^{\frac 1 2}\cdot\frac{1}{r^{\frac{d-1}{2}}}(E(\pi,4)+D(4)+D_{bdy}(4T))^{\frac 1 2}\\
&\lesssim \frac{\tau}{r^{\frac{d-1}{2}}}E(\pi,4)+C_\tau (D(4)+D_{bdy}(4T)).
\end{align*}
To estimate $I_2$,  we recall the Poincar\'e inequality in a thin rectangle: for a function $ f:Q_R\cap\{-\delta \leq x_1\leq \delta\}\to \R $ such that $ f(-\delta,x')\equiv 0 $, it holds that
\begin{align}\label{c17}
	\|f\|_{L^2(Q_R\cap\{-\delta\leq x_1\leq \delta\})}\lesssim \delta \|\nabla f\|_{L^2(Q_R\cap\{-\delta\leq x_1\leq \delta\})}.
\end{align} 
Then we have by definition of $\hat g_{1,R}$ and $\hat g_{0,R}$, that is \eqref{hatg-def},
\begin{align*}
I_2&=\int_{-D_{bdy}^{\frac 1 2}}^{D_{bdy}^{\frac 1 2}}\int_{Q_R'}\phi^r(\kappa_{\lambda,R}\mathbf{1}_{\Omega_0}-\kappa_{\mu,R}\mathbf{1}_{\Omega_1})+\int_{\p\tilde Q_R}\phi^r(\hat g_{1,R}-\hat g_{0,R})\\
&=\int_{-D_{bdy}(4)^{\frac 1 2}}^{D_{bdy}(4)^{\frac 1 2}}\int_{Q_R'}\phi^r(x_1,x')(\kappa_{\lambda,R}\mathbf{1}_{\Omega_0}-\kappa_{\mu,R}\mathbf{1}_{\Omega_1})\\
&+\int_{-D_{bdy}(4)^{\frac 1 2}}^{D_{bdy}(4)^{\frac 1 2}}\int_{Q_R'}\phi^r(-D_{bdy}(4)^{\frac 1 2},x')(\kappa_{\mu,R}\mathbf{1}_{\Omega_1}-\kappa_{\lambda,R}\mathbf{1}_{\Omega_0})
\end{align*}
Using $|\kappa_{\lambda,R}\mathbf{1}_{\Omega_0}-\kappa_{\mu,R}\mathbf{1}_{\Omega_1}|\leq 2 $, this implies
\begin{align*}
|I_2|&\leq 2\int_{-D_{bdy}(4)^{\frac 1 2}}^{D_{bdy}(4)^{\frac 1 2}}\int_{Q_R'}|\phi^r(-D_{bdy}(4)^{\frac 1 2},x')-\phi^r(x_1,x')|dx_1dx'\\
&\lesssim D_{bdy}(4)^{\frac{1}{4}}\left(\int_{-D_{bdy}(4)^{\frac 1 2}}^{D_{bdy}(4)^{\frac 1 2}}\int_{Q_R'}|\phi^r(-D_{bdy}(4)^{\frac 1 2},x')-\phi^r(x_1,x')|^2 \right)^{\frac{1}{2}}\\
&\overset{\eqref{c17},\eqref{energy}}{\lesssim} D_{bdy}(4)^{\frac 3 4}(\int_{\tilde Q_R}|\nabla\phi^r|^2)^{\frac{1}{2}}\\
&\lesssim D_{bdy}(4)^{\frac{3}{4}}(E(\pi,4)+D(4)+D_{bdy}(4T))^{\frac{1}{2}}\leq \tau E(\pi,4)+C_{\tau}(D(4)+D_{bdy}(4T)).
\end{align*}
Combining the estimates for $I_1$ and $I_2$, we obtain that 
\begin{align*}
|I|\lesssim\frac{\tau}{r^{\frac{d-1}{2}}}E(\pi,4)+C_\tau (D(4)+D_{bdy}(4T)).
\end{align*}

The control of $II$ and $III$ follows similarly as in \cite{KochOtto} by exploiting regularity of $\phi^r$. For $II$, we let $\pi$ be optimal for $W_2(\kappa_{\mu,R}dx\llcorner\Omega_{1,R},\mu\llcorner Q_R)$, then 
\begin{align*}
\int_{\tilde Q_R}\phi^r(\kappa_{\mu,R}\mathbf{1}_{\Omega_1}-d\mu)=&\int(\phi^r(x)-\phi^r(y))d\pi\leq \sup_{\tilde Q_R}|\nabla\phi^r|\int|x-y|d\pi\\
\lesssim& \sup_{\tilde Q_R}|\nabla\phi^r|W_2(\kappa_{\mu,R}dx\llcorner\Omega_{1,R},\mu\llcorner Q_R)\\
\overset{\eqref{c15},\eqref{eq:goodData2}}{\lesssim}&\frac{1}{r^{\frac{d-1}{2}}}(E(\pi,4)^{\frac{1}{2}}+D(4)^{\frac{1}{2}}+D_{bdy}(4T)^\frac 1 2)(D(4)^{\frac{1}{2}}+D_{bdy}(4T)^\frac 1 2)
\end{align*}
With similar considerations with $\mu$ replaced by $\lambda$, we obtain
\begin{align*}
II\lesssim& \frac{1}{r^{\frac{d-1}{2}}}(E(\pi,4)^{\frac{1}{2}}+D(4)^{\frac{1}{2}}+D_{bdy}(4T)^\frac 1 2)(D(4)^{\frac{1}{2}}+D_{bdy}(4)^\frac 1 2)\\
\lesssim& \frac{\tau}{r^{\frac{d-1}{2}}}E(\pi,4)+C_\tau(D(4)+D_{bdy}(4T)).
\end{align*}
We next estimate $III$. Letting $\pi$ be optimal for $W_2(\bar g,g)$, we see by similar arguments as for $II$,
\begin{align*}
\int_{\p\tilde Q_R}\phi^r(\bar g^r-dg)&=\int_{\p\tilde Q_R}\phi^r(\bar g^r-\bar g)+\int_{\p\tilde Q_R}\bar g-dg\\
&\lesssim \left(r\int_{\tilde Q_R}|\nabla\phi^r|^2\int_{\p\tilde Q_R}\bar g^2 \right)^{\frac{1}{2}}+\sup_{\tilde Q_R}|\nabla\phi^r|W_2(\bar g,g)(\int_{\p\tilde Q_R}\bar g)^{\frac{1}{2}}\\
&\overset{\eqref{c15},\eqref{eq:goodWasserstein},\eqref{eq:goodBoundaryData}}{\lesssim}\left(r^{\frac{1}{2}}+\frac{\tau}{r^{\frac{d-1}{2}}}\right)(E(\pi,4)+D(4)+D_{bdy}(4T)).
\end{align*}
The same argument shows that the above is also true with $g$ replaced by $f$, hence 
\begin{align*}
III\lesssim \left(r^{\frac{1}{2}}+\frac{\tau}{r^{\frac{d-1}{2}}}\right)(E(\pi,4)+D(4)+D_{bdy}(4T)).
\end{align*}
Combining the estimates for $I,II$ and $III$ above, we obtain the estimate of \eqref{c5},
\begin{align}\label{c20}
\int_{\tilde Q_R}&\phi(-\Delta\phi-d(\mu-\lambda))+\int_{\p \tilde Q_R}\phi(\nabla\phi\cdot\nu-d(g-f))\notag\\
&\lesssim (r^{\frac{1}{2}}+\frac{\tau}{r^{\frac{d-1}{2}}})E(\pi,4)+\frac{C_{\tau}}{r^{\frac{d-1}{2}}}(D(4)+D_{bdy}(4T)).
\end{align}

\step 5 Conclusion.
We combine the estimates \eqref{c3}-\eqref{c5} obtained from previous sections to see that 
\begin{align*}
\int_{\#_R}\int_{\sigma}^{\tau}|y-x-\nabla\phi^r(X(t))|^2dtd\pi&\lesssim \left(\frac{\tau+E(\pi,4)}{r^{d-1+\alpha}}+r^{\frac{1}{2}}\right)E(\pi,4)+\frac{C_\tau}{r^{d-1+\alpha}}(D(4)+D_{bdy}(4T)).
\end{align*}
To finish the proof, we need to estimate the LHS and then choose $\tau$ and $r$. We begin by noting that
\begin{align}\label{c26}
(\tau-\sigma)(x,y)=1\text{ for any }(x,y)\notin \p(\#_R).
\end{align}
Therefore we have that
\begin{align*}
&\int_{\#_1}|y-x-\nabla\phi^r(x)|^2d\pi\\
&\leq\int_{\#_1\backslash \p(\#_R)}|y-x-\nabla\phi^r(x)|^2d\pi+\int_{\p(\#_R)}2|y-x|^2+2|\nabla\phi^r(x)|^2d\pi\\
&\overset{\eqref{c26}}{\lesssim} \int_{\#_1}\int_{\sigma}^{\tau}|y-x-\nabla\phi^r(X(t))|^2dtd\pi+\int_{\#_1}\int_{\sigma}^{\tau}|\nabla\phi^r(X(t))-\nabla\phi^r(X(0))|^2dtd\pi\\
&\quad\quad+\frac{\tau}{r^{d-1}}(E(\pi,4)+D(4)+D_{bdy}(4T))
\end{align*}
where for the last line we also used \eqref{eq:goodCrossing}, \eqref{eq:goodCrossingMeasure} and \eqref{c15}. We next estimate
\begin{align*}
&\int_{\#_1}\int_{\sigma}^{\tau}|\nabla\phi^r(X(t))-\nabla\phi^r(X(0))|^2dtd\pi\lesssim [|\nabla\phi^r|^2]_{\alpha,\tilde Q_{\frac{3}{2}}} \int |y-x|^{2\alpha} d\pi\\
&\overset{\eqref{c16}}{\lesssim}\frac{1}{r^{d-1+\alpha}}\left(E(\pi,4)+D(4)+D_{bdy}(4T)\right)^{1+\alpha}.
\end{align*}
Collecting estimates, we conclude that
\begin{align}\label{eq:finalEstimateApproximation}
\int_{\#_1}&|y-x-\nabla\phi^r(x)|^2\,\mathrm{d}\pi\lesssim(\frac{\tau+E(\pi,4)^\alpha}{r^{d-1+\alpha}}+r^{\frac{1}{2}})E(\pi,4)+\frac{C_{\epsilon_1}}{r^{d-1+\alpha}}(D(4)+D_{bdy}(4T)).
\end{align}

Introduce now
\begin{align*}
-\Delta\tilde\phi^r = \hat\kappa_{\mu,R}-\hat\kappa_{\lambda,R}\mbox{ in }\tilde Q_R,\quad \nabla\tilde\phi^r\cdot\nu = h^r\mbox{ on }\p\tilde Q_R,
\end{align*}
where $h^r=\bar g_R^r-\bar f^r_R+T_r(\hat g_{1,R}-\hat g_{0,R})$. Here for any function $f$ defined on $F=\p\tilde Q_R \cap \{x_1 = - D_{bdy}^\frac 1 2\}$, we define
\begin{align*}
    T_r(f)(x) = f^r(x)-\langle\nabla f^r(0),x\rangle.
\end{align*}
Note that $\nabla' T_r(f)(0)=0$. We note that since $\nabla' \hat g_{1,r}=\nabla' \hat g_{0,R}=0$,
\begin{align*}
\|h^r\|_{L^2(F)}\lesssim& \|\hat g_{1,R}^r-\hat g_{0,R}^r\|_{L^2(F)}+r^\alpha ([\nabla' g_{0,R}]_{\alpha,Q'_R}+[\nabla' g_{1,R}]_{\alpha,Q'_R}).
\end{align*}
By standard properties of mollification and using the definition \eqref{hatg-def} of $\hat g_{0,R},\hat g_{1,R}$ as well as Assumption \ref{ass:tangencyLocal}, we deduce
\begin{align}
\|h^r\|_{L^2(F)}\lesssim& \|\hat g_{1,R}^r-\hat g_{0,R}^r\|_{L^2(F)}+r^\alpha ([\nabla' g_{0,R}]_{\alpha,Q'_R}+[\nabla' g_{1,R}]_{\alpha,Q'_R})\nonumber\\
\lesssim& r \left(\|\hat g_{0,R}\|_{C^{0,1}(\partial \tilde Q_R)}+\|\hat g_{1,R}\|_{C^{0,1}(\partial \tilde Q_R)}\right)+r^\alpha ([\nabla' g_{0,R}]_{\alpha,Q'_R}+[\nabla' g_{1,R}]_{\alpha,Q'_R})\nonumber\\
\lesssim r^\alpha D'_{bdy}(R).\label{eq:difficultFace}
\end{align}

By triangle inequality,
\begin{align*}
&\int_{\#_1} |y-x-\nabla \tilde \phi^r|^2\dd\pi \\
=& \int_{\#_1} |y-x-\nabla \phi^r|^2\dd \pi+\int_{\#_1} \langle \nabla \phi^r-\nabla \tilde \phi^r,2(y-x)-\nabla \phi^r-\nabla \tilde \phi^r\rangle\dd \pi\\
\leq& \int_{\#_1} |y-x-\nabla \phi^r|^2 \dd \pi+\|\nabla \phi^r-\nabla \tilde \phi^r\|_{L^2(\tilde Q_R)}\left(E(\pi,4)+\|\nabla \phi^r\|_{L^2(\tilde Q_R)}^2+\|\nabla \tilde\phi^r\|_{L^2(\tilde Q_R)}^2\right)^\frac 1 2.
\end{align*}

By energy estimate and \eqref{eq:difficultFace},
\begin{align*}
\|\nabla \phi^r-\nabla \tilde \phi^r\|_{L^2(\tilde Q_R)}\lesssim& \|h^r\|_{L^2(\partial \tilde Q_R)}\lesssim r^\alpha D'_{bdy}(R).
\end{align*}

In combination with \eqref{energy}, we have shown
\begin{align*}
\int_{\#_1} |y-x-\nabla \tilde \phi^r|^2\lesssim r D_{bdy}(E(\pi,4)+D(4)+D_{bdy}'(4T))
\end{align*}
and it follows that \eqref{eq:finalEstimateApproximation} holds with $\tilde \phi$ in place of $\phi$. Choosing first $r\ll 1$ sufficiently small and then ensuring $\frac{\tau+E(\pi,4)^\alpha}{r^{d-1+\alpha}}\leq \frac{\tau+\varepsilon_1^\alpha}{r^{d-1+\alpha}}\ll 1$, we conclude the proof.

\step 6 Proof of \eqref{e:phi-C2alpha}. By Lemma \ref{lem:regularity}, in particular \eqref{c53} applied to the equation \eqref{c12}, as well as \eqref{eq:difficultFace}, we have 
\begin{align*}
\|\tilde\phi\|_{C^{3}(\tilde Q_1)}\lesssim \|\bar f^r_R\|^2_{C^2(\p \tilde Q_R)} +\|\bar g^r_R\|^2_{C^2(\p \tilde Q_R)}+\|\hat g_{0,R}^r\|^2_{C^{2}(Q_R')}+\|\hat g_{1,R}^r\|_{C^{2}(Q_R')}^2+D_{bdy}'(R).
\end{align*}
With $r$ being fixed relative to $\tau$ in the last step, standard mollification estimates give 
\begin{align*}
\|\bar f^r_R\|^2_{C^\alpha(\p \tilde Q_R)} +\|\bar g^r_R\|^2_{C^\alpha(\p \tilde Q_R)}\leq C\int_{\p \tilde Q_R}\bar f_R^2+\bar g_R^2\overset{\eqref{eq:goodBoundaryData}}{\leq} C (E(\pi,4)+D(4)+D_{bdy}(4T))
\end{align*}
where $C=C(\tau,d,\alpha)$. On the other hand, similar estimates in combination with the definition \eqref{hatg-def} of $\hat g_{0,R},\hat g_{1,R}$ as well as Assumption \ref{ass:tangencyLocal} give
\begin{align*}
\|\hat g_{0,R}^r\|^2_{C^{2}(Q_R')}+\|\hat g_{1,R}^r\|_{C^{2}(Q_R')}^2\lesssim&\|\hat g_{0,R}\|^2_{C^{1,\alpha}(Q_R')}+\|\hat g_{1,R}\|_{C^{1,\alpha}(Q_R')}^2\\
\lesssim& D_{bdy}(R)+ [\nabla g_0]^2_{\alpha,Q_4'}+ [\nabla g_1]^2_{\alpha,Q_4'}\\
\lesssim& [\nabla g_0]^2_{\alpha,Q_4'}+ [\nabla g_1]^2_{\alpha,Q_4'},
\end{align*}
where the last inequality again follows from Assumption \ref{ass:tangencyLocal}.
Combining the above two estimates finishes the proof of \eqref{e:phi-C2alpha}.
\end{proof}

\section{One-step improvement}
The goal of this section is to show that control of the energy at some scale implies control of the energy at a smaller scale, up to an affine change of coordinates. We derive this fact as a consequence of the harmonic approximation result in Proposition \ref{P:1} in combination with a Taylor expansion argument. Throughout this section, we utilise the notation introduced in Section \ref{sec:harmonixApproximation}. We also recall that
\begin{align}
\mathscr D(R,\lambda,\mu)= R^{2\alpha}\left([\lambda]_{\alpha,R}^2+[\mu]_{\alpha,R}^2\right)+R^{2\alpha}\left([\nabla g_0]^2_{\alpha, Q_R'}+[\nabla g_1]^2_{\alpha, Q_R'}\right)
\end{align}
which encodes the flatness of the marginals $\lambda,\mu$ and of the boundary $g_0,g_1$ in H\"older norms.
We begin by showing how to move from the Wasserstein data term in Section \ref{sec:harmonixApproximation} to the H\"older-based control above.

\begin{lemma}\label{lem:changeOfData}
Assume $\lambda$ is a measure with $C^{\alpha}$ density $f$ supported on a $C^{1,\alpha}$ domain $\Omega$ with $0\in\Omega$ such that $f(0)=1$. There is $R_0\geq 0$ such that if for $R\leq R_0$,
\begin{align}\label{eq:defgForData}
\Omega \cap Q_{2R} = \{(x_1,x')\in Q_{2R}\mid  x_1>g(x'),\,x'\in Q_{2R}'\},\quad g(0)=0,\, \nabla g(0)=1,
\end{align}
then
\begin{align}
&\frac 1 {R^{d+2}} W_2^2(\lambda \llcorner Q_R,\kappa_{\lambda,R}dx \llcorner \Omega \cap Q_R)+(\kappa_{\lambda,R}-1)^2+(\sup_{Q_R} |g|)^2\\
\lesssim& R^{2\alpha}\left([\lambda]_{\alpha,R}^2+ [\nabla g]^2_{\alpha,Q_R}\right),
\end{align}
\end{lemma}
\begin{proof}
We note that \eqref{eq:defgForData} implies that
\begin{align}\label{eq:boundaryDatumEstimate}
\sup_{Q_R} |g|\lesssim R^{1+\alpha}[\nabla g_0]_{Q_R\cap \supp\,\lambda, \alpha}
\end{align}
We next note that for $x\in \Omega \cap Q_R$,
\begin{align}\label{lambdabound}
&|f(x)-1| = |f(x)-f(0)|\lesssim [\lambda]_{\alpha, R}R^\alpha.
\end{align}
Consequently, there is $C>0$ such that
\begin{align}\label{kappabound}
1-C[f]_{\alpha, R}R^\alpha\leq\kappa_{\lambda, R} = \frac{\int_{Q_R\cap \Omega_0}\lambda}{|Q_R\cap \Omega_0|}\leq 1+C[f]_{\alpha, R}R^\alpha.
\end{align}
Further, reducing $R_0$ if necessary, we may ensure that in $\Omega \cap Q_R$, $f(x)\in [1/2,2]$. 

It remains to bound the Wasserstein term. In light of Corollary \ref{cor:changingCubes}, it suffices to estimate
\begin{align*}
W_2^2(\lambda \llcorner \hat Q_R+\hat f_R, \hat\kappa_{\lambda,R} d x \llcorner \hat Q_R+\hat g_{0,R}).
\end{align*}
Using \cite[Lemma 3]{KochOtto}, we find that
\begin{align}\nonumber
W_2^2(\lambda \llcorner Q_R, \hat\kappa_{\lambda,R} d x \llcorner \tilde Q_R+\hat g_0)\leq&\frac 1 {\min(\lambda \llcorner \hat Q_R, \kappa_{\lambda,R}\llcorner\hat{Q}_R}\int_{\hat Q_R} |\nabla \phi|^2
\lesssim \int_{\hat Q_R} |\nabla \phi|^2,
\end{align}
where with $\hat g_{R}$ is defined from $g$ as in \eqref{hatg-def},
\begin{align*}
-\Delta\phi=\kappa_{\lambda, R} -\lambda\quad& \text{in}\;\hat{Q}_R\\
\nabla\phi\cdot \nu = \hat{f}_R-\hat g_{R}\quad& \text{on}\; \partial \hat{Q}_R\cap \{x_1=D_{bdy}\}.
\end{align*}
It is straightforward to see from the definition of $\hat f_R$ and $\hat g_{0,R}$ that
\begin{align*}
|\hat f_R|+|\hat g_R|\lesssim \sup_{Q_R} |g|
\end{align*}
Consequently, energy estimates and \eqref{lambdabound} give
\begin{align}
\int_{\hat Q_R} |\nabla \phi|^2\lesssim R\int_{\partial \hat Q_R \cap \{x_1=D_{bdy}\}} |\hat f_R|^2+|\hat g_{R}|^2+R^2 \|\kappa_{\lambda,R}-\lambda\|_{L^2(\hat Q_R)}\lesssim R^{d+2} \sup_{Q_R}|g|^2+R^{d+2+2\alpha} [\lambda]_{\alpha,R}^2
\end{align}

Finally, we note that as $g(0)=0$, $\nabla g(0)=0$,  $\sup_{Q_R}|g|^2\lesssim R^{2\alpha}[\nabla g]_{\alpha,Q_R}^2$.
Collecting estimates, this completes the proof.
\end{proof}

We are now ready to prove the following one-step improvement.
\begin{proposition}\label{one step improvement}
Assume $\lambda,\mu\in \mathscr M(\R^d)$ admit $C^{0,\alpha}$-densities satisfying $\lambda(0)=\mu(0)=1$. Suppose $\pi\in \Pi(\lambda,\mu)$ satisfies assumptions (A),(B1) and (C). For any $\beta \in (0, 1)$ there exists $\e_1>0$ such that if $E(\pi,R)+\mathscr D(R,\lambda,R)\leq \e_1$, then there exist constants $\theta \in (0, 1]$ and $C_\beta \ge 1$ depending only on $d, \alpha, \beta, r$, symmetric $B \in \mathbb{R}^{d \times d}$, $b \in \mathbb{R}^d$ and $\gamma>0$ satisfying
\begin{equation}\nonumber
|B - Id|^2 + \frac{|b|^2}{R^2} \lesssim E(\pi,R) + \mathscr D(R,\lambda,\mu),
\end{equation}
such that with $s=(A,b,\gamma,1)$ and $\pi_s$ defined as in Section \ref{sec:changeOfCoordinates},
\begin{align*}
E(\pi_s,\theta R)\leq \theta^{2\beta} E(\pi,R)+C_\beta \left(\mathscr D(R,\lambda,\mu)+ \Lambda\right)\\
\mathscr D(\theta R,\lambda_s,\mu_s)\leq \theta^{2\alpha}\left(1+C\left(E(\pi,R)^\frac 1 2+\mathscr D(R,\lambda,\mu)^\frac 1 2\right)\right)\mathscr D(R,\lambda,\mu).
\end{align*}
Moreover, $\lambda_s(0)=\mu_s(0)=1$ and Assumption \ref{ass:tangency} holds for the measures $\lambda_s,\mu_s$ in $Q_{\theta R}(0)$ and
\begin{align}\label{eq:conditiongamma}
|\gamma-1|^2\lesssim E(\pi,R)+\mathscr D(R,\lambda,\mu).
\end{align}
If (B2) holds instead of (B1), all statements hold with $D(R,\lambda,\mu)$ replaced by \begin{align*}
    R^{2\alpha}([\lambda]_{\alpha,R}+[\mu]_{\alpha,R})+R^{2\alpha}([\nabla g_0]_{\alpha,Q_{TR}}+[\nabla g_1]_{\alpha,Q_{TR}}).
\end{align*}
\end{proposition}
\begin{proof}
We only prove the statement with (B2). The proof with (B1) is analogous with the obvious changes arising from the presence of $T$.

By rescaling, it suffices to consider the case $R=4$. Fix $\beta<\alpha$. We apply Proposition \ref{P:1} in combination with Lemma \ref{lem:changeOfData} to obtain $\phi$ satisfying for $\tau>0$ to be determined at a later stage,
\begin{align}\label{eq:harmonicApproximationApplied}
\int_{\#_1} |x-y-\nabla \phi|^2 d\pi\leq \tau E(\pi,4)+ C_\tau\left( \mathscr D(4,\lambda,\mu)+\Lambda\right)\\
\int_{\tilde Q_1} |\nabla \phi|^2 dx \leq C(E(\pi,4)+\mathscr D(4,\lambda,\mu)).
\end{align}
Set $\bar{b}=\nabla \phi(0)$, $\bar{A}=\nabla^2\phi(0)$ and $\bar{B}=e^{\frac{-\bar{A}}{2}}$. Using Lemma \ref{lem:regularity} and recalling the definition of $\phi$ from Proposition \ref{P:1},
\begin{align}\label{estimatebarb}
&\lvert\bar{B}-Id\rvert^2+|\bar b|^2\lesssim |\bar{A}|^2+|\bar b|^2 \lesssim \|h\|_{L^2(\partial\tilde Q_R)}^2 \lesssim E(\pi,4)+\mathscr D(4,\lambda,\mu).
\end{align}

We now modify $\bar b, \bar B$ in order to ensure Assumption \ref{ass:tangency} holds for $\lambda_s,\mu_s$. Write $\bar b= (\bar b_1,\bar b')$, where $\bar b'$ is the $d-1$-dimensional entry. Set $\tilde b = (g_1(b')-\partial_1 \phi(0))$ and $b=\bar b+\tilde b =(g_1(\bar b'),b')$. Using Lemma \ref{lem:regularity}, which in particular implies that the Neumann-boundary condition is assumed in a strong sense at $0$, the definition of $\phi$, \eqref{estimatebarb} and Lemma \ref{lem:changeOfData}, we find
\begin{align}\nonumber
|\tilde b|^2 \lesssim& |g_1(\bar b')|^2+|\partial_1 \phi(0)|^2\\
\lesssim& \left([\nabla g_1]_{\alpha,\tilde Q_4} |\bar b|^{1+\alpha}\right)^2+|\kappa_{\mu,R}(D_{bdy}-g_1(0))-\kappa_{\lambda,R}(D_{bdy}(4)-g_0(0))|\\
\stackrel{\eqref{estimatebarb}}\lesssim& (E(\pi,4)+\mathscr D(4,\lambda,\mu))^{1+\alpha}+|(\kappa_{\mu,R}-\kappa_{\lambda,R})| D_{bdy}(4)\nonumber\\
\lesssim& E(\pi,4)^{1+\alpha}+\mathscr D(4,\lambda,\mu).\label{eq:tildeb}
\end{align}
Consequently,
\begin{align}\label{eq:bEstimate}
|b|^2 \lesssim |\tilde b|^2+|\bar b|^2 \lesssim E(\pi,4)+\mathscr D(4,\lambda,\mu).
\end{align}
We note that this choice of $b$ ensures that $0\in \partial\supp\,\lambda_s,\partial\supp\,\mu_s$. 

We next modify $B$ by constructing a symmetric matrix $\tilde B$ such that $\tilde B^{-\ast}\bar B^{-\ast}\nu_1(b)$ and $\tilde B \bar B \nu_0(0)$ agree and 
\begin{align}\label{eq:tildebEstimate}
|\tilde B-\textup{id}|^2 \lesssim E(\pi,4)^{1+\alpha}+\mathscr D(4,\lambda,\mu).
\end{align}
Setting $B=\bar B+\tilde B$ then ensures that the normals of $\supp\,\lambda_s$ and $\supp\,\mu_s$ agree. As in a small neighbourhood of the identity, the square-root of a symmetric matrix is well-defined and Lipschitz, it suffices to construct such a symmetric matrix $\tilde A$ satisfying $\tilde A \frac{\bar B \nu_0(0)}{|\bar B \nu_0(0)|}=\frac{\bar B^{-\ast}\nu_1(0)}{|\bar B \nu_1(0)|}$ and $|\tilde A-\textup{id}|^2\lesssim E(\pi,4)^{1+\alpha}+\mathscr D(4,\lambda,\mu)$.

We construct $\tilde A$ as follows. We think of $\R^d$ as the span of $v_1 = \frac{\bar B \nu_1(0)}{|\bar B \nu_1(0)|}$ and its orthogonal complement. On the one-dimensional space $\tilde A$ is defined by requiring $\tilde A v_1 = \frac{\bar B^{-\ast}\nu_1(b)}{|\bar B^{-\ast}\nu_1(b)|}$. By the symmetry requirement, this defines $\tilde A$ aside from the $d-1$-dimensional block, where we set $\tilde A$ to be the identity. 

In order to estimate $|\tilde A-\textup{id}|$, note it is sufficient to estimate how close the first column is to the coordinate vector of $v_1$, i.e. it suffices to estimate $|\tilde A v_1-v_1|$.

Set $\hat B = e^{\frac{-\nabla^2\phi(-D_{bdy}^\frac 1 2,0)}{2}}$. Due to $\nabla g_0(x'=0)=\nabla g_1(x'=0)=0$ and the boundary condition, $\hat B^{-\ast}\nu_1(0)$ and $\hat B \nu_0(0)$ are parallel to $-{e_1}=\nu_1(0)$.
Indeed, note that due to Lemma \ref{lem:regularity} the boundary condition is assumed in a $C^2$-sense and we may compute
\begin{align*}
	\nabla'\left(\nabla\phi\cdot \nu(-D_{bdy}^\frac 1 2,0)\right) &= \nabla' h(0)=0,
\end{align*}
giving the claim. This allows us to estimate, recalling that $\nu_1(0)=\nu_0(0)$,
\begin{align*}
|\tilde A-\textup{id}|^2\lesssim&\left|\frac{\bar B^{-\ast}\nu_1(b)}{|\bar B^{-\ast}\nu_1(b)|}-\frac{ \bar B \nu_1(0)}{|\bar B \nu_1(0)|}\right|^2 \\
\leq&\left|\frac{\hat B^{-\ast}\nu_1(b)}{|\hat B^{-\ast}\nu_1(b)|}-\frac{ \hat B \nu_1(0)}{|\hat B \nu_1(0)|}\right|^2 +  \left|\frac{\bar B\nu_1(0)}{|\bar B\nu_1(0)|}-\frac{ \hat B \nu_1(0)}{|\hat B \nu_1(0)|}\right|^2   +\left|\frac{\bar B^{-\ast}\nu_1(b)}{|\bar B^{-\ast}\nu_1(b)|}-\frac{ \hat B^{-\ast} \nu_1(b)}{|\hat B^{-\ast} \nu_1(b)|}\right|^2 \\
=&I + II + III.
\end{align*}

Note that, since $\hat B^{-\ast}\nu_1(0)$ and $\hat B \nu_0(0)$ are parallel to $-{e_1}=\nu_1(0)$,
\begin{align*}
I=&\left|\frac{\hat B^{-\ast}\nu_1(b)}{|\hat B^{-\ast}\nu_1(b)|}-\frac{ \hat B \nu_1(0)}{|\hat B \nu_1(0)|}\right|^2\\
\lesssim & |\nu_1(b)-\nu_1(0)|^2\lesssim ([\nabla' g_1]_{\alpha,\tilde Q_4}|b|^\alpha)^2\stackrel{\eqref{eq:bEstimate}}\lesssim E(\pi,4)^{1+\alpha}+\mathscr D(4,\lambda,\mu).
\end{align*}
$II$ and $III$ are estimated similarly and we focus on $II$. Exactly as for \eqref{estimatebarb}, we find $|\hat B-\textup{id}|^2 \lesssim E(\pi,4)+\mathscr D(4,\lambda,\mu)$. This allows to estimate using Lemma \ref{lem:regularity} and Lemma \ref{lem:changeOfData},
\begin{align*}
II\lesssim& |\hat B-\bar B|^2\lesssim |\nabla^2\phi(0)-\nabla^2 \phi(0,-D_{bdy}^\frac 1 2|^2\\
\lesssim& (E(\pi,4)+\mathscr D(4,\lambda,\mu))D_{bdy}^\frac 1 2\lesssim E(\pi,4)^{1+\alpha}+\mathscr D(4,\lambda,\mu).
\end{align*}

Using triangle inequality, we conclude setting $B=\tilde B\bar B$,
\begin{align}\label{eq:Bestimate}
|B-\textup{id}|^2\lesssim (|B| |\tilde B-\textup{id}|+|\tilde B| |B-\textup{id}|)^2\lesssim E(\pi,4)+\mathscr D(4,\lambda,\mu).
\end{align}

We now consider the change of coordinates $s=(A,b,\gamma,1)$ with $\gamma = \frac{\mu(b)^\frac 1 d}{\textup{det}(B)^\frac 2 d}$ and the associated transport plan $\pi_s$ with marginals $\lambda_s$ and $\mu_s$. Note that due to our work so far, $\lambda_s(0)=\mu_s(0)=1$ and Assumption \ref{ass:tangency} holds for $\lambda_s$ and $\mu_s$. 

We note due to smoothness of $\textup{det}$ near $1$ and by construction of $\tilde B$,
\begin{align*}
|\textup{det}(\tilde B)-1|^2\lesssim |\tilde B-1|^2\lesssim E(\pi,4)^{1+\alpha}+\mathscr D(4,\lambda,\mu).
\end{align*}
By definition of $\bar B$ and $\phi$ and using Lemma \ref{lem:changeOfData},
\begin{align*}
|\textup{det}(\tilde B)-1|^2 = |e^{(\hat \kappa_{\mu,R}-\hat \kappa_{\lambda,R})/2}-1|^2 \lesssim |\hat \kappa_{\lambda,R}-1|^2+|\hat \kappa_{\mu,R}-1|^2 \lesssim \mathscr D(4,\lambda,\mu).    
\end{align*}
We deduce that
\begin{align}\label{eq:gammaEstimate}
|\gamma-1|^2\lesssim& E(\pi,4)^{1+\alpha}+\mathscr D(4,\lambda,\mu).
\end{align}

We now estimate for $\theta \leq \frac 1 8$, noting that we may ensure $Q_s(\#_\theta)\subset \#_{2\theta}$,
\begin{align*}
 E(\pi_s,\theta) =& \theta^{-(d+2)}\int_{Q_s\#_\theta} |\gamma B(y-b)-B^{-\ast}x|^2 |\textup{det}(B)|^{-1} d\pi\\
 \lesssim& \theta^{-(d+2)} \int_{\#_{2\theta}} |(y-b)-B^{-1}B^{-T}x|^2+|\gamma-1|^2 |B^{-T}x|^2 d \pi\\
 \lesssim& \theta^{-(d+2)} \int_{\#_{2\theta}} |y-x-\nabla \phi|^2 + |\nabla \phi-\bar b-\bar A x|^2 + |x+\bar Ax-\bar B^{-2}x|^2\\
 &+ |\bar B^{-2}x-B^{-1}x|^2+|\bar b-b|^2+|\gamma-1|^2 |B^{-T}x|^2\d\pi.
\end{align*}
We now estimate each term in turn. First, using \eqref{eq:harmonicApproximationApplied},
\begin{align*}
\theta^{-(d+2)} \int_{\#_{2\theta}} |y-x-\nabla \phi|^2d\pi\leq \tau \theta^{-(d+2)}E(\pi,4)+C_{\tau,\theta} \mathscr D(4,\lambda,\mu).
\end{align*}
Next, by Taylor expansion and using \eqref{e:phi-C2alpha},
\begin{align*}
\theta^{-(d+2)} \int_{\#_{2\theta}} |\nabla \phi-\bar b-\bar A x|^2d\pi=& \theta^{-(d+2)}\int_{\#_{2\theta}} \int_{\#_{2\theta}} |\nabla^2 \phi(tx)x-\nabla^2 \phi(0)|^2 d\pi\\
\lesssim& \theta^{2} \|\phi\|^2_{C^{3}(\tilde Q_{2\theta})}\lesssim \theta^{2}(E(\pi,4)+\mathscr D(4,\lambda,\mu)).
\end{align*}
By Taylor expansion, recalling the definition of $\bar B, \bar A$ and using Lemma \ref{lem:regularity},
\begin{align*}
    \theta^{-(d+2)} \int_{\#_{2\theta}} |x+\bar Ax-\bar B^{-2}x|^2\lesssim \|\phi\|_{C^{3}(\tilde Q_{2\theta})}^2\lesssim E(\pi,4)^2+\mathscr D(4,\lambda,\mu).
\end{align*}
Next, we note
\begin{align*}
|\bar B^{-2}-B^{-1}|+|\bar b-b|\lesssim |\tilde B-\textup{id}|+|\tilde b|\stackrel{\eqref{eq:tildeb},\eqref{eq:tildebEstimate}}\lesssim E(\pi,4)^{1+\alpha}+\mathscr D(4,\lambda,\mu),
\end{align*}
to find
\begin{align*}
&\theta^{-(d+2)} \int_{\#_{2\theta}} |\bar B^{-2}x-B^{-1}x|^2+|\bar b-b|^2+|\gamma-1|^2 |B^{-T}x|^2\d\pi\\
\lesssim& \theta^{-(d+2)}\left(E(\pi,4)^{1+\alpha}+\mathscr D(4,\lambda,\mu)\right).
\end{align*}
Finally, using \eqref{eq:gammaEstimate},
\begin{equation}\label{gamma-1}
    \theta^{-(d+2)}\int_{{\#}_{2\theta}}\lvert \gamma-1\rvert^2\lvert B^{-T}x\rvert^2\,\mathrm{d}\pi\lesssim \theta^{-(d+2)}\left(E(\pi,4)^{1+\alpha}+\mathscr D(4,\lambda,\mu)\right).
\end{equation}
Collecting estimates, we have shown for some $C>0$,
\begin{align*}
E(\pi_s,\theta)\leq C\left(\tau \theta^{-(d+2)}+\theta^{2}+E(\pi,4)^{\alpha}\right)E(\pi,4)+C_{tau,\theta}\left(\mathscr D(4,\lambda,\mu)+\Lambda\right)
\end{align*}
Fixing first $\tau$ sufficiently small that $C\tau\theta^{-(d+2)}\leq \frac 1 3\theta^{2\beta}$, then $\e_1$ sufficiently small that $CE(\pi,4)^\alpha \leq \frac 1 3\theta^{2\beta}$ and finally $\theta$ sufficiently small that $C\theta^{2}\leq \frac{1} 3\theta^{2\beta}$, completes the proof of the one-step improvement for the energy.

It remains to prove the one-step improvement for the data term. We first note that
\begin{align*}
[\mu_s]_{\alpha,\tilde Q_{\theta}} =& \gamma^{-d}|\gamma^{-1}B^{-1}|^{\alpha} \sup_{x,y\in Q_{\theta}} \frac{|\mu(\gamma^{-1} B^{-1} x+b)-\mu(\gamma^{-1} B^{-1}y+b)|}{|(\gamma^{-1}B^{-1}x+b)-(\gamma^{-1}B^{-1} y+b)|^{\alpha}}\\
\leq& \gamma^{-(d+\alpha)}|B^{-1}|^{\alpha} \sup_{x',y'\in \tilde Q_4} \frac{|\mu(x')-\mu(y')|}{|x'-y'|^\alpha}\\
\stackrel{\eqref{eq:Bestimate}}\leq&\left(1+C\left(E(\pi,4)^\frac 1 2+\mathscr D(4,\lambda,\mu)^\frac 1 2\right)\right)[\lambda]_{\alpha,\tilde Q_4}.
\end{align*}

An analogous (and slightly simpler) argument applies for the term involving $\lambda$. Finally, we note that the functions $g_{0,s},g_{1,s}$ describing the boundary of $\supp\,\lambda_s$, $\supp\,\mu_s$ may be obtained via
\begin{align*}
g_{0,s}=g_0(Bx),\quad g_{1,s}(y)=g_1(\gamma^{-1} B^{-1}y+b).
\end{align*}
Consequently,
\begin{align*}
[\nabla g_{1,s}]_{\alpha,\tilde Q_{\theta}}^2\leq |B|^2 [g_{1}]_{\alpha,\tilde Q_4}\stackrel{\eqref{eq:Bestimate}}\leq \left(1+C\left(E(\pi,4)^\frac 1 2+\mathscr D(4,\lambda,\mu)\right)\right) [g_1]_{\alpha,\tilde Q_4},
\end{align*}
and an analogous estimate holds for $[\nabla g_{0,s}]_{\alpha,\tilde Q_{\theta}}$. Collecting estimates completes the proof.
\end{proof}

We will mainly apply Proposition \ref{one step improvement} in the following form adapted to regularised optimal transport. We remind the reader that we defined the extended energy $E_p$ in \ref{eq:extendedEnergy}.

\begin{corollary}\label{cor:oneStepRegularised}
Let $p\in (1,2]$. Assume $\lambda,\mu$ admit $C^{\alpha}$-densities on $\supp\,\lambda$ and $\supp\,\mu$, respectively. Let $\pi$ a minimiser of $OT_{\e,h_p}(\lambda,\mu)$. There is $R_0>0$ such that if $R\leq R_0$, Assumption \ref{ass:topological} is verified in $Q_R$, Assumption \eqref{ass:tangency} holds for $\supp\,\lambda$ and $\supp\,\mu$ in $Q_R$ and $\lambda(0)=\mu(0)=1$, then the following statement holds:  
For any $\beta \in (0,1)$, there exists $\e_1>0$  such that if 
\begin{align}\label{c85}
E_p(\pi,R)+\mathscr D(R,\lambda,\mu)\leq \e_1 \quad \mbox{and}\quad \frac{\e^{\frac{2}{d(p-1)+2}}}{R}\leq\e_1,
\end{align} and Assumption \eqref{eq:topological} holds in $B_R$, then there exist constants $\theta \in (0, 1]$ and $C_\beta \ge 1$ depending only on $d, \alpha, \beta$, $B \in \mathbb{R}^{d \times d}$ and $b \in \mathbb{R}^d$ satisfying
\begin{equation}\nonumber
|B - Id|^2 + \frac{|b|^2}{R^2} \lesssim E(\pi,R) + \mathscr D(R, \lambda, \mu)+\frac{R_c^2}{R^2},
\end{equation}
such that with $s=(B,b,\mu(b)^\frac 1 d,1)$ and $\pi_s$ defined as in Section \ref{sec:changeOfCoordinates},
\begin{align*}
E_p(\pi_s,\theta R)\leq \theta^{2\beta} E_p(\pi,R)+C_\beta\left( D(R)+\frac{R_c^2}{R^2}\right)\\
\mathscr D(\theta R,\lambda_s,\mu_s)\leq \theta^{2\alpha}\left((1+C\left(E_p(\pi,R)^\frac 1 2+\mathscr D(R,\lambda,\mu)^\frac 1 2\right)\right) \mathscr D(R,\lambda,\mu)\\
|\gamma-1|^2\lesssim E(\pi,R)+\mathscr D(R,\lambda,\mu).
\end{align*}
Moreover, $\lambda_s(0)=\mu_s(0)$ and Assumption \ref{ass:tangency}holds for the measures $\lambda_s,\mu_s$.

If $p=1$, the statement holds assuming that Assumption \ref{ass:tangency} holds in $Q_{TR}$ for some $T\geq 2$, replacing $D(R,\lambda,\mu)$ with
\begin{align*}
    R^{2\alpha}([\lambda]_{\alpha,R}+[\mu]_{\alpha,R})+R^{2\alpha}([\nabla g_0]_{\alpha,Q_{TR}}+[\nabla g_1]_{\alpha,Q_{TR}}).
\end{align*}
and assuming instead of Assumption \ref{ass:topological} that,
\begin{align*}
(TR)^{-(d+2)}\int_{\#_R\cap \{|x-y|\geq (T-1)R\}} |x-y|^2 \dd \pi\leq \frac{\e^2}{T^2 R^2}+D_{bdy}(TR).
\end{align*}
\end{corollary}
\begin{proof}
In light of Proposition \ref{one step improvement}, it suffices to verify assumptions (A)-(C) with $\Lambda = R_c^2 R^d$. For a sufficiently small choice of $R_0$, (A) is an immediate consequence of the equivalence between Assumption \ref{ass:tangency} and Assumption \ref{ass:tangencyLocal} up to a rotation and translation of the coordinate system.

For $1< p\leq 2$, note that due to Assumption \ref{ass:topological}, the topological assumption \eqref{eq:topological2} of Corollary \ref{cor:Linfty} is verified. Hence we obtain that 
\begin{align*}
\sup_{(x,y)\in (\Omega_0\cap B_R)\times\R^d\cap \supp\pi}\frac{|y-x|}{R}\lesssim \max\{E(\pi,R)^{\frac{1}{d+2}},\frac{\e^{\frac{2}{d(p-1)+2}}}{R}\}.
\end{align*}
Then the assumption \eqref{c85} shows that  (B2) is verified.

Lemma \ref{lem:localQuasiminimality} gives
\begin{align}
E_p(\pi,R)\leq W_2^2(\bar \lambda,\bar \mu)+\frac{R_c^2}{R^2}+\int_{\#_{2R}} h_p\left(\frac{\dd\pi}{\dd P}\right)\dd P.
\end{align}
Noting that (C) is only employed in Step 2 of the proof, this shows we may still use (the proof of) Proposition \ref{P:1} to obtain
\begin{align*}
\int_{\#_{R/4}} |x-y-\nabla \phi|^2\dd\pi+\e^2\int_{\#_{R/4}}h_p\left(\frac{\dd\pi}{\dd P}\right)\leq \tau E_p(\pi,R)+C\left(\mathscr D(R, \lambda, \mu)+\frac{R_c^2}{R^2}\right).
\end{align*}
In particular, (the proof of) Proposition \ref{one step improvement} applies and gives for $\tilde \beta=\frac{1+\beta} 2\in (0,1)$, $\theta \in (0, 1]$ and $C_\beta \ge 1$ depending only on $d, \alpha, \beta, r$, $B \in \mathbb{R}^{d \times d}$ and $b \in \mathbb{R}^d$ satisfying
\begin{equation}\nonumber
|B - Id|^2 + \frac{|b|^2}{R^2} \lesssim E(\pi,R) + D(R),
\end{equation}
such that with $s=(B,b,\mu(b)^\frac 1 d,1)$ and $\pi_s$ defined as in Section \ref{sec:changeOfCoordinates},
\begin{align}\label{eq:applicationOfGenericOneStep}
E(\pi_s,\theta R)+\frac{\e^2}{(\theta R)^{d+2}}\int_{\#_{R/4}} h_p\left(\frac{\dd \pi}{\dd P}\right)\dd P\leq \theta^{2\tilde\beta} E_p(\pi,R)+C_\beta \left(\mathscr D(R, \lambda, \mu)+\frac{R_c^2}{R^2}\right).
\end{align}
Note that in light of Section \ref{sec:changeOfCoordinates},
\begin{align*}
\frac{\e^2}{(\theta R)^{d+2}}\int_{\#_{R/4}}h_p\left(\frac{\dd\pi}{\dd P}\right)\dd P_s= \frac{\e^2}{(\theta R)^{d+2}\gamma} \int_{Q(\#_{R/4})} h_p\left(\frac{\dd \pi_s}{\dd P_s}\right)\dd P_s,
\end{align*}
where $Q(x,y) = (B^{-1}x,\gamma B(y-b))$. Reducing the choice of $\theta$ if necessary, we may ensure $\#_{\theta R}\subset Q(\#_R)$ and deduce
\begin{align*}
\frac{\e^2}{(\theta R)^{d+2}}\int_{\#_{R/4}}h_p\left(\frac{\dd\pi}{\dd P}\right)\dd P_s\geq \frac{\e^2}{\gamma(\theta R)^{d+2}}\int_{\#_{\theta R}} h_p\left(\frac{\dd \pi_s}{\dd P_s}\right)\dd P_s-C \frac{\e^2 R^d}{(\theta R)^{d+2}}.
\end{align*}
Thus, we have shown
\begin{align}
E_p(\pi_s,\theta R)\leq \gamma\left(\theta^{2\tilde\beta} E_p(\pi,R)+C_\beta \left(\mathscr D(R, \lambda, \mu)+\frac{R_c^2}{R^2}\right)\right).
\end{align}
In light of \eqref{eq:conditiongamma} as $\tilde \beta>\beta$, reducing $\theta$ further if necessary, we may ensure $\gamma \theta^{2\tilde \beta}\leq \theta^{2\beta},$ completing the proof.

For $p=1$, (B1) is satisfied by Lemma \ref{eq:crossingEntropic}. (C) is given by Lemma \ref{lem:localQuasiminimality}. The proof now proceeds with the obvious changes due to the presence of $T$ in Proposition \ref{one step improvement}.
\end{proof}

\section{The ROT-problem- global regularity and sparsity bound}
We now turn towards establishing our main results. We begin by proving a large-scale boundary $\e$-regularity result. We then study the small-scale geometric properties of $\mathcal S_x$ and $\mathcal T_y$. Combining these two results, we will obtain our main result Theorem \ref{thm:main}.

\subsection{Large-scale regularity}\label{sec:large-scaleRegularity}
The goal of this section is to prove a large-scale $\e$-regularity result for regularised optimal transport. We begin by proving a boundary Campanato iteration result. The proof is essentially a direct consequence of the one-step improvement Corollary \ref{cor:oneStepRegularised} and a Campanato iteration as implemented in the interior case in \cite[Section 1.3]{GvalaniKoch}.

\begin{proposition}\label{prop:largeScale}
Fix $p\in (1,2]$. Assume that for some $\alpha\in (0,1)$, $\lambda,\mu$ admit $C^{1,\alpha}$ densities and $\Omega_0=\supp\,\lambda$, $\Omega_1=\supp\,\mu$ are $C^{1,\alpha}$-domains. Suppose $0\in \partial \Omega_0,\partial \Omega_1$, $\nu_0(0)=\nu_1(0)$ and $\lambda(0)=\mu(0)=1$. Let $\pi$ be a solution of $OT_{\e,h_p}(\lambda,\mu)$. Denote $R_c = \e^\frac 2 {d(p-1)+2}$. Then there is $R_0>0$ and $\e_1>0$ such that if $R\leq R_0$, Assumption \ref{ass:topological} holds in $Q_R$ and
\begin{align*}
E_p(\pi,R)+\mathscr D(R,\lambda,\mu)+\frac{R_c^2}{R^2}<\e_1,
\end{align*}
then for any $R_c = \e^\frac 2 {d(p-1)+2} \ll r \leq R$, we may find $b_r\in \R^d$ and $A_r\in \R^{d\times d}$ symmetric such that
\begin{align*}
\frac 1 {r^{d+2}} \int_{(Q_r(0)\times \R^d)\cup (\R^d\times Q_r(b_k))} |y-A_r x-b_r|^2 d\pi\lesssim E_p(\pi,R)+\mathscr D(R,\lambda,\mu)+\frac{R_c^2}{r^2},\\
\frac{|b_r|^2}{R^2}+|A_r-\textup{id}|^2\lesssim  E_p(\pi,R)+\mathscr D(R,\lambda,\mu)+\frac{R_c^2}{r^2}.
\end{align*}

If $p=1$, the statement holds assuming that Assumption \ref{ass:tangency} holds in $Q_{TR}$ for some $T\geq 2$, replacing $\mathscr D(R,\lambda,\mu)$ with
\begin{align*}
    R^{2\alpha}([\lambda]_{\alpha,R}+[\mu]_{\alpha,R})+R^{2\alpha}([\nabla g_0]_{\alpha,Q_{TR}}+[\nabla g_1]_{\alpha,Q_{TR}}).
\end{align*}
and assuming that  instead of Assumption \ref{ass:topological}, it holds that
\begin{align*}
(TR)^{-(d+2)}\int_{\#_R\cap \{|x-y|\geq (T-1)R\}} |x-y|^2 \dd \pi\leq \frac{\e^2}{T^2 R^2}+D_{bdy}(TR)).
\end{align*}
\end{proposition}

\begin{proof}
We give the proof in the case $p\in (1,2]$. The case $p=1$ follows by straightforward adjustments arising from the presence of $T$ in Corollary \ref{cor:oneStepRegularised}.

We aim to iteratively apply Corollary \ref{cor:oneStepRegularised}. Set $r_k = \theta^k R$, where $\theta$ is as in Corollary \ref{cor:oneStepRegularised} and fix $\beta>\alpha$. As long as we can justify the use of Corollary \ref{cor:oneStepRegularised}, we obtain a sequence of measures $\lambda_k, \mu_k$, $C^{1,\alpha}$-domains $\Omega_{0,k}$, $\Omega_{1,k}$ described by $C^{1,\alpha}$-functions $g_{0,k}, g_{1,k}$, symmetric matrices $B_k\in \R^{d\times d}$, vectors $b_k\in \R^d$ and real numbers $\gamma_k$ such that
\begin{align}\label{eq:IterationUnjustified}
E_k:=&E_p(\pi_k,r_k)\leq \theta^{2\beta} E_{k-1}+C_\beta\left(\mathscr D(r_{k-1},\lambda_{k-1},\mu_{k-1})+\frac{R_c^2}{r_{k-1}^2}\right)\\
|\gamma_k-1|^2+&|B_k-\textup{id}|^2+\frac{|b_k|^2}{r_{k-1}^2}\lesssim E_{k-1}+\mathscr D(r_{k-1},\lambda_{k-1},\mu_{k-1})\\
\mathscr D(r_k,\lambda_k,\mu_k)\leq& \theta^{2\alpha}\left(1+C\left(E_{k-1}+^\frac 1 2 + \mathscr D(r_{k-1},\lambda_{k-1},\mu_{k-1})^\frac 1 1\right)\right)\mathscr D(r_{k-1},\lambda_{k-1},\mu_{k-1}).
\end{align}

Here $\pi_k$ is a minimiser of regularised optimal transport between $\lambda_k$ and $\mu_k$ with regularisation parameter $\prod_{i=0}^k \gamma_k \e^2$.

Before justifying the application of Corollary \ref{cor:oneStepRegularised}, we claim that \eqref{eq:IterationUnjustified} implies for $r_k \gg R_c$ that there is $C>0$ independent of $k$ such that
\begin{align}\label{eq:claimIteration1}
r_k^{-2\alpha}\mathscr D(r_k,\lambda_k,\mu_k) \leq (1+\theta^{k\alpha}+C r_{k-1}^{-2}R_c^2)r_{k-1}^{-2\alpha}\mathscr D(r_{k-1},\lambda_{k-1},\mu_{k-1})\\
\label{eq:claimIteration2}
\sup_{i\leq k} \theta^{-2i\alpha} E_i\leq E(\pi,R)+C\left( \mathscr D(R,\lambda,\mu)+\frac{R_c^2}{\theta^{2k\alpha} r_k^2}\right).
\end{align}

We prove \eqref{eq:claimIteration1} by induction. The kase $k=1$ is clear, so suppose \eqref{eq:claimIteration1} holds for $k=1, \ldots, K-1$. By the induction hypothesis, we find
\begin{align*}
   \mathscr D(r_{k-1},\lambda_{K-1},\mu_{K-1}) \leq \theta^{2\alpha {K-1}}\prod_{k=1}^{K-2} (1+\theta^{\alpha k}+C r_{k-1}^{-1}R_c)\mathscr D(R,\lambda,\mu).
\end{align*}
Now note for $r_K \gg R_c$, i.e. $K \ll \frac{\log(R_c/R)}{\theta}$,
\begin{align*}
\prod_{k-1}^{K-2} \left(1+\theta^{\alpha k}+\frac{\e}{C\theta^{k-1}R}\right)\leq& \prod_{k-1}^{K-2} \left(1+\theta^{\alpha k}+\frac{R_c}{C\theta^{K-1}R}\right)\\
\leq& \prod_{k-1}^{k_K} (1+\theta^{2\alpha k})\prod_{k_K+1}^{K-2} \left(1+\frac{2C R_c}{\theta^{K-1}R}\right).
\end{align*}
Here $k_K = \max\{k\colon \theta^{\alpha k}\geq C(\theta^{K-1}R)^{-1} \e\}$. The first product clearly admits a bound independent of $k_K$. Regarding the second bound, as long as $\theta^k R\geq C_0 R_c$,
\begin{align*}
\left(1+\frac{2C R_c}{\theta^{K-1}R}\right)^K\leq \left(1+\frac{2C}{C_0}\right)^\frac{\log(C_0R_c/R)}{\log \theta}.
\end{align*}
By elementary calculations, this expression is bounded independently of $K$ if $C_0$ is chosen sufficiently large. It follows that there is $C>0$, independent of $K$ such that as long as $\theta^k R \gg R_c$,
\begin{align}\label{eq:uniformDataBound}
    \mathscr D(r_{K-1},\lambda_{K-1},\mu_{K-1})\leq C<+\infty.
\end{align}
Inserting this into the last equation of \eqref{eq:IterationUnjustified} and using the induction hypothesis, we deduce \eqref{eq:claimIteration1} for $K$.

We turn to proving \eqref{eq:claimIteration2} for $K$. Using \eqref{eq:uniformDataBound}, the induction hypothesis and \eqref{eq:IterationUnjustified}, we may ensure that
\begin{align*}
    \sup_{1\leq i\leq K} \theta^{-(2i\alpha)} E_i\leq& \theta^{2(\beta-\alpha)}\left(\sup_{1\leq i\leq K-1} \theta^{-2i\alpha} E_i+E(\pi,R)\right)+C \mathscr D(R,\lambda,\mu)+\frac{C R_c^2}{ \theta^{2K\alpha} r_K^2}\\
    \leq& E(\pi,R)+C\mathscr D(R,\lambda,\mu)+\frac{C R_c^2}{\theta^{2K\alpha} r_K^2}.
\end{align*}
To obtain the last line, we used that $\beta>\alpha$ and potentially decreased the choice of $\theta$. Thus \eqref{eq:claimIteration2} is established. We note that \eqref{eq:claimIteration2} implies that
\begin{align}\label{eq:uniformBoundEnergy}
    E_k\leq C\left(E(\pi,R)+\mathscr D(R,\lambda,\mu)+\frac{C R_c^2}{r_K^2}\right).
\end{align}

This almost justifies the iterated application of Corollary \ref{cor:oneStepRegularised} in \eqref{eq:IterationUnjustified}. However, at each stage the regularisation parameter changes. Thus the proof is complete, once we prove that the regularisation parameter remains comparable to $\e^2$. Thus, it suffices to establish for $K \ll \frac{\log(R_c/R)}{\theta}$,
\begin{align}\label{eq:iterationFinalClaim}
\left|\prod_{i=1}^K \gamma_k-1\right|\ll 1.
\end{align}
Indeed, with $k_K$ as above,
\begin{align*}
\left|\prod_{i=1}^K \gamma_k-1\right|\leq \prod_{j=1}^{k_K} |\gamma_j|\left(\sum_{i=1}^{k_K} |\gamma_i-1|^2\right)^\frac 1 2
\end{align*}
Using \eqref{eq:IterationUnjustified}, \eqref{eq:uniformDataBound} and \eqref{eq:uniformBoundEnergy}, we find
\begin{align*}
\prod_{j=1}^{k_K} |\gamma_j|^2 \leq& \prod_{j=1}^{k_K} \left(1+E_{j-1}+\mathscr D(r_{j-1},\lambda_{j-1},\mu_{j-1})+\frac{R_c^2}{r_{j-1}^2}\right)\\
\leq& \prod_{j=1}^{k_K} \left(1+\theta^{2 k_K\alpha}+C \sum_{i=1}^{j-1} \frac{R_c^2}{r_i^2}+\theta^{2\alpha(j-1)}+\frac{R_c^2}{r_{j-1}^2}\right)\\
\leq& \prod_{j=1}^{k_K} \left(1+\theta^{2k\alpha}+\frac{C R_c^2}{\theta^{2(j-1)}R^2}\right)\leq C<+\infty.
\end{align*}
Similarly, using \eqref{eq:IterationUnjustified} and \eqref{eq:claimIteration2},
\begin{align*}
\sum_{i=1}^{k_K} |\gamma_k-1|^2 \lesssim& \sum_{i=1}^{k_K} E_{i-1}+\mathscr D(r_{i-1},\lambda_{i-1},\mu_{i-1})+\frac{R_c^2}{r_{i-1}^{2}}\\
\lesssim& \sum_{i=1}^{k_K} \e_1 \theta^{2(i-1)\alpha}+\frac{R_c^2}{r_{i-1}^2}\lesssim \e_1+\frac{R_c^2}{r_{k_K}^2}\lesssim \e_1+C_0^{-2}.
\end{align*}
Enlarging $C_0$ if necessary, this shows \eqref{eq:iterationFinalClaim}.

Similar estimates show that with $B_k=\prod_{i=1}^{k}B_i$ and $b_k=\sum_{i=1}^k \Gamma_{i-1}^{-1} M_{i-1}^{-1} b_i$ where $M_i = B_i\cdots B_1$ and $\Gamma_i=\gamma_i\cdots \gamma_1$, we have for $k\leq k_K$,
\begin{align*}
|A_k-\textup{id}|^2+\frac{|b_k|}{R}\lesssim E(\pi,R)+\mathscr D(R,\lambda,\mu)+\frac{R_c^2}{r_k^2}.
\end{align*}

For $k\leq k_K$, rewriting the estimate for $E_{k}$ in \eqref{eq:IterationUnjustified} in terms of $\pi$, it is an easy calculation to see that
\begin{align*}
    \frac 1 {r^{d+2}} \int_{(Q_r(0)\times \R^d)\cup (\R^d\times Q_r(b_k))} |y-A_k x-b_k|^2 d\pi \lesssim E_p(\pi,R)+\mathscr D(R,\lambda,\mu)+\frac{R_c^2}{r_k^2}.
\end{align*}
Filling in the gaps between $r_k$ and $r_{k+1}$ in a routine fashion this completes the proof.
\end{proof}
\begin{remark}\label{CampanatoDecay}
   We remark that it is possible to slightly sharpen Proposition \ref{prop:largeScale}. Retaining the decay in \eqref{eq:claimIteration2}, it is straightforward by the same arguments as in the proof of Proposition \ref{prop:largeScale} to prove
    \[\frac{E_k}{r_k^{2\alpha}}\lesssim\frac{E_p(\pi,R)+\mathscr D(R, \lambda, \mu)}{R^{2\alpha}}+\frac{R_c^2}{r_k^{2+2\alpha}}, \quad r_k = \theta^k R\]
    as long as $\frac{R_c^2}{r_k^{2+2\alpha}} R^{2\alpha}\ll 1 \Leftrightarrow \left(\frac{R_c}{R}\right)^2\ll \left(\frac{r_k}{R}\right)^{2+2\alpha}$.
For quadratic OT, $R_c=0$ and the decay holds at arbitrarily small $r_k$.
\end{remark}

\subsection{Small-scale geometric bounds}\label{sec:smallScaleResults}
We begin with a geometric observation.
\begin{lemma}\label{lem:convexHullGeometry}
Suppose $\Omega\subset \R^d$ is a $C^{1,\alpha}$-domain and $\omega$ is a convex domain. Let $\alpha\in (0,1)$. There are $C=C(\Omega),\delta=\delta(\Omega)>0$ such that if $\textup{diam}(\omega\cap\Omega)\leq \delta$, then for $x\in \textup{conv}(\omega\cap \Omega)$,
\begin{align*}
d(x,\Omega)\leq C\textup{diam}(\omega\cap\Omega)^{1+\alpha}.
\end{align*}
Moreover, if $\omega \cap \Omega$ contains a ball of radius $c \textup{diam}(\omega \cap \Omega)$ for some $c>0$, then the estimate can be improved to
\begin{align*}
d(x,\omega\cap\Omega)\leq C\textup{diam}(\omega\cap\Omega)^{1+\alpha}
\end{align*}
for some $C=C(\Omega,c)$.
\end{lemma}
\begin{proof}
Set $d= \textup{diam}(\omega\cap\Omega)$. We ensure $d<r_0$, where $r_0$ is the radius of a $C^{1,\alpha}$-boundary chart of $\Omega$. Then, we may assume $d(\omega\cap\Omega,\partial \Omega)<d$, as else $\textup{conv}(\omega\cap\Omega)\subset \Omega$ and the claim automatically holds.

After a change of coordinates, we may restrict ourselves to the situation where
\begin{align*}
\{(x',x_d)\colon x'\in B_{r_0}', r_0>x_d>g(x')\} \supset \omega\cap \Omega
\end{align*}
with $g(0)=0$, $\nabla g(0)=0$ and $[\nabla g]_{C^{0,\alpha}}\leq L$.

For $x'\in B_{r_0}'$, set $I(x') = \{t\colon (x',t)\in \omega\}$. As $\omega$ is convex, $I(x')$ is an interval and we write $I(x') = [l(x'),u(x')]$. Note that $l(x')$ is a convex function on the convex set where $I(x')$ is non-empty. In particular, we find
\begin{align*}
\Omega\cap \omega = \{(x',t)\colon x'\in B_{r_0}', t\in [\max(l(x'),g(x')), u(x')]\}.
\end{align*}

Let now $x_i = (x_i',t_i)\in \Omega \cap \omega$ and consider $\bar x = \sum \lambda_i x_i$, where $\sum \lambda_i = 1$. We may assume that $\bar x \not \in \Omega$, as else $d(\bar x,\omega\cap\Omega)=0$. Introduce $\bar x' = \sum \lambda_i x_i'$ and $\bar t = \sum \lambda_i t_i$. As $\omega$ is convex, $\bar t\in I(\bar x')$. Noting $t_i \geq \max(l(x_i'),g(x_i'))$, we estimate
\begin{align*}
\bar t \geq \sum \lambda_i \max(l(x_i'),g(x_i')).
\end{align*}
By convexity of $l$,
\begin{align}\label{eq:convexpart}
\sum \lambda_i l(x_i') \geq l(\bar x').
\end{align}
We claim that
\begin{align}\label{eq:claim}
\sum \lambda_i g(x_i')\geq g\left(\bar x'\right)- C L d^{1+\alpha}.
\end{align}
By Taylor-expansion, for some absolute constant $C>0$,
\begin{align*}
g(x_i') = g(\bar x')+\langle \nabla g(\bar x'),x_i'-\bar x'\rangle+R_i,\quad |R_i|\leq C L |x_i'-\bar x'|^{1+\alpha}.
\end{align*}
Consequently,
\begin{align*}
\sum \lambda_i g(x_i')=& g(\bar x')+\sum \langle \nabla g(\bar x'),\lambda_i(x_i'-\bar x')\rangle + \sum_i \lambda_i R_i\\
=& g(\bar x')+\sum_i \lambda_i R_i\geq g(\bar x')- C L d^{1+\alpha}.
\end{align*}
Combining \eqref{eq:convexpart} and \eqref{eq:claim}, we deduce
\begin{align}\label{eq:bargLower}
\bar t \geq \max(l(\bar x'), g(\bar x')-C L d^{1+\alpha})\geq \max(l(\bar x'), g(\bar x'))-C L d^{1+\alpha}.
\end{align}
Set $\tilde x = (\bar x', \max(l(\bar x'), g(\bar x'))$. Recall that $\bar t< g(\bar x')$ since $\bar x\notin\Omega$ and that by definition $\tilde x\in \Omega$. Therefore, 
\begin{align*}
d(\bar x,\Omega)\leq |\bar x-\tilde x| = \max(l(\bar x'), g(\bar x'))-\bar t\leq C L d^{1+\alpha}.
\end{align*}

For the moreover part, suppose $B_\rho(y_0)\subset \Omega\cap\omega$, where $\rho=c\diam(\omega\cap\Omega)$. Consider $z=y_0+\frac{\rho} 2 e_d$. Note $z=(z',t_z)\in \Omega\cap\omega$ and $t_z-g(z')\geq \frac{\rho} 2$. As $\omega$ is convex, the segment connecting $\bar x$ and $z$ is contained in $\omega$. Consider $(p',t_p)=p:=\bar x+s(z-\bar x)\in \omega$ for $s\in [0,1]$. Now by \eqref{eq:claim},
\begin{align*}
g(p')\leq s g(z')+(1-s) g(\bar x')+C L d^{1+\alpha}.
\end{align*}
We deduce
\begin{align*}
t_p-g(p')\geq& s (t_z-g(z))+(1-s)(\bar t-g(\bar x'))-C L d^{1+\alpha}\\
\stackrel{\eqref{eq:bargLower}}\geq& \frac{s \rho} 4-2C L d^{1+\alpha}
\end{align*}
Choose $s=K d^\alpha$ for a sufficiently large constant $K>0$ to ensure $t_p-g(p')> 0$ and hence $p\in \Omega \cap \omega$. Moreover, since $|y_0-\bar x|\leq 2 d$,
\begin{align*}
d(\bar x, \omega\cap \Omega)\leq |p-\bar x|\leq s (|y_0-\bar x|+\frac{\rho} 2)\lesssim  d^{1+\alpha}.
\end{align*}
This concludes the proof.
\end{proof}

We now proceed to establish small-scale geometric properties of $\mathcal S_x$ and $\mathcal T_y$.
\begin{proposition}\label{lem:smallScaleGeometric}
Let $p\in (1,2]$. Assume $\supp\,\lambda,\,\supp\,\mu$ are $C^{1,\alpha}$-domains. Suppose $\pi$ solves $OT_{\e,h_p}(\lambda,\mu)$. There are $C_0,\e_1>0$ such that if $E(\pi,C_0 R_c)+\e+C_0^{-1}+\e\leq \e_1$, then there is $C_1>0$, independent of $\e$, such that for $x\in B_{C_0 R_c/6}, \,y\in B_{C_0 R_c/6}$,
\begin{align*}
 &B\left(\frac 1 2\nabla \phi_\e(x),C_1^{-1}R_c\right)\cap \supp\,\mu \subset \mathcal S_x\subset B\left(\frac 1 2\nabla \phi_\e(x),C_1 R_c\right),\\
 &B\left(\frac 1 2\nabla \psi_\e(y),C_1^{-1}R_c\right)\cap \supp\,\lambda \subset \mathcal T_y\subset B\left(\frac 1 2\nabla \psi_\e(y),C_1 R_c\right).
\end{align*}
Moreover, $\left|B\left(\frac 1 2\nabla \phi_\e(x),C_1^{-1}R_c\right)\cap \supp\,\mu\right|\geq C_1^{-1} R_c^d$.
\end{proposition}

\begin{proof}
By symmetry, it suffices to consider $x\in B_{C_0 R_c/3}$. The second inclusion is an immediate consequence of Section \ref{sec:Linfty}. Moreover, by Proposition \ref{prop:smallScale}, $\phi_\e,\psi_\e$ are $W^{2,\infty}$-functions on $B_{C_0 R_c/3}$. Increasing $C_0$ if necessary, we deduce that for some $C>0$, independent of $\e$ and $C_0$,
\begin{align}\label{eq:Lipschitz}
\|\nabla\phi_\e\|_{W^{1,\infty}\left(\cup_{y\in B_{C_0 R_c/6}}\mathcal T_y\right)}+\|\nabla\psi_\e\|_{W^{1,\infty}\left(\cup_{x\in B_{C_0 R_c/6}}\mathcal S_x\right)}\leq C.
\end{align}

Note that by Jensen's inequality, with $\frac 1 p+\frac 1 q=1$,
\begin{align*}
\xi(x,\frac{1}{2} \nabla \phi_\e(x)) \geq \left(\int_{\mathcal S_x} \xi(x,y)^{q-2}\dd\mu(y)\right)^{-1}\int_{\mathcal S_x} \xi(x,y)^{q-1}\dd \mu(y).
\end{align*}

We estimate,
\begin{align*}
\int_{\mathcal S_x} \xi(x,y)^{q-2}\dd \mu(y)\lesssim \max_{y\in \mathcal S_x} \xi(x,y)^{q-2} |\mathcal S_x|.
\end{align*}
Consequently, using the optimality condition and that we already established $|\mathcal S_x|\lesssim R_c^d$,
\begin{align*}
\xi\left(x,\frac{1}{2}\nabla \phi_\e(x)\right)\gtrsim& \frac{\e^{2(q-1)}}{R_c^d\max_{y\in \mathcal S_x} \xi(x,y)^{q-2}}.
\end{align*}

Finally, we use the concavity of $\xi$ to estimate for $y\in \mathcal S_x$,
\begin{align*}
   \xi(x,y) \geq  \xi(x,y_m)-\langle \nabla_y \xi(x,y),y_m-y\rangle.
\end{align*}
where $y_m$ is chosen so that $\xi(x,y_m) = \xi_m:=\max_{y\in \mathcal S_x} \xi(x,y)$. Now, for any $z\in \mathcal T_y$,
\begin{align*}
|\nabla_y \xi(x,y)|=2\left|x-\frac 1 2\nabla \psi_\e(y)\right|\leq  2|z-x|+2\left|z-\frac 1 2\nabla \psi_\e(y)\right|\leq 2\textup{diam}(\mathcal T_y)+2\left|z-\frac 1 2\nabla \psi_\e(y)\right|.
\end{align*}

As $\frac 1 2 \nabla \psi_\e(y)\in \textup{conv}(\mathcal T_y)$, $d\left(\frac 1 2 \nabla \psi_\e(y),\mathcal T_y\right)\leq \textup{diam}(\mathcal T_y)\lesssim R_c$. Taking infimums over $z$, we thus deduce $|\nabla_y \xi(x,y)|\lesssim R_c$ and hence for some $C>0$,
\begin{align*}
    \xi(x,y)\geq \xi(x,y_m)-C R_c |y_m-y|.
\end{align*}
Inserting this estimate into the Schr\"odinger equation and using the interior cone property, we deduce for $\lambda$-a.~e.~$x$,
\begin{align*}
q^{q-1}\e^{2(q-1)}=& \int_{\mathcal S_x} \xi(x,y)^{q-1}\dd\mu(y)\geq \int (\xi(x,y_m)-C R_c |y_m-y|)_+^{q-1}\dd\mu(y)\\
\gtrsim& \int (\xi(x,y_m)-C R_c t)_+^{q-1}t^{d-1}\dd t\gtrsim \frac{\xi(x,y_m)^{q+d-1}}{R_c^d}
\end{align*}
Combining estimates, we have shown
\begin{align*}
    \xi\left(x,\frac 1 2\nabla \phi_\e(x)\right)\geq C\frac{\e^{2(q-1)}}{R_c^d \e^\frac{2(q-1)(q-2)}{q+d-1}R_c^\frac{d(q-2)}{q+d-1}}=C R_c^2.
\end{align*}
By concavity of $\xi$, this shows for any $y$,
\begin{align*}
    \xi(x,y)\geq& \xi\left(x,\frac 1 2\nabla \phi_\e(x)\right)-\langle \nabla_y \xi(x,y),\frac 1 2\nabla \phi_\e(x)-y\rangle\\
    \geq& C R_c^2 - 4\left|x-\frac 1 2\nabla \psi_\e(y)\right|\left|y-\frac 1 2\nabla \phi_\e(x)\right|.
\end{align*}

By Lemma \ref{lem:convexHullGeometry}, applied to $\omega = \{y\colon \xi(x,y)\geq 0\}$ and $\Omega = \supp\,\mu$ we may find $y_0\in \supp\,\mu$ such that $\left|y_0-\frac 1 2 \nabla \phi_\e(x)\right|\leq c R_c^{1+\alpha}$. Further, for any $y\in \mathcal S_x$, using Proposition \ref{prop:smallScale},
\begin{align*}
\left|x-\frac 1 2 \nabla \psi_\e(y_0)\right|\leq& \left| x-\frac 1 2 \nabla \psi_\e(y)\right|+C |y-y_0|\\
\leq&\left| x-\frac 1 2 \nabla \psi_\e(y)\right|+C \left|y-\nabla \phi_\e(x)\right|+\left|\frac 1 2\nabla \phi_\e(x)-y_0\right|
\end{align*}
We have already shown that  $\left| x-\frac 1 2 \nabla \psi_\e(y)\right|\lesssim R_c$. A symmetric argument gives $\left| y-\frac 1 2 \nabla \phi_\e(x)\right|\lesssim R_c$. We deduce that for some $c>0$,
\begin{align*}
\xi(x,y_0)\geq C R_c^2 - c R_c^{2+\alpha}.
\end{align*}
In particular, $\xi(x,y_m)\geq \xi(x,y_0)\geq \frac{C} 2 R_c^2$ for a sufficiently small choice of $\e_1$. Since for $y\in\mathcal S_x$, due to Section \ref{sec:Linfty} and Lemma \ref{lem:convexHullGeometry},

\begin{align*}
    |\nabla_y \xi(x,y)| = 2|x-\frac 1 2\nabla \psi_\e(y)|\leq \inf_{z\in \mathcal T_y} 2|\frac 12 \nabla \psi_\e(y)-z|+2|x-z|\leq c(\e_1) R_c+R_c^{1+\alpha},
\end{align*}
we conclude that there is $c>0$ such that $B(y_m,2c R_c)\cap \supp\,\mu\subset \mathcal S_x$. By the interior cone property, it holds that $|B(y_m,2c R_c)\cap \supp\,\mu|\gtrsim R_c^d$. Reducing $\e_1$ if necessary, we may further ensure $B(y_m,c/2 R_c)\cap \supp\,\mu\subset B\left(\frac 1 2\nabla \phi_\e(x),c R_c\right)\cap \supp\,\mu\subset B(y_m,2c R_c)\cap \supp\,\mu$ by applying the moreover part of Lemma \ref{lem:convexHullGeometry}. Consequently,
\begin{align*}
    B\left(\frac 1 2\nabla \phi_\e(x),c R_c\right)\cap \supp\,\mu\subset \mathcal S_x,
\end{align*}
and
\begin{align*}
    \left|B\left(\frac 1 2\nabla \phi_\e(x),c R_c\right)\cap \supp\,\mu\right|\geq |B(y_m,c/2 R_c)\cap \supp\,\mu|\gtrsim R_c^d.
\end{align*}
This completes the proof.
\end{proof}

\subsection{Global regularity and geometric bounds}\label{sec:mainResultProof}
We now prove Theorem \ref{thm:main}.

\begin{proof}
We consider first the case $p\in (1,2]$ and comment at the end on the necessary changes for $p=1$. By symmetry it suffices to prove the statements involving $\phi_\e$.

For any $\e_2>0$, there exists $R_0>0$ such that for $R\leq R_0$, for any cube $Q_R(x)$ centered at some $x$,
\begin{align*}
R^{2\alpha}([\lambda]_{\alpha,Q_R(x)}^2+[\mu]_{\alpha,Q_R(x)}^2)<\e_2,
\end{align*}
as well as
\begin{align*}
R^{2\alpha}([\nabla g_0]^2_{\alpha,Q_R'(x)}+[\nabla g_1]^2_{\alpha,Q_R'(x)})<\e_2,
\end{align*}
whenever $g_0$ and $g_1$ describe (up to a rigid motion) $\partial \supp\,\lambda$ and $\partial \supp\,\mu$, respectively, in $Q_R(x)$. If $Q_R(x)\cap \partial\supp\lambda =\emptyset$ or $Q_R(x)\cap \partial \supp\,\mu=\emptyset$, we set the corresponding term to be $0$. 

Fix $x_0\in \supp\,\lambda$. If $Q_{R_0}(x_0)\subset \supp\,\lambda$, the result follows from the interior results \cite[Theorem 3]{GvalaniKoch1} and \cite[Theorem 3.1]{GonzalezGvalaniKoch}. Thus, we assume $Q_{R_0}(x_0)\cap \partial \supp\,\lambda\neq\emptyset$ and find $x_1\in \partial \supp\,\lambda$ such that $|x_0-x_1|=d(x_0,\partial \supp\,\lambda)=:R$. By translation, we may assume that $x_1=0$. Consider the affine change of coordinates, as described in Section \ref{sec:changeOfCoordinates} associated to $s=(A,b,\gamma,\kappa)=\left(\nabla T(0)^\frac 1 2,  T(0),\left(\frac{\lambda(0)}{\mu(T(0))}\right)^\frac 1 d,\frac 1 {\lambda(0)}\right)$. Then by Section \ref{SS:bdy-normalization}, $\Omega_0=\supp\,\lambda_s$ and $\Omega_1=\supp\,\mu_s$ satisfy Assumption \ref{ass:tangency}. Moreover, $\lambda_s(0)=\mu_s(0)=1$. 

Due to the $C^{1,\alpha}$-regularity of $T$, it holds $|A|+|A^{-1}|+|b|\lesssim \|T^{-1}\|_{C^{1,\alpha}(\supp\,\mu)}+\|T\|_{C^{1,\alpha}(\supp\,\lambda)}$. Note that $A$ is symmetric. In particular, this shows that choosing $\tau>0$ sufficiently small, we can ensure $\#_{\tau R_0}\subset Q(\#_{R_0})$, where $Q$ is the affine transformation associated to $s$.

Further, since $T$ is a $C^{1,\alpha}$ map,
\begin{align*}
\int_{Q_{R_0}} |T(x)-A^2x-b|^2 \dd \lambda(x)\lesssim R^{d+2+2\alpha}.
\end{align*}
Consequently, applying also Lemma \ref{prop:EnergyConvergence} and Lemma \ref{lem:EnergyConvergence},
\begin{align*}
\int_{\#_{R_0}} |y-A^2x-b|^2 \dd\pi_\e\lesssim& \int_{\#_{R_0}} |y-T(x)|^2+\int_{\#_{R_0}} |T(x)-A^2x-b|^2 \dd\pi_\e(x)\\
\lesssim& R_c^2+ R_0^{d+2+2\alpha}.
\end{align*}
Moreover, note that due to Proposition \ref{prop:EnergyConvergence} ,
\begin{align*}
\e^2\int_{\#_{R_0}} h_p\left(\frac{\dd \pi_\e}{\dd P}\right)\lesssim R_c^2.
\end{align*}
Consequently,
\begin{align*}
E_p(\pi_{s},\tau R_0)\lesssim \frac{R_c^2}{R_0^{d+2}} + R_0^{2\alpha}.
\end{align*}
Further, by \cite[Corollary 5.4.]{Wiesel}, there exists $\delta(\e)\to 0$ as $\e\to 0$ such that
\begin{align*}
\sup_{(x,y)\in \supp\pi_\e}|y-T(x)|\leq \delta(\e).
\end{align*}
Reducing first $R_0$ and ensuring $\e$ is sufficiently small, we may ensure, independently of $p$,
\begin{align}\label{eq:smallnessEnergy}
E_p(\tau R_0,\pi_{s})+\mathscr D(\tau R,\lambda_{s},\mu_{s})\leq \e_2,
\end{align}
as well as
\begin{align}\label{eq:topological}
\sup_{(x_s,y_s)\in  \supp\pi_{s}\cap \#_{\tau R_0}} \frac{|x_s-y_s|}{\tau R_0}\leq 1.
\end{align}
We recall that $\pi_s$ solves $OT_{\kappa^\frac{p-1} 2 \gamma^{\frac 1 2} \e,h_p}(\lambda_s,\mu_s)$. Note that $\kappa^\frac{p-1} 2 \gamma^{\frac 1 2}\e\simeq \e$.

Set $R_1 = \tau R_0$. 

\step 1 Treatment of the range $\max(C R_c,2R)\leq r\leq R_1$ by boundary regularity. 

In light of \eqref{eq:smallnessEnergy}, \eqref{eq:topological}, the fact that $\lambda_s(0)=\mu_s(0)=1$ and Assumption \ref{ass:tangency} holds in $Q_{\tau R}$, if $\e_2$ is sufficiently small, we are in a position to apply Corollary \ref{cor:oneStepRegularised}. We then find $C,C_2>0$, such that for $C R_c\leq r\leq R_1$, there exists $b_r\in \R^d$ and $A_r\in \R^{d\times d}$ symmetric with
\begin{align*}
\frac{|b_r|^2}{R_1^2}+|A_r-\textup{id}|^2\leq C_2 \e_2+\frac{C_2} C\\
\frac 1 {r^{d+2}}\int_{(Q_r(0)\times \R^d)\cup (\R^d \times Q_r(b_r))}|y-A_r x-b_r|^2 \dd \pi_{s}\leq C_2 \e_2+\frac{C_2 R_c^2}{r^2}\leq C_2 \e_2+\frac{C_2}{C}.
\end{align*}

\step 2 Treatment of the range $C R_c \leq r \leq 2R$ by interior regularity. In light of Step 1, we find $b\in \R^d$ and $A\in \R^{d\times d}$ symmetric such that
\begin{align*}
\frac 1 {(2R)^{d+2}}\int_{(Q_{2R}(0)\times \R^d)\cup (\R^d\times Q_{2R}(b))}|y-A x-b|^2 \dd \pi_{s}\leq C_2 \e_2+\frac{C_2} C\\
\frac{|b|^2}{R_1^2}+|A-\textup{id}|^2\leq C_2 \e_2+\frac{C_2} C.
\end{align*}
Note that with $z=Q_1(x_0)$, reducing $\tau$ if necessary, $(Q_{\tau R}(z)\times \R^d)\cup (\R^d\times Q_{\tau R}(z+b))\subset (Q_{2R}(0)\times \R^d)\cup (\R^d\times Q_{2R}(b))$ and
\begin{align*}
\frac 1 {(\tau R)^{d+2}}\int_{(Q_{\tau R}(z)\times \R^d)\cup (\R^d\times Q_{\tau R}(z+b))}|y-Ax-b|^2 \dd \pi_{s}\leq \tau^{-(d+2)} C_2 \e_2+\frac{C_2 }{ C\tau^{d+2}}.
\end{align*}

In order to apply the interior large-scale-regularity result \cite[Theorem 15]{GvalaniKoch1}, we now apply a further affine change of coordinates $s'$ to $\pi_s$, $\lambda_s$ and $\mu_s$, obtaining $\pi_{s'}$, $\lambda_{s'}$ and $\mu_{s'}$. Set
$s'= \left(A,b,\frac{\lambda_s(z)^\frac 1 d}{\mu_s(z)^\frac 1 d},\lambda_s(z)^{-1}\right)$. Then $\lambda_s(z)=\mu_s(z')=1$. Decreasing $\e_2$ and increasing $C$ if necessary, we can further ensure $\#_{\tau R/2}\subset Q'((Q_{\tau R}(z)\times \R^d)\cup (\R^d\times Q_{\tau R}(z+b)))$, where $Q'$ is the affine transformation associated to $s'$. Hence
\begin{align}\label{eq:interiorFinal}
\int_{\#_{\tau R}}|x-y|^2 \dd \pi_{s'}\lesssim \e_2+\frac 1 C.
\end{align}
Moreover, the $L^\infty$-bound Corollary \ref{cor:Linfty} applied to $\pi_s$ at scale $R$ gives
\begin{align*}
\sup_{(x_s,y_s)\in  \supp\pi_{s}\cap \#_{R/2}} \frac{|x_s-y_s|}{R}\lesssim_\tau \delta\left(\e_2,\frac 1 C\right),
\end{align*}
where $\delta\left(\e_2,\frac 1 C\right)\to 0$ as $\e_2,\frac 1 C \to 0$. Consequently, decreasing $\e_2$ and increasing $C$ further, if necessary, we are now able to apply \cite[Theorem 15]{GvalaniKoch1} to find $b'\in \R^d$ and $A'\in \R^{d\times d}$ symmetric such that
\begin{align*}
\frac 1 {(CR_c)^{d+2}}\int_{(Q_{C R_c}(0)\times \R^d)\cup (\R^d\times Q_{C R_c}(b'))}|y-A' x-b'|^2 \dd \pi_{s'}\lesssim \e_2+\frac{1} C\\
\frac{|b'|^2}{R_1^2}+|A'-\textup{id}|^2\lesssim \e_2+\frac{1} C.
\end{align*}

\step 3 Conclusion. Combining Step 1 and Step 2 with a suitably large choice of $C$, we find $R_{fin}=c R_c$ for some $c>0$ such that for some $b_{fin}\in \R^d$ and $A_{fin}\in \R^{d\times d}$ symmetric, and $\tilde s=s$ or $\tilde s=s'$,
\begin{align*}
\frac 1 {R_{fin}^{d+2}}\int_{(Q_{R_{fin}}(0)\times \R^d)\cup (\R^d\times Q_{R_{fin}}(b_{fin}))}|y-A_{fin}x-b_{fin}|^2 \dd \pi_{\tilde s}\lesssim \e_2\\
\frac{|b_{fin}|^2}{(CR_c)^2}+|A_{fin}-\textup{id}|^2\lesssim \e_2.
\end{align*}
Reducing $\e_2$ if necessary, we find,  arguing analogously as for \eqref{eq:interiorFinal},
\begin{align*}
\frac 1{ (\tau R_{fin})^{d+2}} \int_{\#_{\tau R_{fin}}} |y-x|^2 \dd \pi_{\hat s}\lesssim \e_2,    
\end{align*}
for some change of coordinates $\hat s=(\hat A,\hat b,\hat \gamma,\hat \kappa)$ with
\begin{align*}
    |\hat A|+|\hat A^{-1}|+|\hat \gamma|+|\hat \kappa|\lesssim 1.
\end{align*}
Let $\phi_{\e,\hat s}, \psi_{\e,\hat s}$ be the potentials for $\pi_{\hat s}$. 
Proposition \ref{prop:smallScale} shows that 
\begin{align*}
    \|\nabla^2 \phi_{\e,\hat s}\|_{L^\infty(B_{C R_c/3})}+\|\nabla^2 \psi_{\e,\hat s}\|_{L^\infty(B_{C R_c/3})}\lesssim 1.
\end{align*}

By Proposition \ref{lem:smallScaleGeometric}, there is $\tilde C>0$ such that for $x\in B_{C R_c/6}$, $y\in B_{C R_c/6}$,
\begin{align*}
B\left(\frac 1 2\nabla \phi_{\e,\hat s}(x), \tilde C^{-1} R_c\right)\cap \supp\,\mu_{\hat s}\subset \mathcal S_x\subset B\left(\frac 1 2\nabla \phi_{\e,\hat s}(x),\tilde C R_c\right).
\end{align*}

Returning to original variables and recalling that $p$ was arbitrary, this completes the proof.

\step 4. The case $p=1$. The case $p=1$ proceeds as the case $p>1$ with the exception of obtaining \eqref{eq:topological} at the initial scale $R_1$. In order to apply Proposition \ref{prop:largeScale} at scale $R_1$, it suffices to ensure
\begin{align}\label{eq:largeEntropic}
R_1^{-(d+2)}\int_{\#_{R_1}\cap \{|x-y|\geq R_1\}} |x-y|^2\dd \pi_s  \leq \e_1.
\end{align}

Denote by $T_s$ the optimal transport map between $\lambda_s$ and $\mu_s$. 
By Proposition \ref{prop:EnergyConvergence} and \eqref{lem:EnergyConvergence},
\begin{align*}
\int |x-T_s^{-1}(y)|^2+|y-T_s(x)|^2 \,d\pi_s\lesssim \e^2.
\end{align*}
For $(x,y)\in \#_{R_1}$ with $|x-y|\geq R_1$, by triangle inequality, we have if $x\in Q_{R_1}$,
\begin{align*}
|x-y|\leq |x-T_s(x)|+|y-T_s(x)|\leq c R_1^{1+\alpha}+|y-T_s(x)|\leq 2 |y-T_s(x)|.
\end{align*}
To obtain the last inequality, we used that $T_s(0)=0$ and $\nabla T_s(0)=\textup{id}$ in addition to the $C^{1,\alpha}$-regularity of $T_s$. If $y\in Q_{R_1}$, a similar estimate gives
\begin{align*}
|x-y|\leq |y-T_s^{-1}(y)|+|x-T_s^{-1}(y)|\leq  2 |x-T_s^{-1}(y)|.
\end{align*}
Consequently,
\begin{align*}
\int_{\#_{R_1}\cap \{|x-y|\geq R_1\}} |x-y|^2\dd \pi_s  \lesssim \int |x-T_s^{-1}(y)|^2+|y-T_s(x)|^2\dd \pi_s\lesssim \e^2.
\end{align*}
Reducing $\e$ if necessary, this gives \eqref{eq:largeEntropic}.
\end{proof}

We finally note that if $\pi$ is a minimiser of optimal transport, our theory applies with $\e=0$. In other words, we can iterate down to arbitrarily small scales.
\begin{proof}[Proof of Corollary \ref{cor:OT}]
Note that by the theory of Campanato spaces, Proposition \ref{prop:largeScale} and remark \eqref{CampanatoDecay} in the case $\e=0$ gives a $C^{1,\alpha}$-regularity statement. Hence Corollary \ref{cor:OT} is a straightforward consequence of applying the arguments of \cite[Section 6.3]{MiuraOtto} with Proposition \ref{prop:largeScale} replacing \cite[Proposition 2.5]{MiuraOtto}. We leave the details to the reader.
\end{proof}

\section{Acknowledgments:} L.K. was supported by the Engineering and Physical Sciences Research Council [APP65579: Regularity Properties For Penalised Optimal Transport ].

\bibliographystyle{amsalpha}

\end{document}